\documentclass[a4paper]{amsart}
\usepackage{a4wide}
\usepackage{amsmath}
\usepackage{hyperref} 
\usepackage{amsfonts}
\usepackage{amssymb}
\usepackage{amsthm}
\usepackage[arrow, matrix, curve]{xy} 
\usepackage{hyperref}
\usepackage{svn-multi}
\usepackage{amssymb,amscd,url,tikz,stmaryrd,mathtools,fixltx2e}

\usepackage{paralist,csquotes,stackrel}
 
\usepackage[mathscr]{eucal}

\newtheorem{proposition}{\textbf{Proposition}}
\newtheorem{lemma}[proposition]{\textbf{Lemma}}
\newtheorem{corollary}[proposition]{\textbf{Corollary}}
\newtheorem{theorem}[proposition]{\textbf{Theorem}}

\theoremstyle{definition}
\newtheorem{definition}[proposition]{\textbf{Definition}}

\newtheorem{example}[proposition]{\textbf{Example}}
\newtheorem{remark}[proposition]{\textbf{Remark}}
\newtheorem{acknowledgements}{\textbf{Acknowledgements}}

\newcommand{\hkq}{{\sslash\mkern-6mu/}}

\newcommand{\symp}{\sslash}

\newcommand{\Lie}[1]{\operatorname{\textsl{#1}}}

\newcommand{\lie}[1]{\operatorname{\mathfrak{#1}}}

\newcommand{\GL}{\Lie{GL}}

\newcommand{\sln}{\lie{sl}}

\newcommand{\su}{\lie{su}}

\newcommand{\cf}{\lie c}

\newcommand{\gf}{\lie g}

\newcommand{\kf}{\lie k}

\newcommand{\lf}{\lie l}

\newcommand{\tf}{\lie t}

\newcommand{\SL}{\Lie{SL}}

\newcommand{\SU}{\Lie{SU}}

\newcommand{\Un}{\Lie{U}}

\newcommand\C{{\mathbb C}}

\newcommand{\HH}{{\mathbb H}}

\newcommand{\R}{{\mathbb R}}

\numberwithin{proposition}{section}

\title{Hyperk\"ahler Implosion and Nahm's Equations}
\author{Andrew Dancer}
\address[Dancer]{Jesus College\\
Oxford\\
OX1 3DW\\
United Kingdom} \email{dancer@maths.ox.ac.uk}

\author{Frances Kirwan}
\address[Kirwan]{Balliol College\\
Oxford\\
OX1 3BJ\\
United Kingdom} \email{kirwan@maths.ox.ac.uk}

\author{Markus R\"oser}
\address[R\"oser]{Institut f\"ur Differentialgeometrie\\
Welfengarten 1\\
30167 Hannover\\
Germany} \email{markus.roeser@math.uni-hannover.de}

\date{\today}

\makeindex

\begin{document}

\maketitle
\begin{abstract}
We use Nahm data to describe candidates for the universal hyperk\"ahler
implosion with respect to a compact Lie group.
\end{abstract}

\tableofcontents

\section{Introduction}

Symplectic implosion is an abelianisation construction in symplectic
geometry introduced by Guillemin, Jeffrey and Sjamaar \cite{GJS:2002}.
Given a symplectic manifold with a Hamiltonian action of a compact 
Lie group $K$,
the implosion (or imploded cross-section) of the manifold is a stratified symplectic space with
an action of the maximal torus $T$ of $K$, and the symplectic reduction
of the implosion by $T$ can be identified with the reduction of the
original manifold by $K$.

The key step is to define the implosion $(T^*K)_{\rm impl}$ 
of the cotangent bundle $T^*K$
with respect to one of the commuting $K$ actions on this space.
This example is universal in the sense that the implosion of
a general manifold $M$ with $K$-action may be obtained by symplectic
reduction from the product $M \times (T^*K)_{\rm impl}$. 
Symplectic reductions of the universal implosion by the maximal torus give the
coadjoint orbits of $K$.  The implosion $(T^*K)_{\rm impl}$ may be
defined by taking the product of $K$ with the closed positive Weyl
chamber, stratifying by the centraliser $C$ of points in the closed
Weyl chamber, and then collapsing by the commutator subgroup of $C$. In
particular, no collapsing occurs on the open dense set which is the
product of $K$ with the interior of the Weyl chamber.  As explained in
\cite{GJS:2002}, the implosion may also be viewed as an affine variety
with an algebro-geometric stratification that agrees with the
symplectic stratification above.

In \cite{DKS} a hyperk\"ahler analogue of the symplectic implosion
was introduced for the case $K = SU(n)$. The space is an
affine variety and hyperk\"ahler reductions by the maximal torus 
give the Kostant varieties, which are affine completions
of complex coadjoint orbits. 

The construction of \cite{DKS} works by using quiver diagrams, and
is related to the description of complex coadjoint orbits for $SU(n)$
via quivers. For other groups, no such finite-dimensional
description of the coadjoint orbits is as yet available, except 
in the case of the nilpotent variety \cite{KS}. Instead, 
complex coadjoint orbits have been realised via gauge-theoretic methods
as moduli spaces of solutions to the Nahm equations on the half-line
\cite{Kronheimer:1990}, \cite{Kronheimer:nilpotent},
\cite{Biquard:1996}, \cite{Kovalev:1996}.

The sequel \cite{DKS-Sesh} to \cite{DKS} included a possible
approach to hyperk\"ahler implosion for general compact groups 
using the Nahm equations. In this paper we shall develop an analytical
framework for this approach. We introduce in $\S 3$ a space
of Nahm data modulo gauge transformations 
and stratify it by the centraliser $C$ of the asymptotic
triple $\tau$. We obtain a candidate for the universal hyperk\"ahler implosion $\mathcal Q$ by performing further collapsings
by gauge transformations valued in $[C,C]$ at infinity.
This can be viewed as an analogue of the collapsing procedure
mentioned above in the construction of the symplectic implosion. 
It turns out, however, that this construction does not reproduce the quiver moduli space 
in the case $K=\SU(n)$, as the torus action on the bottom stratum corresponding to $C=K$ 
is trivial in the case $K$ is semi-simple ($K= [K,K]$). A modification of this first construction is described later (see Remark \ref{comparison}) which, at least for $K=\SU(2)$ (see $\S$8), can be identified with the quiver construction from  \cite{DKS}.

We show that fixing the asymptotic triple
$\tau$ and further quotienting out by a torus yields the corresponding Kostant variety; this holds both for our first  construction and its modification. 
In order to prove this, we introduce in $\S 4$ related spaces of Nahm data
which admit hyperk\"ahler structures with $T$ action, such that the hyperk\"ahler
moment map for $T$ gives the asymptotic triple $\tau$. The hyperk\"ahler quotient,
in a suitable stratified sense, is identified both with the space obtained by 
fixing $\tau$ in $\mathcal Q$ and quotienting by $T$, and also with the Kostant variety.

The tangent vectors to the spaces of solutions to the Nahm equations that we use are not necessarily square integrable. 
Therefore, instead of using the $L^2$-metric to define our hyperk\"ahler structure, we use a modification of it modelled on work 
of Bielawski \cite{Bielawski:1998} on monopole moduli spaces. In our situation the Bielawski metric is nondegenerate but may 
not be positive definite, although the induced metric on any Kostant variety is always a positive definite hyperk\"ahler metric.

In $\S 6$ we discuss how $\mathcal Q$ could play the role of universal hyperk\'ahler implosion. Let us define the Nahm  hyperk\"ahler implosion of a hyperk\"ahler manifold $M$ with a tri-hamiltonian $K$-action to be the hyperkahler quotient of $M \times \mathcal Q$ by the diagonal action of $K$. Then $\mathcal Q$ may be viewed as the Nahm
hyperk\"ahler implosion of $T^* K_{\mathbb C}$, where $K_\C$ is the complexification of $K$, and is thus a
hyperk\"ahler analogue of the universal symplectic implosion
$(T^*K)_{\rm impl}$. In $\S 7$ we produce a holomorphic description of
our strata in $\mathcal Q$ by making a choice of complex structure such that the asymptotic triple $\tau$ satisfies the
condition of Biquard regularity. In $\S 8$ we illustrate the construction by discussing the $\SU(2)$-case and point out some of the differences to the quiver construction. Finally $\S 9$ gives a brief
discussion of how symplectic implosion may be interpreted in terms
of Nahm data.

{\bf Notation}. We shall use $M \hkq K$ 
\index{$M\hkq K$} 
to denote the hyperk\"ahler
quotient of a space $M$ by a group $K$.

\begin{acknowledgements}
It is a pleasure to thank Roger Bielawski for useful
conversations.
\end{acknowledgements}

\section{Nahm's equations on the half-line}

Let us recall some facts concerning the Nahm equations.
These are the ordinary differential equation system
\begin{equation} \label{Nahm}
  \frac{dT_i}{dt} + [T_0 , T_i] = [T_j, T_k], \qquad \text{\( (ijk) \)
  cyclic permutation of \( (123) \),}
\end{equation}
where $T_i$ takes
values in the Lie algebra \( \mathfrak k \) of $K$. They can be studied
on different intervals $\mathcal I$ but we shall be primarily interested in the equations
on the half-line $[0,\infty)$.
There is a gauge group action
\begin{equation} \label{gauge}
  T_0 \mapsto u T_0 u^{-1} - \dot{u} u^{-1},\qquad 
  T_i \mapsto u T_i u^{-1} \ (i=1,2,3),
\end{equation}
where \( u \colon \mathcal{I} \mapsto K \).
We shall often denote the image of $T_i$ under $u$ by $u.T_i$.

\medskip
The constructions of complex coadjoint orbits mentioned above by
Kronheimer, Biquard and Kovalev
\cite{Kronheimer:1990}, \cite{Kronheimer:nilpotent},\cite{Biquard:1996}, \cite{Kovalev:1996}
proceed by considering suitable moduli spaces of Nahm data on $\mathcal I = [0,\infty)$ 
with prescribed asymptotics. It can be shown that Nahm matrices 
can always be put into a gauge in which they have asymptotics
\begin{equation*}
  T_i = \tau_i + \frac{\sigma_i}{2(t+1)} + \dotsb \qquad (i=1,2,3),
\end{equation*}
where $\tau=(\tau_1, \tau_2, \tau_3)$ is a commuting triple.
Moreover $\sigma_i = \rho(e_i)$, where
$e_1,e_2,e_3$ is a standard basis for $\su(2)$ and $\rho \colon \su(2)
\rightarrow \kf$ is a Lie algebra homomorphism, so $[\sigma_1, \sigma_2] =
-2 \sigma_3$ etc.  In addition we must have $[\tau_i, \sigma_j]=0$ 
for all $i,j$, so $\rho$ takes values in the Lie algebra
$\cf$ of the common centraliser $C:=C(\tau_1,\tau_2,\tau_3)$ of the triple 
$(\tau_1,\tau_2,\tau_3)$. We call $(\sigma_1, \sigma_2, \sigma_3)$
an {\em $\su(2)$-triple} in this situation.

The leading terms in fact define an exact solution to Nahm's equations 
\index{$T_{\tau,\sigma}$}
\[
T_{\tau,\sigma} =  \left(0, \left(\tau_i+\frac{\sigma_i}{2(t+1)}\right)_{i=1}^3\right),
\]
which we  call a {\em model solution}. 

In the work of
Kronheimer, Biquard and Kovalev a complex coadjoint orbit
is obtained by fixing the model solution, that is fixing the commuting triple
$\tau$, an $\su(2)$-triple $\sigma$ and the corresponding model solution $T_{\tau,\sigma}$, and considering a suitable
moduli space of solutions to the Nahm equations asymptotic to this model solution. If
$\sigma =0$ we obtain the semisimple orbits--in particular if
$\tau$ is regular in the sense that $C$ is the maximal torus $T$ then
$\sigma$ must be $0$ and we obtain regular semisimple orbits
\cite{Kronheimer:1990}. At the other extreme,
if $\tau=0$ we get the nilpotent orbits \cite{Kronheimer:nilpotent}.

In these arguments one uses the $L^2$ metric to define a 
hyperk\"ahler structure--this metric is finite because the $\tau$ are fixed.


\section{Stratifying Nahm data by the Centraliser of their Limit}
Let us now, as discussed in \cite{DKS-Sesh}, consider how one might
approach hyperk\"ahler implosion for a compact group $K$ in this framework. We wish to obtain
a hyperk\"ahler space (with some singularities) carrying commuting actions of $K$ and its maximal torus $T$,
such that the hyperk\"ahler reductions by the torus give the Kostant varieties,
that is, the varieties obtained by fixing values of the invariant
polynomials on the complex Lie algebra. These varieties are unions
of complex coadjoint orbits, and in the regular semsimple case
are exactly coadjoint orbits.

Let $\mathfrak k$ be the Lie algebra of $K$ equipped with a
bi-invariant metric $\langle,\rangle$. We will consider $\kf$-valued
Nahm matrices on $[0,\infty)$.  As we now want to obtain all
coadjoint orbits, we must allow the commuting triple and $\su(2)$-triple
to vary.  Accordingly we fix a Cartan sub-algebra $\mathfrak t \subset
\mathfrak k$, and consider Nahm data $T_i:[0,\infty)\to\mathfrak k$
which converge to a limit $T_i(\infty) =\tau_i\in \mathfrak t$, but we
allow the $\tau_i$ to vary in $\mathfrak t$.  As above we use $\tau$ 
to denote the commuting triple
$(\tau_1,\tau_2,\tau_3)$. 

We recall that the universal symplectic implosion may be stratified
by the centraliser of the element of the closed Weyl chamber.
Motivated by this, in the hyperk\"ahler situation we stratify the
space of Nahm matrices with limits in $\mathfrak t$ by the
centraliser $C = C(\tau_1,\tau_2,\tau_3)$ 
\index{$C$} 
of the limiting triple.

It is well-known that formally Nahm's equations can be interpreted
as the hyperk\"ahler moment map for the action of the gauge group on
the infinite-dimensional quaternionic vector space of Nahm data.

Our goal is to construct this  hyperk\"ahler quotient in such a way that  the subspace of solutions to Nahm's equations converging to a fixed triple $\tau$ with centraliser $C$ stratifies
into torus bundles over the strata (complex coadjoint orbits) of the
Kostant variety defined by $\tau$. Moreover, $\tau$
should be the value of the moment map for the action of the maximal torus
$T$ of $K$. The torus fibre over a
semisimple orbit will be $C/[C,C]$ which equals $T$
in the generic case when $C(\tau) = T$. For other strata the fibres
are quotients of $T$.
In this way the Kostant varieties
will arise as hyperk\"ahler quotients of our Nahm moduli space by $T$, as desired.

In order to realise the space of Nahm data with fixed centraliser as a hyperk\"ahler manifold, and to obtain the statements about moment maps, 
we have to provide an analytical framework. In particular, since we work on the half-line $[0,\infty)$, we have to specify the asymptotic behaviour of the Nahm data and gauge transformations. The $L^2$-metric
is no longer appropriate as it is infinite in directions which correspond to
varying $\tau$. We shall see that a modification due to Bielawski
\cite{Bielawski:1998} of the $L^2$ metric has better finiteness properties which fit in
well with our stratification by common centralisers.


\subsection{The Space of Nahm Data}
We fix a maximal torus $T$ in a simply-connected compact Lie group
$K$ with Lie algebra $\tf = \mathrm{Lie}(T)$. Biquard has proved in \cite{Biquard:1996} that any solution $(T_0,T_1,T_2,T_3)$ to the Nahm equations can be put into a gauge such that there exists a commuting triple $\tau = (\tau_1,\tau_2,\tau_3)$ 
of elements $\tau_i$ in $\tf$ with common centraliser $C =C(\tau)\subset K$ such that
\begin{eqnarray*}
T_0^D &=& 0, \\
T_i^D  &=& \tau_i +\frac{\sigma_i}{2(t+1)} + \mathrm O((1+t)^{-(1+\zeta)}),\quad i=1,2,3, \quad \text{for some $\zeta = \zeta(\sigma)>0$,}\\
T_i^H &=& \mathrm O (e^{-\eta t}) \quad i=0,1,2,3, \quad \text{for some $\eta = \eta(\tau)>0$.}
\end{eqnarray*}
We have used Biquard's notation in \cite{Biquard:1996} to write a Lie algebra element $\xi\in \kf$ as $\xi^D +\xi^H$ 
\index{$\xi^D, \xi^H$} 
with respect to the orthogonal decomposition $\kf = \cf \oplus \cf^\perp$.
Moreover, $\sigma = (\sigma_1,\sigma_2,\sigma_3)$ is an $\su(2)$-triple in $\cf=\mathrm{Lie}(C)$, i.e. $-2\sigma_i = [\sigma_j,\sigma_k]$ whenever $(ijk)$ is an even permutation of $(123)$. This condition means in particular that $\sigma_i\in[\cf,\cf]$.

This asymptotic behaviour is a result of the existence of a Coulomb gauge, as proved in \cite{Biquard:1996}, Proposition 2.1 and Corollary 2.2: Writing the solution $T = (T_0,T_1,T_2,T_3)$ as $(0,\tau_1,\tau_2,\tau_3) + \tilde T$, Biquard showed that if $\tilde T$ is sufficiently small in a suitable norm, then one can put $T$ into Coulomb gauge with respect to the model solution $(0,\tau_1,\tau_2,\tau_3)$. That is, $\tilde T$ satisfies 
$$\frac{d \tilde T_0}{dt} - \sum_{i=1}^3[\tau_i,\tilde T_i] = 0, \quad \tilde T_0(0) = 0.$$
Moreover, one has the estimates
\begin{equation}\label{CoulombEstimates}
\sup|(1+t)\tilde T_i^D|, \sup|(1+t)^2\tilde T_i^H|,\sup|(1+t)^2\dot{\tilde{T_i}}|\leq \kappa\sup_{i,j}\{|(1+t)^2(\dot T_i+[T_0,T_i])|, |(1+t)^2[T_i,T_j]|\},
\end{equation}
where $\kappa>0$ is a constant.
Combining these estimates, the Coulomb gauge condition and Nahm's equations, a Weitzenb\"ock argument then gives the exponential decay of the $H$-components, whose rate $\eta$ only depends on the $\tau_i$. The constraint is that $\eta$ should be less than or equal the first positive eigenvalue of 
$\sum_{i=1}^3 -\mathrm{ad}(\tau_i)^2$.

The parameter $\zeta >0$ has to be chosen as follows. Define the Chern-Simons function 
$$\phi: \mathfrak c^3 \to \mathbb R, \qquad \phi(\xi_1,\xi_2,\xi_3) = \sum_{i=1}^3\langle \xi_i,\xi_i\rangle + \langle \xi_1,[\xi_2,\xi_3]\rangle,$$
whose critical points are precisely the $\mathfrak{su}(2)$-triples. 
Also define the Casimir operator associated with the adjoint action of the $\mathfrak{su}(2)$-triple $\sigma$ on $\mathfrak c$ to be 
$$\gamma(\sigma) = - \sum_{i=1}^3 \mathrm{ad}(\sigma_i)^2\in \mathrm{End}(\mathfrak c).$$
Then in the above gauge fixing result  Biquard requires $\zeta$ to be chosen less than or equal to the first positive eigenvalue of $\mathrm{Hess}(\phi)_\sigma$. Moreover, in order to make certain transversality arguments work, he needs $\zeta$ to be smaller than the first positive eigenvalue of the Casimir $\gamma(\sigma)$.

An easy calculation leads to the following observation: 
\begin{lemma}
Let $\sigma$ be a critical point of $\phi$, i.e. an  $\mathfrak{su}(2)$-triple in $\mathfrak c$ and let $u\in C$. Then the Hessian of $\phi$ satisfies 
$$\mathrm{Hess}(\phi)_{u\sigma u^{-1}}(\xi,\psi) = \mathrm{Hess}(\phi)_{\sigma}(u^{-1}\xi u,u^{-1}\psi u).$$
In particular, the eigenvalues of the Hessian  $\mathrm{Hess}(\phi)_\sigma$ at $\sigma$ depend only on the conjugacy class of $\sigma$.
\end{lemma}
\begin{proof}
The Hessian of $\phi$ at the point $\sigma\in \mathfrak c^3$ applied to $\xi,\psi\in\mathfrak c^3$ is given by 
\begin{eqnarray*}
\mathrm{Hess}(\phi)_\sigma(\xi,\psi) &=&\frac{d^2}{dsdt}|_{(s,t) = (0,0)}\phi(\sigma +s\xi+t\psi)\\
&=& \langle \xi_1, 2\psi_1-[\sigma_3,\psi_2] + [\sigma_2,\psi_3]\rangle+ \langle \xi_2,2\psi_2 +[\sigma_3,\psi_1,]- [\sigma_1,\psi_3,]\rangle \\ & & + \langle \xi_3,2\psi_3+ [\sigma_1,\psi_2]- [\sigma_2,\psi_1]\rangle.\\
\end{eqnarray*}
Now since for any $u\in C$ we have 
\begin{equation}\label{conj}
[u\sigma_iu^{-1},\psi_j] = u[\sigma_i,u^{-1}\psi_ju] u^{-1},
\end{equation}
it follows from the bi-invariance of the metric that 
$$\mathrm{Hess}(\phi)_{u\sigma u^{-1}}(\xi,\psi) = \mathrm{Hess}(\phi)_{\sigma}(u^{-1}\xi u,u^{-1}\psi u).$$
\end{proof}
In addition, the relation (\ref{conj}) implies that 
$$\gamma(\mathrm{Ad}(u)(\sigma)) = \mathrm{Ad}(u)\circ \gamma(\sigma)\circ \mathrm{Ad}(u^{-1}).$$
So altogether we deduce the following fact.

\begin{corollary} \label{eigenval}
The eigenvalues of $\mathrm{Hess}(\phi)_\sigma$ and $\gamma(\sigma)$ and hence the constraints  on the parameter $\zeta$ only depend on the conjugacy class of $\sigma$ in $C$.
\end{corollary}

\medskip
Our gauge-theoretic framework will be based on certain weighted function spaces, which we introduce now. Fix $\zeta>0$ and introduce the Banach space of $\mathfrak k$-valued $C^1$-functions, which up to first derivative decay like $1/(1+t)^{1+\zeta}$: 
$$\Omega_\zeta(\kf) = \{f\in C^1([0,\infty), \mathfrak k)\ |\  \Vert f\Vert_{\Omega_\zeta} := \sup_{t}((1+t)^{1+\zeta}|f(t)| + (1+t)^{2+\zeta}|\dot f(t)|) < \infty\}. $$
\index{$\Omega_\zeta(\kf)$}
As usual we write $\dot f = \frac{df}{dt}$. Note that if $\zeta'>\zeta$, then we have a continuous inclusion $\Omega_{\zeta'}\subset \Omega_\zeta$. 

We also need to work with the Banach space $\Omega_{\exp;\eta}(\kf)$ 
\index{$\Omega_{\exp;\eta}(\kf)$} 
of $C^1$-functions which up to first derivative decay like $e^{-\eta t}$ for $\eta>0$:
$$\Omega_{\exp;\eta}(\mathfrak k) = \{f\in C^1([0,\infty), \mathfrak k) \ |\  \Vert f\Vert_{\Omega_{\exp;\eta}} := \sup_{t}(e^{\eta t}|f(t)|) + \sup_t(e^{\eta t}|\dot f(t)|) < \infty\}.$$
Again we have continuous inclusions $\Omega_{\exp;\eta'} \subset \Omega_{\exp;\eta}$ whenever $\eta'>\eta$.
We will also consider the space of exponentially decaying $C^1$-functions 
\index{$\Omega_{\exp}(\kf)$}
$$\Omega_{\exp}(\kf) = \bigcup_{\eta>0}\Omega_{\exp;\eta}(\kf) = \{f\in C^1([0,\infty), \mathfrak k) \ |\  \exists \eta>0 : f\in \Omega_{\exp;\eta}(\kf)\}.$$

Note that the latter space does not come equipped with a natural norm.

\begin{lemma}
Fix positive real numbers $\zeta >0$ and $\eta>0$. Then the spaces $\Omega_\zeta(\kf)$ and $\Omega_{\exp;\eta}(\kf)$ as well as $\Omega_{\exp}(\kf)$ 
are closed under point-wise Lie bracket. The Lie algebra $\mathfrak k$, considered as constant functions on $[0,\infty)$, acts on each of them by the adjoint action.
\end{lemma}

\begin{proof}
By the definition of $\Omega_\zeta(\kf)$ there exist constants $A_1,A_2$ for $\xi \in\Omega_\zeta(\kf)$ such that 
$$|\xi(t)| \leq \frac{A_1}{(1+t)^{1+\zeta}}, \qquad |\dot\xi(t)| \leq\frac{A_2}{(1+t)^{2+\zeta}}.$$
The standard inequalities 
$$|[\xi(t),\psi(t)]| \leq 2|\xi(t)||\psi(t)|$$
and 
$$\left|\frac{d}{dt}[\xi(t),\psi(t)]\right| \leq 2(|\dot\xi(t)||\psi(t)| +|\xi(t)||\dot\psi(t)|)$$
give the result for $\Omega_\zeta(\kf)$. 

The proofs for $\Omega_{\exp;\eta}(\kf)$ and $\Omega_{\exp}(\kf)$ work in a similar fashion.
\end{proof}

When it is clear in which Lie algebra the functions we are considering take their values, we will often drop the $\kf$ from the notation and simply write $\Omega_\zeta$ 
\index{$\Omega_\zeta$} 
etc. 

\bigskip
We define the space $\mathcal A$ 
\index{$\mathcal A$} 
of Nahm data to be the space of all $C^1$-functions $T=(T_0,T_1,T_2,T_3): [0,\infty) \to \kf \otimes \R^4$ such that there exists a commuting triple $\tau\subset \tf$ and an $\su(2)$-triple $\sigma\subset \cf(\tau) = \mathrm{Lie}(C(\tau))$ with 
$$T_0, T_i-\tau_i-\frac{\sigma_i}{2(t+1)} \in\Omega_{\zeta}(\cf(\tau))\oplus\Omega_{\exp}(\cf(\tau)^\perp)\quad i=1,2,3.$$
Here $\zeta$ is fixed as follows.
As remarked in the introduction, the space $\mathcal A$ is stratified by the centraliser $C$ of the limiting triple $\tau$. After applying an $SO(3)$ rotation
(which does not alter $C$)
one can make $\tau$ Biquard-regular, that is $C(\tau) = C(\tau_2, \tau_3)$,
and now $K_{\C} /C_{\C}$ is a semisimple orbit for $K_{\C}$ \cite{Kovalev:1996},
\cite{Biquard:1996}, and $C_{\C}$ and $C$ are
connected as we are assuming $K$ simply-connected 
\cite{Springer-Steinberg:1970}. Note also that
as the maximal torus $T$ of $K$ is contained in $C$, then $T$ is
a maximal torus in $C$ and hence contains the identity component of $Z(C)$
as the latter is fixed by conjugation in $C$. In particular
$Z(\cf) \subset \tf$.

It has been observed in \cite{DKS-Sesh} that the space of all $\mathfrak{su}(2)$-triples in $\mathfrak c$ is the union of finitely many $C$-orbits, where $C$ acts by conjugation (see also the appendix in \cite{Kronheimer:nilpotent}). Moreover, there are only finitely many strata, as there are only finitely many subgroups of $K$ that appear as centralisers of elements of $\tf$, and hence only finitely many subgroups that appear as triple intersections of these centralisers. Therefore, the above discussion in Corollary \ref{eigenval}
shows that we can choose a single $\zeta>0$ such that the Biquard constraints are satisfied simultaneously for all $\su(2)$-triples $\sigma\subset\cf$ and all centralisers $\cf$.

\begin{definition}
Fix a compact subgroup $C\subset K$ centralising a triple $(\tau_1,\tau_2,\tau_3)$ in $\tf$. We define 
the stratum 
\index{$\mathcal A_{C}$}
$$\mathcal A_{C} = \left\{T\in\mathcal A \ | \ T_0, T_i-\tau_i-\frac{\sigma_i}{2(t+1)} \in\Omega_{\zeta}(\cf)\oplus\Omega_{\exp}(\cf^\perp), i=1,2,3, \; : \;
C(\tau) = C, \;\; \sigma\subset \cf\right\}.$$
\end{definition}

The space  $\mathcal A_{C}$ consists thus of Nahm data asymptotic to a model solution
$T_{\tau,\sigma} =  \left(0, \left(\tau_i+\frac{\sigma_i}{2(t+1)}\right)_{i=1}^3\right)$, but $\sigma$ and $\tau$ are allowed to vary subject to the above relations and the constraint that $C(\tau) = C$. Note that the $\tau$ satisfying this condition form an open set $U$ in a vector
subspace of $\tf^3$, while the condition that $\sigma \subset \cf$ defines a compact
space $V$ of triples $\sigma$, acted on by $C$ with finitely many orbits. Now
$\mathcal A_{C}$ is an affine bundle over $U \times V$ with fibres copies of
$\Omega_{\zeta} (\cf) \oplus \Omega_{\exp}(\cf^\perp)$.

\bigskip
At a point $T = (T_0,T_1,T_2,T_3) = T_{\tau,\sigma} +\tilde{T}$, with $\tilde{T}\in (\Omega_\zeta(\cf)\oplus\Omega_{\exp}(\cf^\perp))^4$, a tangent vector to $\mathcal A_{C}$ may be written as a quadruple 
$X  = (X_0,X_1,X_2,X_3):[0,\infty) \to \mathfrak k\otimes \mathbb R^4$.

We use the notation 
\index{$X_{\delta,\epsilon}$}
 $X_{\delta,\epsilon} = \left(0, \left(\delta_i+\frac{\epsilon_i}{2(t+1)}\right)_{i=1}^3\right)$ 
and can write $X - X_{\delta,\epsilon} \in (\Omega_\zeta(\cf)\oplus\Omega_{\exp}(\cf^\perp))^4$ for some triples $\delta, \epsilon$.
In Biquard's notation this means
$$X_0^D, X_i^D-\delta_i-\frac{\epsilon_i}{2(t+1)} \in \Omega_\zeta(\cf),\quad i=1,2,3,  \qquad X_i^H\in \Omega_{\exp}(\cf^\perp), \quad i=0,1,2,3.$$
We shall see shortly that $\delta, \epsilon$ must satisfy certain algebraic relations implied by the condition $[\tau_i,\sigma_j] = 0$ and the other algebraic properties of the $\tau_i$'s and $\sigma_i$'s.

Observe that since $X_0$ satisfies asymptotic conditions different from the ones satisfied by the other $X_i$ the spaces $\mathcal A_{C}$ do \emph{not} carry the natural hypercomplex structure coming from the identification $\R^4\cong \mathbb H$.

As mentioned above, the $\delta_i$ terms mean that tangent vectors
are not in general $L^2$. To deal with this difficulty, rather than use
the $L^2$ metric 
we shall put on $\mathcal A_{C}$ Bielawski's symmetric bilinear form (see \cite{Bielawski:1998}): 
\index{$\Vert X\Vert_{B,b}^2$}
$$\Vert X\Vert_{B,b}^2 = b\sum_{i=0}^3 \langle X_i(\infty), X_i(\infty)\rangle + \int_0^\infty \sum_{i=0}^3(\langle X_i(t), X_i(t)\rangle - \langle X_i(\infty), X_i(\infty)\rangle)dt,$$
where $b\in\mathbb R$ is an arbitrary real constant.
 
In order to check that this gives a well-defined  metric on $\mathcal A_{C}$, i.e. that $\Vert X\Vert_{B,b}^2 <\infty$ for any tangent vector $X$, we have to collect some properties of the $\delta_i = X_i(\infty)$ and $\epsilon_i$ implied by the asymptotic conditions we imposed on our Nahm data. 
The next lemma is essentially contained in \cite{DKS-Sesh}.

\begin{lemma} \label{algrel} 
For a tangent vector $X = X_{\delta,\epsilon} +\tilde X \in T_{T} \mathcal A_{C}$, where $T = T_{\tau,\sigma}+\tilde{T} \in \mathcal A_{C}$, we have the following algebraic relations:
\begin{enumerate}
\item $\delta_i\in Z(\mathfrak c)$, that is $[\delta_i, \gamma] = 0 \quad \forall \gamma\in \mathfrak c,\forall i\geq 1$. In particular $[\delta_i,\tau_j] = 0$
for all $i,j\geq 1$,
\item $[\tau_i,\epsilon_j] = 0 \quad \forall i,j,$
\item $[\delta_i,\epsilon_j] = 0,$
\item $\langle\gamma,\delta_i\rangle = 0 \quad\forall \gamma \in [\mathfrak c,\mathfrak c], i = 1,2,3.$
\item $\langle \delta_i, \epsilon_i \rangle =0$.
\end{enumerate}
\end{lemma}
\begin{proof}
(1) The triple  $(\delta_1,\delta_2,\delta_3)$ is tangent to the space of commuting triples with centraliser $C$. This implies in particular that 
$$[\delta_i, \gamma] = 0 \quad \forall \gamma\in \mathfrak c.$$ 
(2) We have $\sigma_i \in [\cf,\cf]$ so $\epsilon_i \in [\cf,\cf]\subset\cf$,
which implies
$$[\tau_i,\epsilon_j] = 0 \quad \forall i,j.$$
Alternatively, observe that differentiating the relation 
$[\tau_i, \sigma_j] =0$ gives
\[
[\tau_i, \epsilon_j] + [\delta_i, \sigma_j] =0
\]
and the second term vanishes by (1).

(3)
Since $[\delta_i,\gamma] = 0$ for any $\gamma\in\mathfrak c$ we observe that
$$[\delta_i,\epsilon_j] = 0.$$ 

(4)
Now for any $\xi \in \mathfrak k$ the map $\mathrm{ad}(\xi) = [\xi,-] \in \mathrm{End}(\mathfrak k)$ is skew-symmetric with respect to the invariant inner product we fixed. Since $\delta_i$ is contained in the the centre $Z(\mathfrak c)$, i.e. $\mathrm{ad}(\delta_i)$ vanishes on $\mathfrak c$ for $i=1,2,3$, we have, for $\gamma_1,\gamma_2\in \mathfrak c$
$$\langle[\gamma_1,\gamma_2],\delta_i\rangle = -\langle\gamma_2,[\gamma_1,\delta_i]\rangle =\langle \gamma_2 ,0\rangle =0.$$ 

(5) This follows from (4), and the fact that $\epsilon_i \in [\cf,\cf]$.
\end{proof}
We are now in a position to check the finiteness of the Bielawski metric, again
following \cite{DKS-Sesh}.

\begin{proposition}\label{metricfinite}
If $T\in \mathcal A_{C}$ and $X \in T_T\mathcal A_{C}$, then
$$\Vert X\Vert_{B,b}^2 = b\sum_{i=0}^3 \langle X_i(\infty), X_i(\infty)\rangle + \int_0^\infty \sum_{i=0}^3(\langle X_i(t), X_i(t)\rangle - \langle X_i(\infty), X_i(\infty)\rangle)dt <\infty.$$
Moreover, the Bielawski metric is non-degenerate on $T_T\mathcal A_C$.

Thus, we may interpret $(\mathcal A_{C}, \Vert\cdot\Vert_{B,b})$ as an 
infinite-dimensional space with a translation-invariant non-degenerate
symmetric bilinear form..
\end{proposition}
\begin{proof}
Given $X \in T_T\mathcal A_{C}$, we know that there exist $\epsilon_i$ and $\delta_i$ as in the above discussion such that
$$X_0 = \mathrm O((1+t)^{-(1+\zeta)})$$
and
$$X_i = \delta_i +\frac{\epsilon_i}{2(t+1)} + \mathrm O((1+t)^{-(1+\zeta)}),\quad i=1,2,3.$$
The term involving $b$ and $X_i(\infty)$ is therefore always finite. So we may assume $b=0$ and it is enough to prove the proposition for $\Vert\cdot\Vert_B := \Vert\cdot\Vert_{B,0}$.
Then 
\begin{eqnarray*}
\Vert X\Vert_B^2 &=& \int_0^\infty \sum_{i=0}^3(\langle X_i(t), X_i(t)\rangle - \langle X_i(\infty), X_i(\infty)\rangle)dt \\
&=& \int_0^\infty \sum_{i=0}^3(\langle X_i(t), X_i(t)\rangle - \langle \delta_i, \delta_i\rangle)dt \\
 &=& \int_0^\infty \left( \langle \delta_i +\frac{\epsilon_i}{2(t+1)}, \delta_i +\frac{\epsilon_i}{2(t+1)}\rangle - \langle \delta_i, \delta_i\rangle + \mathrm O((1+t)^{-(1+\zeta)})\right)dt\\
&=& \int_0^\infty \left(\frac{\langle\delta_i,\epsilon_i\rangle}{t+1} + \frac{\langle\epsilon_i,\epsilon_i\rangle}{4(t+1)^2} + \mathrm O((1+t)^{-(1+\zeta)})\right)dt\\
&=& \int_0^\infty \left(\frac{\langle\epsilon_i,\epsilon_i\rangle}{4(t+1)^2} + \mathrm O((1+t)^{-(1+\zeta)})\right)dt,
\end{eqnarray*} 
which is finite. In the last step we have used Lemma \ref{algrel}, in particular
the result (5) that $\langle\delta_i,\epsilon_i\rangle = 0$.

To see that the Bielawski metric is non-degenerate, we observe that any quadruple of $\kf$-valued $C^1$-functions on $[0,\infty)$ with compact support defines a tangent vector. Hence, we may use bump functions to argue by contradiction as usual: Suppose $X\in T_T\mathcal A_C$ satisfies 
$$\langle X,Y\rangle_{B,b} = 0,\quad \forall Y\in T_T\mathcal A_C.$$
If $X$ does not vanish, then there exists a $t_0\in[0,\infty)$ such that $X_i(t_0) \neq 0$ for some $i\in \{0,1,2,3\}$. By continuity, there exists a $c>0$ such that  $X_i(t)\neq 0$ for all $t$ such that $|t-t_0|< c$. Now pick a non-negative bump function $h$ with compact support in $(t_0-c, t_0+c)$ such that $h(t_0) =1$ and consider the tangent vector $Y$ whose only non-zero component is given by $Y_i(t) = h(t)X_i(t)$. Then, since $Y(\infty) =0$, we get
$$0 = \langle X,Y\rangle_{B,b} = \int_{t_0-c}^{t_0+c} h(t)\langle X_i(t), X_i(t)\rangle dt >0,$$
a contradiction.
\end{proof}

We have thus defined on $\mathcal A_{C}$ a non-degenerate translation-invariant symmetric bilinear form. We shall see in the following two examples that it is indefinite, so that, formally, we may think of $\mathcal A_C$ as a flat infinite-dimensional pseudo-Riemannian manifold.

\begin{example} 
Put $b=0$ and take a model solution  $T_{\tau,\sigma} = \left(0,\left(\tau_i+\frac{\sigma_i}{2(t+1)}\right)_{i=1}^3\right)$, as well as $\delta_i \in Z(\mathfrak c)$ for $i=0,1,2,3$ and consider for sufficiently small $\theta$ the family
$$T_{\tau,\sigma}^\theta =  \left(0,\left(\tau_i +\theta\delta_i+\frac{\sigma_i}{2(t+1)}\right)_{i=1}^3 \right).$$
Then the tangent vector 
$$X = \frac{d}{d\theta}|_{\theta = 0} T_{\tau, \sigma}^\theta = (0,\delta_1,\delta_2,\delta_3)$$
satisfies
$$\Vert X\Vert_{B,0}^2 = \int_0^\infty \sum_{i=1}^3\left( \langle X_i(t)X_i(t)\rangle - \langle\delta_i,\delta_i\rangle\right)dt =  \int_0^\infty \sum_{i=1}^3\left( \langle \delta_i,\delta_i\rangle - \langle\delta_i,\delta_i\rangle\right)dt = 0.$$

Thus, tangent directions corresponding to deforming the limiting triple $\tau$ give rise to examples of null vectors in $(T\mathcal A_{C},\Vert\cdot\Vert_{B,0})$, i.e. if we fix the parameter $b$ in the definition of the Bielawski bilinear form to be zero.
\end{example}
\begin{example}
For $b>0$ we may consider for $\eta>0$ and $\delta_1\in Z(\cf)$ the following tangent vector: $X = (0, (1-e^{-\eta t})\delta_1, 0,0)$. Then writing $|\delta_1|^2 = \langle \delta_1, \delta_1\rangle$, we find
$$\Vert X\Vert_{B,b}^2 = b|\delta_1|^2 + \int_0^\infty \left( (1-e^{-\eta t})^2|\delta_1|^2 - |\delta_1|^2\right)dt =  |\delta_1|^2 \left(b+ \int_0^\infty (e^{-2\eta t} - 2e^{-\eta t})dt\right) =  |\delta_1|^2 \left(b-\frac{3}{2\eta}\right),$$
which is \emph{negative} if we choose $\eta<\frac{3}{2b}$ and equal to zero if $\eta=\frac{3}{2b}$.
\end{example} 
It is useful to decompose the Lie algebra $\mathfrak k$ into orthogonal subspaces as follows
$$\mathfrak k = Z(\cf) \oplus [\cf,\cf] \oplus \cf^\perp.$$
We write $\xi\in\kf$ as 
$$\xi= \xi^{D,0} + \xi^{D,1}+\xi^H,$$
with respect to this decomposition. Recall that for tangent vectors we have $X_i(\infty) = \delta_i\in Z(\cf)$, i.e. 
$$X_i(\infty) = X_i^{D,0}(\infty).$$
By definition, then, the Bielawski metric can be written as follows
\begin{eqnarray*}
\Vert X\Vert_{B,b}^2 &=& b\sum_{i=0}^3 \langle X_i(\infty), X_i(\infty)\rangle + \int_0^\infty \sum_{i=0}^3(\langle X_i(t), X_i(t)\rangle - \langle X_i(\infty), X_i(\infty)\rangle)dt\\
&=&  b\sum_{i=0}^3 \langle \delta_i, \delta_i\rangle + \int_0^\infty \sum_{i=0}^3(\langle X_i^H(t), X_i^H(t)\rangle +\langle X_i^{D,1}(t), X_i^{D,1}(t)\rangle+\langle X_i^{D,0}(t), X_i^{D,0}(t)\rangle - \langle \delta_i, \delta_i\rangle)dt\\
&=& \Vert X^H\Vert_{L^2}^2 + \Vert X^{D,1}\Vert_{L^2}^2 + \Vert X^{D,0}\Vert_{B,b}^2.
\end{eqnarray*}
Thus, it is the $Z(\cf)$-component of a tangent vector which we have to consider in order to study definiteness properties of the Bielawski metric. 

\medskip
On the other hand, for any element $T \in \mathcal A_{C}$ and any $b\in\mathbb R$, the metric $\Vert\cdot\Vert_{B,b}$ is certainly positive definite on the space of tangent vectors which vanish at infinity, i.e. $X$ such that $\delta_i=0$. In fact, for such vectors the metric $\Vert\cdot \Vert_{B,b}$ reduces to the usual $L^2$-metric.  In particular, if we restrict to a subvariety where the triple $\tau$ is fixed, then
$\Vert \cdot \Vert_{B,b}$ reduces to the $L^2$ metric for any choice of the parameter $b$. The process of obtaining Kostant varieties by torus reduction will
involve fixing $\tau$ in this way.

\subsection{The Gauge Group}
Let $G\subset C \subset K$ be a Lie subgroup with Lie algebra $\mathfrak g$. On the stratum $\mathcal A_C$ we shall describe
an action of the gauge
group $\mathcal G_{G}$, 
\index{$\mathcal G_{G}$} 
whose Lie algebra $\mathrm{Lie}(\mathcal G_{G})$ 
\index{$\mathrm{Lie}(\mathcal G_{G})$} 
is given by the space of $C^2$ paths $\xi:[0,\infty) \to \mathfrak k$ such that
\begin{itemize}
\item $\xi(0) = 0$, 
\item $\xi$ has a limit $\xi(\infty) \in \gf$
\item $\dot \xi \in \Omega_\zeta(\cf)\oplus \Omega_{\exp}(\cf^\perp).$
\end{itemize}
The cases which will be of most interest to us are when $G$ equals one of $C,[C,C],C\cap C(\sigma),\{1\}$.
\begin{lemma} \label{fourth}
For $\xi$ in $\mathrm{Lie}(\mathcal G_{G})$, we have
 $[\tau+\frac{\sigma}{2(t+1)},\xi(\infty)-\xi(t)^D] \in \Omega_{\zeta}(\cf)$ 
and $[\tau+\frac{\sigma}{2(t+1)},\xi(t)^H] \in \Omega_{\exp}(\cf^\perp)$ 
for all $\tau\in \tf$ and $\sigma\in [\mathfrak c,\mathfrak c].$
\end{lemma}
\begin{proof}
Let $\xi\in \mathrm{Lie}(\mathcal G_{G})$, so $\xi(0) = 0$ and $\xi$ has a limit in the Lie algebra $\gf\subset\cf$, i.e. $\xi(\infty) = \xi(\infty)^D$. Moreover,  we have  $\dot{\xi}^D\in\Omega_\zeta(\cf)$  as well as $\dot{\xi}^H\in\Omega_{\exp}(\cf^\perp)$. 

We check that the two terms in the lemma have the required asymptotic behaviour. By writing 
$$\xi(\infty) - \xi(t) = \int_t^\infty\dot{\xi}(s)ds,$$
and using $\dot{\xi}^D\in\Omega_\zeta(\cf)$, $\dot{\xi}^H\in\Omega_{\exp}(\cf^\perp)$, we see that  $$\xi(\infty) - \xi^D = \mathrm O((1+t)^{-\zeta})\quad \text{and}\quad \xi^H \in \Omega_{\exp}(\cf^\perp).$$ 
Then since $\tau$ acts trivially on $\cf$, i.e. commutes with $\xi^D$, we find 
$$[\tau + \frac{\sigma}{2(t+1)},\xi(\infty)-\xi(t)^D] = [\frac{\sigma}{2(t+1)},\xi(\infty)-\xi(t)^D] = \mathrm O((1+t)^{-(1+\zeta)})$$
and 
\begin{eqnarray*}
\frac{d}{dt}[\tau + \frac{\sigma}{2(t+1)},\xi(\infty)-\xi(t)^D] &=& \frac{d}{dt} [\frac{\sigma}{2(t+1)},\xi(\infty)-\xi(t)^D] \\
&=& -[\frac{\sigma}{2(t+1)^2},\xi(\infty)-\xi(t)^D] - [\frac{\sigma}{2(t+1)},\dot{\xi}(t)^D] \\
&=& \mathrm O((1+t)^{-(2+\zeta)}).
\end{eqnarray*}
For the $\cf^\perp$-component, we calculate similarly
$$[\tau + \frac{\sigma}{2(t+1)},\xi(t)^H] = O(e^{-\eta t})$$
and 
$$\frac{d}{dt}[\tau + \frac{\sigma}{2(t+1)},\xi(t)^H] = -[\frac{\sigma}{2(t+1)^2},\xi(t)^H] + [\tau + \frac{\sigma}{2(t+1)},\dot{\xi}(t)^H] = \mathrm O(e^{-\eta t})$$
for some $\eta >0$.
\end{proof}

\begin{lemma} 
$\mathrm{Lie}(\mathcal G_{G})$ is a Lie algebra.
\end{lemma}
\begin{proof}
We only have to check that $\mathrm{Lie}(\mathcal G_{G})$ is closed under the Lie bracket.
Let $\xi,\psi \in \mathrm{Lie}(\mathcal G_{G})$. 

Clearly, the first and second conditions hold for $[\xi, \psi]$.

To check the third condition, we note that $\dot{\xi},\dot{\psi} = \mathrm O((1+t)^{-(1+\zeta)})$. We have to consider the term  $\frac{d}{dt} [\xi, \psi]$, which we decompose according to the splitting $\kf=\cf\oplus\cf^\perp$
(note that $\frac{d}{dt} (\xi^D) = (\frac{d \xi}{dt})^D$ and similarly for the 
$H$ superscript):
\begin{eqnarray*}
\frac{d}{ dt} [\xi, \psi] &=& [\dot{\xi}^D +\dot{\xi}^H,\psi^D+\psi^H] + [\xi^D+\xi^H,\dot{\psi}^D+\dot{\psi}^H]\\
&=&[\dot{\xi}^D,\psi^D] + [\xi^D,\dot{\psi}^D] + ([\dot{\xi}^H,\psi^H] + [\xi^H,\dot{\psi}^H])^D \\
& & \oplus  [\dot{\xi}^H,\psi^D] + [\dot{\xi}^D,\psi^H] + [\xi^D,\dot{\psi}^H]  + [\xi^H,\dot{\psi}^D] + ([\dot{\xi}^H,\psi^H] + [\xi^H,\dot{\psi}^H])^H.
\end{eqnarray*}
The terms $([\dot{\xi}^H,\psi^H] + [\xi^H,\dot{\psi}^H])^D$ and $([\dot{\xi}^H,\psi^H] + [\xi^H,\dot{\psi}^H])^H$ clearly lie in $\Omega_{\exp}(\cf)\subset\Omega_\zeta(\cf)$ and $\Omega_{\exp}(\cf^\perp)$ respectively.  The remaining terms lie in $\Omega_\zeta(\cf)\oplus\Omega_{\exp}(\cf^\perp)$ as is implied by the following lemma.

\begin{lemma} \label{fg}
\begin{enumerate}
\item
Let $\zeta>0$ and let $f:[0,\infty)\to\mathfrak c$ be bounded with $\dot f  = \mathrm O((1+t)^{-1})$ and let $g\in\Omega_\zeta(\cf)$. Then the bracket $[f,g]$ lies in $\Omega_\zeta(\cf)$.
\item Let $\zeta >0$ and let $f:[0,\infty)\to\mathfrak c$ be such that $f$ and $\dot f$ are bounded. Let $g\in\Omega_{\exp;\eta}(\cf^\perp)$, then the bracket $[f,g]$ lies in $\Omega_{\exp;\eta}(\cf^\perp)$. In particular, if $g\in\Omega_{\exp}(\cf^\perp)$, then $[f,g]\in\Omega_{\exp}(\cf^\perp)$.
\end{enumerate}
\end{lemma}

\begin{proof}
(1) The assumptions on $f$ and $g$ imply that there exist positive constants $A_1,A_2$ and $A_3$ such that 
$$|g(t)|\leq \frac{A_1}{(1+t)^{1+\zeta}}, \qquad
|\dot g(t)|\leq  \frac{A_2}{(1+t)^{2+\zeta}}$$ 
as well as 
$$|\dot f(t)|\leq  \frac{A_3}{(1+t)}.$$ 
Thus,  $$|[f(t),g(t)]| \leq 2|f(t)||g(t)|\leq  \frac{2A_1\sup(|f|)}{(1+t)^{1+\zeta}}$$ and $$\left|\frac{d}{dt}[f(t),g(t)]\right| \leq |[\dot f(t), g(t)]|+|[f,\dot g(t)]| \leq  \frac{2A_1A_3}{(1+t)^{2+\zeta}}+ \frac{2A_2\sup (|f|)}{(1+t)^{2+\zeta}}.$$ 
(2) This is very similar to (1).  By assumption there exists $\eta>0$ and positive constants $A_1$ and $A_2$ such that 
$$|g(t)|\leq A_1e^{-\eta t} \;\;\; : \;\;\; 
|\dot g(t)|\leq A_2e^{-\eta t}.$$ 
Thus,  
$$|[f(t),g(t)]| \leq 2|f(t)||g(t)|\leq  2A_1\sup(|f|)e^{-\eta t}$$
and
$$\left|\frac{d}{dt}[f(t),g(t)]\right| \leq |[\dot f(t), g(t)]|+|[f,\dot g(t)]| \leq  \left(2A_1\sup(|\dot f|)+ 2A_2\sup (|f|)\right)e^{-\eta t}.$$ 
\end{proof}
Altogether we have verified that $[\xi,\psi]\in\mathrm{Lie}(\mathcal G_{G})$, which is therefore indeed a Lie algebra.
\end{proof}

Note that we have actually proved a slightly stronger statement, namely that the space $\mathrm{Lie}(\mathcal G_{G;\eta})$ 
\index{$\mathrm{Lie}(\mathcal G_{G;\eta})$}
obtained by replacing $\Omega_{\exp}$ by $\Omega_{\exp;\eta}$, i.e. fixing the exponential decay rate,  in the definition of $\mathrm{Lie}(\mathcal G_{G})$, is a Lie algebra. 

The space $\mathrm{Lie}(\mathcal G_{G})$ is the Lie algebra of the gauge group $\mathcal G_G$ given by the space of $K$-valued $C^2$-paths
$u:[0,\infty) \to K$ such that 
\begin{itemize}
\item $u(0) = 1,$ 
\item $ u(\infty) \in G\subset [C,C],$
\item $(\dot u u^{-1})^D \in \Omega_\zeta(\cf)$ and $(\dot u u^{-1})^H \in \Omega_{\exp}(\cf^\perp),$
\end{itemize}
It acts on Nahm matrices by the usual gauge action
\begin{equation}\label{gaugeaction}
u.T_0 = \mathrm{Ad}(u)(T_0) - \dot uu^{-1},\qquad u.T_i = \mathrm{Ad}(u)(T_i)\quad i =1,2,3.
\end{equation}

\begin{lemma}\label{uinfty}
Modulo terms in $\Omega_\zeta(\cf) \oplus \Omega_{\exp}(\cf^\perp)$, $u\in\mathcal G_G$ acts on $\tau + \frac{\sigma}{2(t+1)}$
like $u(\infty) \in G \subset [C,C]$. That is, 
$$(\mathrm{Ad}(u)-\mathrm{Ad}(u(\infty)))\left(\tau + \frac{\sigma}{2(t+1)}\right) \in \Omega_\zeta(\cf)\oplus\Omega_{\exp;\eta}(\cf^\perp)$$Note that this preserves the
condition $C(\tau)=C$ (indeed it fixes $\tau$) and the condition
$\sigma \subset \cf$.
\end{lemma} 
\begin{proof}
This is a consequence of the properties of the spaces $\Omega_\zeta(\cf)$ and $\Omega_{\exp}(\cf^\perp)$, but we spell out the details.
We can calculate for a constant $\psi\in\cf$ using $u(\infty) \in G\subset [C,C]$ and the bi-invariance of the norm on $\kf$,
\begin{eqnarray*}
|(\mathrm{Ad}(u(\infty))(\psi) - \mathrm{Ad}(u(t))(\psi))^D| &\leq& \int_t^\infty \left|\frac{d}{ds}(\mathrm{Ad}(u(s))(\psi))^D\right|ds\\
&=&\int_t^\infty|[\dot u(s)u^{-1}(s), \mathrm{Ad}(u(s))(\psi)]^D| ds \\
&=&\int_t^\infty|[u^{-1}(s)\dot u(s), \psi]^D| ds \\
&=& \int_t^\infty|[(u^{-1}(s)\dot u(s))^D, \psi]|ds\\
&\leq& 2|\psi|\Vert (\dot u(s)u^{-1})^D\Vert_{\Omega_\zeta}\int_t^\infty \frac{1}{(1+s)^{1+\zeta}}ds\\
&=& \mathrm O((1+t)^{-\zeta}).
\end{eqnarray*}
Similarly   
\begin{eqnarray*}
|(\mathrm{Ad}(u(\infty))(\psi) - \mathrm{Ad}(u(t))(\psi))^H| &\leq& \int_t^\infty \left|\frac{d}{ds}(\mathrm{Ad}(u(s))(\psi))^H\right|ds\\
&=& \int_t^\infty|[(u^{-1}(s)\dot u(s))^H, \psi]|ds\\
&\leq& 2|\psi|\Vert (\dot u(s)u^{-1})^H\Vert_{\Omega_{\exp;\eta}}\int_t^\infty e^{-\eta s}ds\\
&=& \mathrm O(e^{-\eta t}).
\end{eqnarray*}
We have seen in this calculation that the derivative $\frac{d}{dt}\mathrm{Ad}(u(t))(\psi)$ satisfies 
$$\left|\frac{d}{dt}\mathrm{Ad}(u(t))(\psi)\right| = |[u^{-1}(t)\dot u(t), \psi]|.$$
Since $\psi\in\cf$ is constant, it is immediate that the $D$ and $H$-parts of this are $\mathrm O(1+t)^{-(1+\zeta)}$ and $\mathrm O(e^{-\eta t})$, respectively, because $\mathrm{ad}(\psi)$ preserves the decomposition into $D$- and $H$-parts. 
Consequently,  the $H$-part of $(\mathrm{Ad}(u(t)) - \mathrm{Ad}(u(\infty)))(\tau + \frac{\sigma}{2(t+1)})$ lies in $\Omega_{\exp;\eta}(\cf^\perp)$. 
\medskip

To deal with $D$-component, we first consider $\psi = \tau\in Z(\cf)$. We see that since $[(\dot uu^{-1})^D, \tau] = 0$ the above discussion yields $|(\mathrm{Ad}(u(\infty))(\tau) - \mathrm{Ad}(u(t))(\tau))^D| \equiv 0$, which obviously lies in $\Omega_\zeta(\cf)$. 

If we take instead $\psi = \sigma\in[\cf,\cf]$, we find $|\mathrm{Ad}(u(\infty))(\sigma)^D - \mathrm{Ad}(u(t))(\sigma)^D| = \mathrm O(1+t)^{-\zeta}$ and $|\frac{d}{dt}\mathrm{Ad}(u(t))(\sigma)^D| = \mathrm O(1+t)^{-(1+\zeta)}$. This implies that $\left(\mathrm{Ad}(u(\infty))(\frac{\sigma}{2(t+1)}) - \mathrm{Ad}(u(t))(\frac{\sigma}{2(t+1)})\right)^D \in \Omega_\zeta(\cf)$. 
\end{proof}

In the next lemma we verify that the group in fact does act on the space $\mathcal A_C$, i.e. it preserves the asymptotic conditions imposed on our Nahm data.

\begin{lemma}\label{groupaction}
The gauge group $\mathcal G_G$ acts on the space $(\mathcal A_C, \Vert\cdot \Vert_{B,b})$ by isometries.
\end{lemma}
\begin{proof}
We calculate the fundamental vector fields of the gauge action. At a point $T \in \mathcal A_C$ these are given by 
\index{$X^\xi_T$}
$$X^\xi_T = \frac{d}{d\theta}|_{\theta = 0} \exp(\theta\xi).T.$$
A short calculation gives 
\begin{equation}\label{fundvectorfields}
X^\xi_T = ([\xi,T_0] -\dot \xi, [\xi,T_1],[\xi,T_2],[\xi,T_3]).
\end{equation}
Writing $T = T_{\tau,\sigma} + \tilde T$ with $\tilde T\in (\Omega_\zeta(\cf)\oplus \Omega_{\exp}(\cf^\perp))^4$, and using $[\xi(\infty),\tau_i] =0$, we can write the components of $X^\xi_T$ as follows:
\begin{eqnarray*}
(X^\xi_T)_0 &=&  [\xi(t),\tilde T_0(t)] -\dot\xi(t),\\
(X^\xi_T)_i &=&\frac{[\xi(\infty),\sigma_i]}{2(t+1)} +[\xi(t) -\xi(\infty), \tau_i+\frac{\sigma_i}{2(t+1)}] + [\xi(t),\tilde T_i(t)], \quad i=1,2,3.
\end{eqnarray*}
We shall now show that the asymptotic conditions imposed on $\xi \in \mathrm{Lie}(\mathcal G_{G})$ and $\tilde T_i\in \Omega_\zeta(\cf) \oplus\Omega_{\exp}(\cf^\perp)$ imply that this  has the correct asymptotic behaviour for being a member of $T_T \mathcal A_C$.
 By definition of $\mathrm{Lie}(\mathcal G_G)$, $\dot\xi$ lies in $\Omega_\zeta(\cf) \oplus\Omega_{\exp}(\cf^\perp)$, which is correct. Furthermore, since $\xi(\infty) = \xi(\infty)^D$, we have 
$$[\tau_i+\frac{\sigma_i}{2(t+1)}, \xi(t) -\xi(\infty)] = [\tau_i+\frac{\sigma_i}{2(t+1)}, \xi(t)^D -\xi(\infty)] + [\tau_i+\frac{\sigma_i}{2(t+1)}, \xi(t)^H],$$
which also lies in $\Omega_\zeta(\cf) \oplus\Omega_{\exp}(\cf^\perp)$
by Lemma \ref{fourth}. 
 The $\frac{[\xi(\infty), \sigma_i]}{2(t+1)}$ term satisfies the
required conditions for the $\frac{\epsilon_i}{2(t+1)}$ term in the tangent vector.

Now $\tilde T_i \in \Omega_\zeta(\cf)\oplus\Omega_{\exp}(\cf^\perp)$ so that for $i=0,1,2,3$ we can write 
\begin{eqnarray*}
[\xi,\tilde T_i] &=& [\xi^D+\xi^H,\tilde T_i^D+\tilde T_i^H]\\
&=& ([\xi^D,\tilde T_i^D]+[\xi^H,\tilde T_i^H]^D) \oplus ([\xi^D,\tilde T_i^H]+[\xi^H,\tilde T_i^D] + [\xi^H,\tilde T_i^H]^H).
\end{eqnarray*}
Since $\xi$ and the $\tilde T_i$ are bounded, all terms lie in the correct function spaces to be tangent to $\mathcal A_C$ by lemma \ref{fg}. Thus, the fundamental vector fields of the action are tangent to $\mathcal A_C$,
as required. We therefore obtain an infinitesimal action of $\mathrm{Lie}(\mathcal G_G)$ on $\mathcal A_C$. 
In fact, as in Lemma \ref{fourth} this proves the stronger result that for $\xi\in\mathrm{Lie}(\mathcal G_{G;\eta})$ the vector field $X^\xi$ will be tangent to $\mathcal A_{C;\eta'}$ whenever $\eta\geq\eta'$. Here $\mathcal A_{C;\eta'}$ is the space obtained by replacing $\Omega_{\exp}$ by $\Omega_{\exp;\eta'}$ in the definition of $\mathcal A_C$. 
\index{$\mathcal A_{C;\eta'}$}

We wish to show that this infinitesimal action integrates to an action of $\mathcal G_G$. 
For this, we observe that all the conditions defining $\mathcal A_C$ are asymptotic in nature, and so we only have to worry about the behaviour for large $t$. First of all, for $u\in\mathcal G_G$ the terms $(\dot uu^{-1})^D$ and $(\dot uu^{-1})^H$ clearly satisfy the correct decay conditions. And we have already seen in Lemma \ref{uinfty} that $u$ essentially acts like $u(\infty)$ on a model solution. It thus remains to take care of the terms of the form $\mathrm{Ad}(u)(\tilde T_i)$ for $\tilde T_i\in\Omega_\zeta(\cf)\oplus\Omega_{\exp}(\cf^\perp)$. 

Note that the conditions are preserved by constant $C$-valued gauge transformations, hence also by gauge transformations $u\in\mathcal G_G$ such that $u(t) \equiv u(\infty) \in G$ for all $t$ larger than some $t_0>0$. 

For a general $u\in\mathcal G_G$ with $u(\infty)\in G$, the compactness of $K$ implies that we can write $u(\infty) = \exp(\xi_\infty)$ for some $\xi_\infty\in\gf$. Now take a smooth non-negative function $h$ on $[0,\infty)$ such that $h(0) =0$ and $h(t) \equiv 1$ for $t>1$. Then the gauge transformation $\exp(h(t)\xi_\infty)$ lies in $\mathcal G_G$ and is constant for $t>1$, equal to $u(\infty)\in G$. Hence it preserves the space $\mathcal A_C$. Writing $u = \tilde u \exp(h(t)\xi)$, where $\tilde u(\infty) = 1$, we may thus assume without loss of generality that $u$ is asymptotic to $1\in G$ as $t$ tends to infinity. 

Now $u\in\mathcal G_G$ means that there exists an $\eta>0$ such that $u$ is contained in the Banach Lie group $\mathcal G_{G;\eta}$. Since $u$ tends to $1$ as $t\to \infty$, we may write $u(t) = \exp(\xi(t))$ for $t$ sufficiently large, where $\xi$ satisfies the asymptotic conditions for being a member of $\mathrm{Lie}(\mathcal G_{G;\eta})$. The above discussion shows that $\mathrm{ad}(\xi(t))$ preserves the space $\mathcal A_{C;\eta'}$ provided $\eta'\leq\eta$. 

Since $\mathcal G_{G;\eta}$ is a Banach Lie group, we have the identity $\mathrm{Ad}(\exp\xi) = e^{\mathrm{ad}\xi}$ in $\mathrm{End}(\Omega_\zeta(\cf)\oplus \Omega_{\exp;\eta}(\cf^\perp))$, which hence implies that $\exp(\xi(t))$ preserves the space $\mathcal A_{C;\eta'}$, too. Since $\Omega_{\exp}$ is the nested union of the spaces $\Omega_{\exp;\eta}$, this shows that $u$ preserves the space $\mathcal A_C$.

Finally, observe that the induced action of $\mathcal G_G$ on tangent vectors is just point-wise conjugation which clearly preserves the Bielawski metric.
\end{proof}
We remark that the above results and their proofs remain valid if we work with $\mathcal A_{C;\eta}$ and $\mathcal G_{G;\eta}$, i.e. if we consider functions with fixed decay rate $\eta$.  We will use a similar setting with fixed $\eta$ in $\S 8$, when we discuss a gauge-theoretic picture of the universal symplectic implosion.

\subsection{Constructing a candidate $\mathcal Q$ for the Universal Hyperk\"ahler Implosion}
We are now in a position to give a definition of the Nahm universal hyperk\"ahler implosion $\mathcal Q$. 
\index{$\mathcal Q$}

Let $\mathcal N$ denote the set of solutions to the Nahm equations in $\mathcal A$ and let $\mathcal N_C = \mathcal N \cap \mathcal A_C$. 
\index{$\mathcal N$}\index{$\mathcal N_C$}
In order to obtain the full implosion space $\mathcal Q$ we take the space $\mathcal N$ which is 
the union of the strata $\mathcal N_C$ over all compact subgroups $C$ of $K$ and
first quotient by $\mathcal G_1$, the group of gauge transformations as defined above with $G = \{1\}$, i.e.
asymptotic to the identity at infinity. The strata $\mathcal Q_C$ 
\index{$\mathcal Q_C$} 
of the implosion are obtained by performing further collapsings
on the strata $\mathcal N_C /\mathcal G_1 $ by $[C,C]$, where we identify $[C,C]$ with
$\mathcal G_{[C,C]}/\mathcal G_1$ in the obvious way via the homomorphism
given by evaluation at infinity. In particular on the top stratum, corresponding to triples that are regular in the sense that $C(\tau)=T$, no further
collapsing takes place. This is an obvious analogue of the picture for symplectic implosion discussed in $\S 1$. 

The implosion is now stratified by the moduli spaces $\mathcal Q_C =
\mathcal N_C/\mathcal G_{[C,C]}$. Although formally the Nahm equations
have an interpretation as the zero condition for the hyperk\"ahler
moment map associated with the action of the gauge group on the space
of Nahm data, neither the spaces $\mathcal N_C/\mathcal G_I$ nor
$\mathcal Q_C$ can be interpreted as hyperk\"ahler quotients in a
straightforward way. This is because the natural hypercomplex
structure is not well-defined on $\mathcal A_C$, as we have observed.
However, for each choice of $C$, the action of $[C,C]$ can be used to
move the triples $\sigma$ into one of a finite list of standard
possibilities (one for each nilpotent orbit in the complexification
$C_\mathbb C$). The remaining freedom we have by which to quotient is
$[C,C] \cap C(\sigma)$, the intersection of $[C,C]$ with the common
centraliser of $\sigma$. We thus obtain a refined stratification of
$\mathcal Q_C$ into strata $\mathcal Q_{C,\sigma} = \mathcal
N_{C,\sigma}/(\mathcal G_{[C,C]\cap C(\sigma)})$ indexed by the
finitely many standard triples $\sigma$. Here $\mathcal N_{C,\sigma}$ 
\index{$\mathcal N_{C,\sigma}$}
denotes the space of solutions in $\mathcal N_{C}$ with
$\su(2)$-triple equal to the standard triple $\sigma$.  In $\S 4$ we
will provide an analytical framework to show that the strata $\mathcal
Q_{C,\sigma}$ 
\index{$\mathcal
Q_{C,\sigma}$,$\mathcal Q_{C,\sigma}(b)$} arise in a natural way as hyperk\"ahler quotients. Thus,
$\mathcal Q_C$ is a stratified hyperk\"ahler space with metric
induced by $\Vert\cdot\Vert_{B,b}$.  We shall sometimes write the strata as
$\mathcal Q_{C,\sigma}(b)$ to emphasise the metric dependence on $b$.

Moreover, we will show that hyperk\"ahler reduction of $\mathcal Q_{C,\sigma}$ by the action of the maximal torus $T$ at level $\tau$ then amounts to
fixing $\tau$ and collapsing $\mathcal N_{C,\sigma}/\mathcal G_{1}$ by $C(\sigma)\cap C$ rather than by 
$C(\sigma) \cap [C,C]$. This is equivalent to collapsing the space of solutions in $\mathcal N_C /\mathcal G_{1}$ asymptotic to the triple $\tau$ by $C$. As $\tau$ is now fixed, the Bielawski metric on the hyperk\"ahler torus quotients of the strata $\mathcal Q_{C,\sigma}$ reduces to the usual $L^2$-metric and we will show that, if the triple $\tau$ is Biquard-regular, these torus reductions give the strata of the Kostant variety associated with $\tau$. In other words, the hyperk\"ahler torus quotients of the spaces $\mathcal Q_{C,\sigma}$ now give exactly the coadjoint orbits
making up the Kostant variety, as in \cite{Biquard:1996}, \cite{Kovalev:1996}.
To make this more precise, we shall relate $\mathcal Q_{C,\sigma} \hkq T$ 
to Kronheimer's (respectively Biquard's) moduli spaces used to construct coadjoint orbits.

In particular if $C(\tau)=T$ then we are just fixing $\tau$ and factoring out by gauge transformations asymptotic to an element of $T$ at infinity.
This gives the regular semsimple orbit of $\tau$, which is the
unique stratum in the Kostant variety in this case, as in 
\cite{Kronheimer:1990}. 

On the other hand on the stratum given by $C=K=[K,K]$ (that is, if $\tau=0$), the torus action will be trivial. So we see the nilpotent variety sitting inside the implosion $\mathcal Q$ as the stratum $\mathcal Q_K$. Note that this is not the case in the quiver approach to hyperk\"ahler implosion when $K=\SU(n)$; see \cite{DKS},\cite{DKS-Sesh}. In the latter approach a stratification labelled by pairs $(C,\sigma)$ is obtained, but these strata are not hyperk\"ahler. In Remark \ref{comparison} below we suggest a modification of the construction of $\mathcal Q$ which, at least in the case when $K=\SU(2)$, can be identified with the universal hyperk\"ahler implosion constructed in \cite{DKS}.


\section{The Stratum $\mathcal Q_{C,\sigma}$ as a Hyperk\"ahler Quotient}
As remarked above, $\mathcal A_{C}$ is not a hypercomplex manifold. In order
to exhibit $\mathcal Q_{C}$ as a union of hyperk\"ahler strata
$\mathcal Q_{C, \sigma}$, we shall introduce a new family of spaces which
are better adapted to hyperk\"ahler geometry.

Choose a $C$-conjugacy class of $\su(2)$-triples in $\cf$ represented by the fixed triple  $\sigma =(\sigma_1,\sigma_2,\sigma_3)$. Define $\tilde{\mathcal A}_{C,\sigma}$ 
\index{$\tilde{\mathcal A}_{C,\sigma}$} 
to be the space of quadruples of $C^1$ functions 
$$(T_0,T_1,T_2,T_3):[0,\infty) \to \mathfrak k\otimes\mathbb R^4$$ 
such that there exists $\tau_0\in Z(\cf)$ and a commuting triple $\tau = (\tau_1, \tau_2, \tau_3)$ in $ \mathfrak t$, such that $C(\tau) = C$,  satisfying
\begin{itemize}
\item $T_0^D-\tau_0 \in \Omega_\zeta(\cf)$
\item $T_i^D - \tau_i -\frac{\sigma_i}{2(t+1)} \in \Omega_\zeta(\cf) \quad i=1,2,3, $
\item $T_i^H\in\Omega_{\exp}(\cf^\perp) \quad i=0,1,2,3.$
\end{itemize}
Here $\zeta$ is again fixed, depending on $C$, as explained in the preceding section. If we denote by $\mathcal A_{C,\sigma}$ the space of Nahm data defined as in $\S 3$, but with fixed $\sigma$, then we can write $\mathcal A_{C,\sigma} = \{T\in\tilde{\mathcal A}_{C,\sigma}�|  \tau_0 = 0\}\subset \tilde{\mathcal A}_{C,\sigma}$.

Our Nahm data are therefore asymptotic to Nahm data of the form
$T_{\tau_0,\tau,\sigma} =  \left(\tau_0, \left(\tau_i+\frac{\sigma_i}{2(t+1)}\right)_{i=1}^3\right)$, where $\sigma$ is fixed and the elements $\tau_i\in
\tf$ are allowed to vary subject to the constraint that $C(\tau) = C$, $\tau_0\in Z(\cf)$. Recall from the previous section
that $Z(\cf) \subset \tf$. Again, we call $T_{\tau_0,\tau,\sigma}$ 
\index{$T_{\tau_0,\tau,\sigma}$} 
a \emph{model solution}. Observe that by construction also $\tau_i\in Z(\cf)$ for $i=1,2,3$.
So $(\tau_0, \tau_1, \tau_2, \tau_3)$ lies in the open dense
subset of the quaternionic vector space
$Z(\cf) \otimes {\mathbb H} \subset \tf \otimes \mathbb H$ defined by
the condition that $C(\tau_1, \tau_2, \tau_3)$ is exactly $C$ rather than
just containing $C$.

At a point $T = (T_0,T_1,T_2,T_3) = T_{\tau_0,\tau,\sigma} +\tilde{T}$, with $\tilde{T}\in (\Omega_\zeta(\cf)\oplus\Omega_{\exp}(\cf^\perp))^4$, a tangent vector to $\tilde{\mathcal A}_{C,\sigma}$ may be written as a quadruple 
$X  = (X_0,X_1,X_2,X_3):[0,\infty) \to \mathfrak k\otimes \mathbb R^4$. As we now work with a fixed triple $\sigma$, no $\epsilon_i$-terms appear in the $X_i$.

The proof of Lemma \ref{algrel} shows that $\delta_0 = X_0(\infty)$ satisfies the same algebraic relations as the other $\delta_i$.

We use the notation
 $X_{\delta}  = (\delta_i)_{i=0}^3$ 
 \index{$X_\delta$}
and can now write $X - X_{\delta} \in (\Omega_\zeta(\cf)\oplus\Omega_{\exp}(\cf^\perp))^4$ for some $\delta$.
In other words
$$X_i^D-\delta_i \in \Omega_\zeta(\cf), \qquad X_i^H\in \Omega_{\exp}(\cf^\perp) \quad i=0,1,2,3.$$

We still use on $\tilde{\mathcal A}_{C,\sigma}$ the Bielawski  bilinear form
$$\Vert X\Vert_{B,b}^2 = b\sum_{i=0}^3 \langle X_i(\infty), X_i(\infty)\rangle + \int_0^\infty \sum_{i=0}^3(\langle X_i(t), X_i(t)\rangle - \langle X_i(\infty), X_i(\infty)\rangle)dt.$$
 This is finite on all tangent vectors to $\tilde{\mathcal A}_{C,\sigma}$, since for each tangent vector $X_i-\delta_i$ is integrable and we do not have to deal with $\epsilon_i$-terms, as the triple $\sigma$ is now fixed. The same argument as used in the proof of Lemma \ref{metricfinite} shows that it is non-degenerate at each point of $\tilde{\mathcal A}_{C,\sigma}$

As a consequence of the introduction of the limit $\tau_0$ and of fixing $\sigma$ the space $\tilde{\mathcal A}_{C,\sigma}$ is now hypercomplex: We have three anti-commuting complex structures $I,J,K$ 
\index{$I,J,K$} 
obeying the quaternionic relations induced by right-multiplication by $-i,-j,-k$ on $\mathfrak k\otimes \mathbb R^4 \cong \mathfrak k\otimes \mathbb H$.
Explicitly,
\begin{eqnarray*}
IX &=&(X_1,-X_0,-X_3,X_2),\\
JX &=&(X_2,X_3,-X_0,-X_1),\\
KX &=&(X_3,-X_2,X_1,-X_0).
\end{eqnarray*}
Obviously, these complex structures are orthogonal with respect to the Bielawski metric on $\tilde{\mathcal A}_{C,\sigma}$ and preserve the tangent spaces $T_T\tilde{\mathcal A}_{C,\sigma}$, i.e. take tangent vectors to tangent vectors.
 
Thus, we may interpret $(\tilde{\mathcal A}_{C,\sigma}, \Vert\cdot\Vert_{B,b}, I,J,K)$ 
as an infinite-dimensional flat pseudo-hyperk\"ahler space, i.e. a 
 hypercomplex space with a compatible non-degenerate, translation-invariant symmetric bilinear form. 

\subsection{The Gauge Group}
To define the appropriate gauge group, we draw inspiration from
\cite{Bielawski:1998}. 

\begin{definition}\label{tildeG}
Let  $G\subset (C\cap C(\sigma))$ be a subgroup with Lie algebra $\gf$. We consider the gauge
group $\tilde{\mathcal G}_{G}(b)$, 
\index{$\tilde{\mathcal G}_{G}(b)$} 
whose Lie algebra $\mathrm{Lie}(\tilde{\mathcal G}_{G}(b))$ 
\index{$\mathrm{Lie}(\tilde{\mathcal G}_{G}(b))$} 
is given by the space of $C^2$ paths $\xi:[0,\infty) \to \mathfrak k$ such that
\begin{itemize}
\item $\xi(0) = 0$, 
\item $\dot\xi$ has a limit $\dot\xi(\infty) \in Z(\cf)$,  and $\tilde\xi(\infty) := b\dot\xi(\infty) + \lim_{t\to\infty}(\xi(t) - t\dot\xi(\infty)) \in \gf$,
\item $(\dot \xi^D- \dot\xi(\infty)) \in \Omega_\zeta(\cf)$ and $\dot\xi^H \in \Omega_{\exp}(\cf^\perp)$,
\end{itemize} 
\index{$\tilde\xi(t)$}
\end{definition}
\begin{remark}
\begin{enumerate}
\item The limit $\lim_{t\to\infty}(\xi(t) - t\dot\xi(\infty))$ exists, because the derivative of the function $\xi(t) -t\dot\xi(\infty)$ satisfies $\dot\xi(t) - \dot\xi(\infty) = \mathrm O((1+t)^{-1-\zeta})$. Hence, the mean value theorem implies that if $t<t'$ there exists a constant $M>0$ such that $|\xi(t) -t\dot\xi(\infty)- (\xi(t') - t'\dot\xi(\infty))| \leq M(1+t)^{-\zeta}$.
\item The second and third conditions mean that we can write any such $\xi$ in the form 
\begin{equation}\label{tildexi}
\xi(t) = (t-b)\dot\xi(\infty)  + \tilde\xi(t),
\end{equation}
where $\tilde \xi$ has a limit in $\gf$ and $\dot{\tilde\xi} \in \Omega_\zeta(\cf)\oplus\Omega_{\exp}(\cf^\perp)$. 
\item Note that the Lie algebra $\mathrm{Lie}(\mathcal G_G)$ defined earlier is contained in $\mathrm{Lie}(\tilde{\mathcal G}_G(b))$ as the subspace defined by the condition $\dot\xi(\infty) = 0$. The real  parameter $b$ is the same as the one that appears in the definition of the Bielawski metric. We will see in the next section, when we calculate moment maps, why the elements in $\mathrm{Lie}(\tilde{\mathcal G}_G(b))$ have to depend on $b$ in this way.
\end{enumerate}
\end{remark}

\begin{lemma} \label{fourth2}
If $\xi$ is in  $\mathrm{Lie}(\tilde{\mathcal G}_{G}(b))$ 
then,
$[\tau+\frac{\sigma}{2(t+1)},\xi(t)^D] \in \Omega_\zeta(\cf)$ and $[\tau+\frac{\sigma}{2(t+1)},\xi(t)^H] \in \Omega_{\exp}(\cf^\perp)$ for all $\tau\in \tf$ and $\sigma\in [\mathfrak c,\mathfrak c].$
\end{lemma}
\begin{proof}
Let $\xi(t) = (t-b)\dot\xi(\infty) + \tilde \xi(t)\in \mathrm{Lie}(\tilde{\mathcal G}_{G}(b))$, where $\xi(0) = 0$ and $\tilde\xi$ has a limit in the subalgebra $\gf\subset\cf$. In particular, $\tilde\xi(\infty) = \tilde\xi(\infty)^D$. Moreover, since $\dot\xi- \dot\xi(\infty) = \dot{\tilde\xi}$, we have  $\dot{\tilde\xi}^D\in\Omega_\zeta(\cf)$  as well as $\dot{\tilde\xi}^H\in\Omega_{\exp}(\cf^\perp)$. 

Since  $\tilde{\xi}(\infty) \in \gf$ and we are assuming
that $G \subset C \cap C(\sigma)$, and since
 $\dot\xi(\infty) \in Z(\cf)$ commutes with $\tau$ and $\sigma$, we have
$$ -[\tau + \frac{\sigma}{2(t+1)}, \xi(t)^D] =
[\tau + \frac{\sigma}{2(t+1)},\tilde\xi(\infty)-\xi(t)^D] = [\tau +\frac{\sigma}{2(t+1)},\tilde\xi(\infty)-\tilde\xi(t)^D],$$

$$\quad [\tau +\frac{\sigma}{2(t+1)},\xi(t)^H] = [\tau + \frac{\sigma}{2(t+1)},\tilde\xi(t)^H].$$
We check that these two terms have the required asymptotic behaviour. By writing 
$$\tilde\xi(\infty) - \tilde\xi(t) = \int_t^\infty\dot{\tilde\xi}(s)ds,$$
and using $\dot{\tilde\xi}^D\in\Omega_\zeta(\cf)$, $\dot{\tilde\xi}^H\in\Omega_{\exp}(\cf^\perp)$, we see that  
$$\tilde\xi(\infty) - \tilde\xi^D = \mathrm O((1+t)^{-\zeta})\quad \text{and}\quad \tilde\xi^H \in \Omega_{\exp}(\cf^\perp).$$ 
Then since $\tau$ and $\dot\xi(\infty)$ act trivially on $\cf$, i.e. commute with $\xi^D$, we find for any $\sigma\in[\cf,\cf]$ 
$$[\tau + \frac{\sigma}{2(t+1)},\tilde\xi(\infty)-\xi(t)^D] = [\frac{\sigma}{2(t+1)},\tilde\xi(\infty)-\tilde\xi(t)^D] = \mathrm O((1+t)^{-(1+\zeta)})$$
and 
\begin{eqnarray*}
\frac{d}{dt}[\tau + \frac{\sigma}{2(t+1)},\tilde\xi(\infty)-\xi(t)^D] &=& \frac{d}{dt} [\frac{\sigma}{2(t+1)},\tilde\xi(\infty)-\tilde\xi(t)^D] \\
&=& -[\frac{\sigma}{2(t+1)^2},\tilde\xi(\infty)-\tilde\xi(t)^D] - [\frac{\sigma}{2(t+1)},\dot{\tilde\xi}(t)^D] \\
&=& \mathrm O((1+t)^{-(2+\zeta)}).
\end{eqnarray*}
For the $\cf^\perp$-component, we calculate similarly
$$[\tau + \frac{\sigma}{2(t+1)},\tilde\xi(t)^H] = O(e^{-\eta t})$$
and 
$$\frac{d}{dt}[\tau + \frac{\sigma}{2(t+1)},\tilde\xi(t)^H] = -[\frac{\sigma}{2(t+1)^2},\tilde\xi(t)^H] + [\tau + \frac{\sigma}{2(t+1)},\dot{\tilde\xi}(t)^H] = \mathrm O(e^{-\eta t})$$
for some $\eta >0$.
\end{proof}

\begin{lemma} 
$\mathrm{Lie}(\tilde{\mathcal G}_{G}(b))$ is a Lie algebra.
\end{lemma}

\begin{proof}
We only have to check that $\mathrm{Lie}(\tilde{\mathcal G}_{G}(b))$ is closed under the Lie bracket.
Let $\xi,\psi \in \mathrm{Lie}(\tilde{\mathcal G}_{G}(b))$. 

Clearly, the first condition holds for $[\xi, \psi]$.

We may write $\xi(t) = (t-b)\dot\xi(\infty) + \tilde \xi(t)$ and $\psi(t) = (t-b)\dot\psi(\infty) + \tilde \psi(t)$ with $\dot\xi(\infty),\dot\psi(\infty) \in Z(\cf)$ and $\dot{\tilde\xi},\dot{\tilde\psi} = \mathrm O((1+t)^{-(1+\zeta)})$. Then using that $Z(\cf)$ acts trivially on $\cf$, we calculate
\begin{eqnarray*}
\frac{d}{dt}[\xi,\psi]  &=& [\dot\xi,\psi] + [\xi,\dot\psi]\\
&=& [\dot \xi(\infty) + \dot{\tilde\xi},(t-b)\dot\psi(\infty) + \tilde\psi] + [(t-b)\dot\xi(\infty) + \tilde\xi,\dot\psi(\infty) + \dot{\tilde\psi}]\\
&=& [\dot\xi(\infty), \tilde\psi^H] + [\dot{\tilde\xi}^H,(t-b)\dot\psi(\infty)]  + [(t-b)\dot\xi(\infty), \dot{\tilde\psi}^H] + [\tilde\xi^H,\dot\psi(\infty)] + [\dot{\tilde\xi}, \tilde\psi] + [\tilde\xi, \dot{\tilde\psi}].
\end{eqnarray*} 
We have seen in the proof of the previous lemma that  $\tilde\xi^H, \tilde\psi^H\in\Omega_{\exp}(\cf^\perp)$, so the first four terms in the last line of the above equation converge to zero as $t\to\infty$. Moreover, $\tilde\psi,\tilde\xi$ are bounded with $\dot{\tilde\xi},\dot{\tilde\psi} = \mathrm O((1+t)^{-(1+\zeta)})$. Thus, the last two terms also converge to $0\in Z(\cf)$. Now 
\begin{eqnarray*}
\lim_{t\to\infty}[\xi(t),\psi(t)] &=& \lim_{t\to\infty}\left( [(t-b)\dot\xi(\infty) + \tilde\xi, (t-b)\dot\psi(\infty) + \tilde\psi] \right)\\
&=& \lim_{t\to\infty}\left((t-b)([\dot\xi(\infty), \tilde\psi^H] + [\tilde\xi^H, \dot\psi(\infty)])\right) + [\tilde\xi(\infty),\tilde\psi(\infty)]\\
&=& [\tilde\xi(\infty),\tilde\psi(\infty)],
\end{eqnarray*}
where we used again the exponential decay of $\tilde\xi^H,\tilde\psi^H$ to deduce that $ \lim_{t\to\infty}(t-b)([\dot\xi(\infty), \tilde\psi^H] + [\tilde\xi^H, \psi(\infty)]) =0$. Note in particular that $[\tilde\xi(\infty),\tilde\psi(\infty)]\in \gf$. Thus, we have verified the second condition for $[\xi,\psi]$.

To check the third condition, we consider again
\begin{eqnarray*}
\frac{d}{dt}[\xi,\psi] &=& [\dot\xi(\infty), \tilde\psi^H] + [\dot{\tilde\xi}^H,(t-b)\dot\psi(\infty)]  + [(t-b)\dot\xi(\infty), \dot{\tilde\psi}^H] + [\tilde\xi^H,\dot\psi(\infty)] + \frac{d}{dt} [\tilde\xi, \tilde\psi].
\end{eqnarray*}
The first four terms lie in $\Omega_{\exp}(\cf^\perp)$, since $\tilde \xi^H,\tilde\psi^H$ and their first derivatives do. So we have to consider the term  $\frac{d}{dt} [\tilde\xi, \tilde\psi]$, which we decompose according to the splitting $\kf=\cf\oplus\cf^\perp$, i.e. into $D$ and $H$-parts:
\begin{eqnarray*}
\frac{d}{dt} [\tilde\xi, \tilde\psi] &=& [\dot{\tilde\xi}^D +\dot{\tilde\xi}^H,\tilde\psi^D+\tilde\psi^H] + [\tilde\xi^D+\tilde\xi^H,\dot{\tilde\psi}^D+\dot{\tilde\psi}^H]\\
&=&[\dot{\tilde\xi}^D,\tilde\psi^D] + [\tilde\xi^D,\dot{\tilde\psi}^D] + ([\dot{\tilde\xi}^H,\tilde\psi^H] + [\tilde\xi^H,\dot{\tilde\psi}^H])^D \\
& & \oplus  [\dot{\tilde\xi}^H,\tilde\psi^D] + [\dot{\tilde\xi}^D,\tilde\psi^H] + [\tilde\xi^D,\dot{\tilde\psi}^H]  + [\tilde\xi^H,\dot{\tilde\psi}^D] + ([\dot{\tilde\xi}^H,\tilde\psi^H] + [\tilde\xi^H,\dot{\tilde\psi}^H])^H.
\end{eqnarray*}
The terms $([\dot{\tilde\xi}^H,\tilde\psi^H] + [\tilde\xi^H,\dot{\tilde\psi}^H])^D$ and $([\dot{\tilde\xi}^H,\tilde\psi^H] + [\tilde\xi^H,\dot{\tilde\psi}^H])^H$ clearly lie in $\Omega_{\exp}(\cf)\subset\Omega_\zeta(\cf)$ and $\Omega_{\exp}(\cf^\perp)$ respectively.  The remaining terms lie in $\Omega_\zeta(\cf)\oplus\Omega_{\exp}(\cf^\perp)$ as is implied by lemma \ref{fg}.

Altogether we have verified that $[\xi,\psi]\in\mathrm{Lie}(\tilde{\mathcal G}_{G}(b))$, which is therefore indeed a Lie algebra.
\end{proof}

The space $\mathrm{Lie}(\tilde{\mathcal G}_{G}(b))$ is the Lie algebra of the gauge group $\tilde{\mathcal G}_{G}(b)$ given by the space of $K$-valued $C^2$-paths
$u:[0,\infty) \to K$ such that 
\begin{itemize}
\item $u(0) = 1,$ 
\item $s(u) := \lim_{t\to\infty}(\dot u(t) u(t)^{-1})$ exists and lies in $Z(\cf)$ and $\tilde u(\infty) := \exp(bs(u))(\lim_{t\to\infty}\exp(-ts(u))u(t)) \in G,$
\item $(\dot u u^{-1})^D-s(u) \in \Omega_\zeta(\cf)$ and $(\dot u u^{-1})^H \in \Omega_{\exp}(\cf^\perp),$
\end{itemize}
In the next lemma, we verify that it in fact does act on the space $\tilde{\mathcal A}_{C,\sigma}$, i.e. it preserves the asymptotic conditions imposed on our Nahm data.

\begin{lemma}\label{groupaction2}
The action of the gauge group $\tilde{\mathcal G}_{G}(b)$ preserves the space $(\tilde{\mathcal A}_{C,\sigma},\Vert\cdot\Vert_{B,b})$.
\end{lemma}
\begin{proof}
This is similar to the proof of the analogous statement in the previous section, but we have to take care of the different asymptotics. 
Let $T \in \tilde{\mathcal A}_{C,\sigma}$, $\xi\in\mathrm{Lie}(\tilde{\mathcal G}_G(b))$ and consider a fundamental vector field
$$X^\xi_T = ([\xi,T_0] -\dot \xi, [\xi,T_1],[\xi,T_2],[\xi,T_3]).$$
Writing $T = T_{\tau_0,\tau,\sigma} + \tilde T$ with $\tilde T\in (\Omega_\zeta(\cf)\oplus \Omega_{\exp}(\cf^\perp))^4$, we have
\begin{eqnarray*}
(X^\xi_T)_0 &=& [\xi(t), \tau_0] + [\xi(t),\tilde T_0(t)] -\dot\xi(\infty) -(\dot\xi(t) - \dot\xi(\infty))  \\
(X^\xi_T)_i &=&[\xi(t), \tau_i+\frac{\sigma_i}{2(t+1)}] + [\xi(t),\tilde T_i(t)], \quad i=1,2,3
\end{eqnarray*}
Lemma \ref{fourth2} shows that the first term for $i=1,2,3$ lies in
$\Omega_\zeta(\cf) \oplus\Omega_{\exp}(\cf^\perp)$, as required.
For $i=0$ the situation is simpler and we find using $[\tau_0,\xi(t)^D] = 0$ that
$$[\tau_0, \xi(t)] =  [\tau_0,  \xi(t)^H],$$  and applying Lemma \ref{fourth2} with $\sigma =0$, we get the desired result for the $[\xi(t), \tau_0]$
term. 

Similarly, decomposing $\dot\xi-\dot\xi(\infty) = \dot{\tilde\xi}$
into $D$ and $H$ parts, and using the third condition in the
definition of the gauge group, we see that this also satisfies the
correct decay conditions for a tangent vector to $\tilde{\mathcal
  A}_{C,\sigma}$. Finally $\dot{\xi}(\infty) \in Z(\cf)$ is a tangent
vector corresponding to deforming $\tau_0$.

Now $\tilde T_i \in \Omega_\zeta(\cf)\oplus\Omega_{\exp}(\cf^\perp)$ so that for $i=0,1,2,3$ we can write 
\begin{eqnarray*}
[\xi,\tilde T_i] &=& [\xi^D+\xi^H,\tilde T_i^D+\tilde T_i^H]\\
&=& ([\xi^D,\tilde T_i^D]+[\xi^H,\tilde T_i^H]^D) \oplus ([\xi^D,\tilde T_i^H]+[\xi^H,\tilde T_i^D] + [\xi^H,\tilde T_i^H]^H).
\end{eqnarray*}
Write as usual $\xi(t) = (t-b)\dot\xi(\infty) + \tilde\xi(t)$. Observing that $\xi^D(t) = (t-b)\dot\xi(\infty) + \tilde\xi^D(t)$, $\xi^H(t) = \tilde\xi^H(t)$ since $\dot\xi(\infty) \in Z(\cf)\subset \cf$, and that $\dot\xi(\infty)$ acts trivially on $D$-parts, we can refine this to
$$([\tilde\xi^D,\tilde T_i^D]+[\tilde\xi^H,\tilde T_i^H]^D) \oplus ([(t-b)\dot\xi(\infty) +\tilde\xi^D,\tilde T_i^H]+[\tilde\xi^H,\tilde T_i^D] + [\xi^H,\tilde T_i^H]^H).$$
Now since $\tilde\xi$ and the $\tilde T_i$ are bounded, all terms lie in the appropriate function spaces to be tangent to $\tilde{\mathcal A}_{C,\sigma}$ by lemma \ref{fg}. The only term not taken care of by this lemma is 
$$[t\dot\xi(\infty), \tilde T_i^H].$$
However, $\tilde T_i^H\in \Omega_{\exp}(\cf^\perp)$ and this implies easily that also $[t\dot\xi(\infty), \tilde T_i^H]\in \Omega_{\exp}(\cf^\perp)$.

Thus, the fundamental vector fields of the action are tangent to $\tilde{\mathcal A}_{C,\sigma}$. To see that this infinitesimal action integrates again to an action of the group $\tilde{\mathcal G_{G}}(b)$, we can essentially follow the same line of argument as in the proof of the corresponding statement in Lemma \ref{groupaction}. The only additional observation we need is the following: The asymptotic conditions defining $\tilde{\mathcal A}_{C,\sigma}$ are preserved by gauge transformations $u$ such that there exists a $\xi_0\in Z(\cf)$ so that for large $t$ we have $u(t) = \exp((t-b)\xi_0)$. Indeed, we can verify this directly: For large $t$, $\dot u(t)u^{-1}(t)$ is equal to $s(u) = \xi_0\in Z(\cf)$, so that $u.T_0(\infty) = \tau_0+\xi_0$, which is allowed. 

Since $u$ is $Z(C)$-valued $\mathrm{Ad}(u)$ acts trivially on $\tau+\frac{\sigma}{2(t+1)}$ and also on $\Omega_\zeta(\cf)$. To see that it preserves  $\Omega_{\exp}(\cf^\perp)$, we first use the bi-invariance of the norm on $\kf$ to see that 
$$|\mathrm{Ad}(u)(\tilde T_i)^H| = |\mathrm{Ad}(u)(\tilde T_i^H)| = |\tilde T_i^H| = \mathrm O(e^{-\eta t}),$$ 
whenever $\tilde T_i^H = \mathrm O(e^{-\eta t})$. Next for $t$ large enough so that $u(t) = \exp((t-b)\xi_0)$, consider the derivative
$$\left|\frac{d}{dt} (\mathrm{Ad}(u)(\tilde T_i))^H\right| \leq  |[\xi_0, \mathrm{Ad}(u)(\tilde T_i^H)]| +  |\mathrm{Ad}(u)(\dot{\tilde T_i}^H)| \leq 2|\xi_0||\tilde T_i^H| +  |\dot{\tilde T_i}^H| = \mathrm O(e^{-\eta t})$$
which again lies in the correct function spaces.

Now, for large $t$, we can write any $u\in\tilde{\mathcal G}_{G}(b)$ as the product of a $Z(C)$-valued gauge transformation $\exp((t-b)\xi_0)$ and a gauge transformation that lies in $\mathcal G_G$ (as defined in $\S 3$). The proof of  Lemma \ref{groupaction} shows that such gauge transformations preserve not just the space $\mathcal A_C$, but also the slightly different space $\tilde{\mathcal A}_{C,\sigma}$ in which $\sigma$ is fixed and $\tau_0\neq 0$ is allowed (recall that we assume in the definition of $\tilde{\mathcal G}_{G}(b)$ that $G\subset C\cap C(\sigma)$, so that $\tilde u(\infty)$ acts now trivially on $\sigma$).
\end{proof}

\subsection{Moment Maps}
In this section we calculate the hyperk\"ahler moment maps associated with the action of the gauge groups $\tilde{\mathcal G}_{G}(b)$ for various choices of subgroup $G\subset C\cap C(\sigma)$. The calculations also explain the dependence of $\tilde{\mathcal G}_{G}(b)$ on the parameter $b$ appearing in the Bielawski metric.  First we need a technical lemma which explains why the moment maps we obtain can be thought of as taking values in the dual of $\mathrm{Lie}(\tilde{\mathcal G}_{G}(b))$. 

For this it is useful to decompose  $\kf$ into orthogonal subspaces as follows 
$$\kf = Z(\cf) \oplus Z(\cf)^\perp,$$
where we observe $Z(\cf)^\perp = [\cf,\cf] \oplus \cf^\perp$ (cf. lemma \ref{algrel}). 
We write an element $\xi$ as $\xi^0+\xi^1$ 
\index{$\xi^0,\xi^1$} 
with respect to this splitting. 

\begin{lemma}\label{momentdual}
Let $\mu:[0,\infty) \to \kf$ be such that there exists a $\zeta>0$ with $\mu^0 = \mathrm O((1+t)^{-2-\zeta})$, $\mu^1 = \mathrm O((1+t)^{-2})$. Then $\mu$ determines an element of the dual of $\mathrm{Lie}(\tilde{\mathcal G}_{G}(b))$ via 
$$\mu(\xi) = \int_0^\infty \langle \xi(t),\mu(t)\rangle dt$$
\end{lemma}
\begin{proof}
Observe that by the definition of $\mathrm{Lie}(\tilde{\mathcal G}_{G}(b))$ we have 
$$\xi^0(t) = (t-b)\dot\xi(\infty) +\tilde\xi^0(t),\quad \xi^1(t) = \tilde\xi^1(t),$$
 where $\tilde \xi^i$ is bounded for $i=0,1$. Now
\begin{eqnarray*}
\int_0^\infty \langle \mu(t), \xi(t) \rangle dt &=& \int_0^\infty \langle \mu^0(t), (t-b)\dot\xi(\infty) +\tilde\xi^0(t) \rangle + \langle \mu^1(t), \tilde\xi^1(t)\rangle dt\\
&=& \int_0^\infty \mathrm O((1+t)^{-1-\zeta}) + \mathrm O((1+t)^{-2})  dt <\infty.
\end{eqnarray*}
To show that the association $\mu\mapsto \int_0^{\infty} \langle \mu(t), -\rangle dt$ is injective, we can use an argument analogous to the one we used to show the non-degeneracy of the Bielawski metric on $\mathcal A_C$. We just note that $\mathrm{Lie}(\tilde{\mathcal G}_{G}(b))$ contains the space of $\kf$-valued functions on $[0,\infty)$ which have compact support and vanish at $t=0$. 
\end{proof}

\subsubsection{The Moment Map for the Action of the Gauge Group associated with $G=[C,C]\cap C(\sigma)$}
\begin{proposition}\label{propmomentmaps}
The gauge group $\tilde{\mathcal G}_{[C,C]\cap C(\sigma)}(b)$ acts on the flat hyperk\"ahler manifold \\$(\tilde{\mathcal A}_{C,\sigma}, \Vert\cdot\Vert_{B,b}, I,J,K)$ in a tri-Hamiltonian fashion with hyperk\"ahler moment map \\ $\mu = (\mu_I,\mu_J,\mu_K): \tilde{\mathcal A}_{C,\sigma} \to \mathrm{Lie}(\tilde{\mathcal G}_{[C,C]\cap C(\sigma)}(b))^*$ given by
\begin{eqnarray*}
-\mu_I(T) &=& \dot T_1 + [T_0,T_1] - [T_2,T_3], \\
-\mu_J(T) &=& \dot T_2 + [T_0,T_2] - [T_3,T_1], \\
-\mu_K(T) &=& \dot T_3 + [T_0,T_3] - [T_1,T_2]. \\
\end{eqnarray*}
\end{proposition}
\begin{remark}
Observe that by definition we have $T_0 = \tau_0+\tilde T_0$ and for $i=1,2,3$,  $T_i = \tau_i+\frac{\sigma_i}{2(t+1)} + \tilde T_i$ with $\sigma_i\in[\cf,\cf]$ and $\tilde T_i \in \Omega_\zeta(\cf)\oplus\Omega_{\exp}(\cf^\perp)$ for any $T\in\tilde{\mathcal A}_{C,\sigma}$. Thus, 
$$T_i^0 = \tau_i + \tilde T_i^0,\quad i=0,1,2,3\qquad  : \qquad
T_0^1 = \tilde T_0^1\quad : \quad T_i^1 = \frac{\sigma_i}{2(t+1)} + \tilde T_i^1,
\quad i=1,2,3.$$
This implies that 
$$\dot T_0 = \dot{\tilde{T_0}} = \mathrm O((1+t)^{-(2+\zeta)})$$
and
$$\dot T_i^0 = \dot{\tilde T_i}^0 = \mathrm O((1+t)^{-(2+\zeta)}), \quad \dot T_i^1 = -\frac{\sigma_i}{2(t+1)^2} + \dot{\tilde T_i}^1 = \mathrm O((1+t)^{-2}).$$ 
Further, we see 
$$[T_0,T_i] = \mathrm O((1+t)^{-(2+\zeta)}$$
and for any $j,k\in\{1,2,3\}$, since the $\tau_i$ mutually commute and $[\tau_j,\sigma_k]=0$, we have
\begin{eqnarray*}
[T_j,T_k] &=& [\tau_j+\frac{\sigma_j}{2(t+1)} + \tilde T_j, \tau_k+\frac{\sigma_k}{2(t+1)} + \tilde T_k] \\
&=& \frac{[\sigma_j,\sigma_k]}{4(t+1)^2} + \mathrm O((1+t)^{-(2+\zeta)}).
\end{eqnarray*}
The first term lies in $Z(\cf)^\perp$. Thus, for any $T\in \tilde{\mathcal A}_{C,\sigma}$ and any $i=I,J,K$, $\mu_i(T)$ satisfies the assumption of lemma \ref{momentdual} and so defines an element of $\mathrm{Lie}(\tilde{\mathcal G}_{[C,C]\cap C(\sigma)}(b))^*$. 
\end{remark}

\begin{proof}[Proof of the proposition]
The gauge group clearly preserves the hyperk\"ahler structure, since the induced action on $T\tilde{\mathcal A}_{C,\sigma}$ is given by 
$$u.(X_i)_{i=0}^3 = (\mathrm{Ad}(u)(X_i))_{i=0}^3,$$
which preserves the bi-invariant inner product and commutes with the complex structures. 

Observe that for $\xi\in\mathrm{Lie}(\tilde{\mathcal G}_{[C,C]\cap C\sigma)}(b))$ we have 
$$[\xi,T_i] = [(t-b)\dot\xi(\infty) + \tilde\xi, T_i^D+T_i^H] = [(t-b)\dot\xi(\infty) + \tilde\xi, T_i^H] + [\tilde\xi,T_i^D].$$
As $t\to\infty$ the first term converges to zero due to the exponential decay of $T_i^H$, while the second one  converges to 
$$[\tilde\xi(\infty), \tau_i] = 0,$$
since $\tilde\xi(\infty)\in [\cf,\cf]\subset \cf$. Altogether, we see 
$$X^\xi_T(\infty) =\lim_{t\to\infty}(-\dot\xi +[\xi,T_0], ([\xi,T_i])_{i=1}^3) = (-\dot\xi(\infty),0,0,0).$$

We now calculate the moment map $\mu_I$ associated with the complex structure $I$. It is characterised by the relation $d\mu_I(\xi)(X) = \omega_I(X^\xi, X)$ where $X$ is any tangent vector to $\tilde{\mathcal A}_{C,\sigma}$.
Let $\xi(t) = (t-b)\dot\xi(\infty) + \tilde \xi(t) \in \mathrm{Lie}(\tilde{\mathcal G}_{[C,C]\cap C(\sigma)}(b))$ and let $X = X_\delta + \tilde X$ be a tangent vector. As usual we will write $\delta_i = X_i(\infty)$. We calculate using $IX^\xi_T(\infty) = (0,\dot\xi(\infty),0,0)$

\begin{eqnarray*}
\omega_I(X^\xi_T,X) &=& \langle IX^\xi_T,X\rangle_{B,b}\\
&=& \langle([\xi,T_1], -[\xi,T_0] +\dot \xi, -[\xi,T_3],[\xi,T_2]),(X_0,X_1,X_2,X_3) \rangle_{B,b}\\
&=& \int_0^\infty(\langle[\xi(t),T_1(t)],X_0(t)\rangle + \langle \dot\xi(t) - [\xi(t),T_0(t)],X_1(t)\rangle -\langle \dot\xi(\infty),\delta_1\rangle  -\langle [\xi(t),T_3(t)],X_2(t)\rangle \\
& &  +\langle [\xi(t),T_2(t)],X_3(t)\rangle )dt +b\sum_{i=0}^3\langle (IX_T^\xi)_i(\infty), X_i(\infty)\rangle \\
&=& \int_0^\infty(\langle[\xi(t),T_1(t)],X_0(t)\rangle  - \langle [\xi(t),T_0(t)],X_1(t)\rangle -\langle [\xi(t),T_3(t)],X_2(t)\rangle +\langle [\xi(t),T_2(t)],X_3(t)\rangle \\
& &  + \frac{d}{dt}(\langle\xi(t),X_1(t)\rangle -\langle t\dot \xi(\infty),\delta_1\rangle) - \langle\xi(t),\dot X_1(t)\rangle) dt +  b\langle\dot\xi(\infty), \delta_1\rangle \\
&=& \int_0^\infty \langle \xi(t), -\dot X_1-[T_0(t),X_1(t)] -[X_0(t),T_1(t)] +[T_2(t),X_3(t)]+[X_2(t),T_3(t)]\rangle dt\\
& & +(\langle\xi(t),X_1(t)\rangle- \langle t\dot\xi(\infty), \delta_1\rangle)|_0^\infty  + b\langle\dot\xi(\infty), \delta_1\rangle
\end{eqnarray*}
We show that 
$$(\langle\xi(t),X_1(t)\rangle- \langle t\dot\xi(\infty), \delta_1\rangle)|_0^\infty = -b\langle\dot\xi(\infty), \delta_1\rangle.$$
Since the term on the lefthand side vanishes at $t=0$, we only have to deal with its limit as $t\to\infty$. Using $\xi(t) = -b\dot\xi(\infty) + t\dot\xi(\infty) +\tilde\xi$ we may rewrite this as follows:
$$\lim_{t\to\infty}(\langle\xi(t),X_1(t)\rangle- \langle t\dot\xi(\infty), \delta_1\rangle) = \lim_{t\to\infty}(\langle t\dot\xi(\infty), X_1-\delta_1\rangle) + \langle \tilde\xi(\infty),\delta_1  \rangle -b\langle\dot\xi(\infty),\delta_1\rangle.$$
Since 
$$X_1-\delta_1 = \mathrm O((1+t)^{-(1+\zeta)}),$$ the first term satisfies 
$$ \langle t\dot\xi(\infty), X_1-\delta_1\rangle = O((1+t)^{-\zeta}),$$ 
and thus converges to zero. The second term 
$$\langle \tilde\xi(\infty),\delta_1\rangle$$ 
vanishes by lemma \ref{algrel}, as $\tilde\xi(\infty) \in [\cf,\cf]$. Thus, 
$$b\langle\dot\xi(\infty), \delta_1\rangle +(\langle\xi(t),X_1(t)\rangle- \langle t\dot\xi(\infty), \delta_1\rangle)|_0^\infty = -b\langle\dot\xi(\infty), \delta_1\rangle+b\langle\dot\xi(\infty), \delta_1\rangle =0.$$
Hence,
$$\omega_I(X^\xi_T,X) = \int_0^\infty \langle \xi(t), -\dot X_1-[T_0(t),X_1(t)] -[X_0(t),T_1(t)] +[T_2(t),X_3(t)]+[X_2(t),T_3(t)]\rangle dt,$$
and a calculation analogous to the one in the proof of lemma \ref{momentdual} shows that this integral is well-defined.

In other words 
$$d\mu_I(T) (X) = \frac{d}{d\theta}|_{\theta = 0} \mu_I(T +\theta X) = -\dot X_1-[T_0(t),X_1(t)] -[X_0(t),T_1(t)] +[T_2(t),X_3(t)]+[X_2(t),T_3(t)],$$
from which the proposition follows.
\end{proof}

\begin{remark}
Note that in the above proof the assumption that $\tilde\xi(\infty) \in [\mathfrak c,\mathfrak c]$ is only needed to ensure that the terms $\langle\tilde\xi(\infty),\delta_1\rangle$ vanish. For the terms of the form $[\tilde\xi(\infty),\tau_i]$ to vanish we only need to assume $\tilde\xi(\infty) \in\mathfrak c$.

Note also that for the calculation $X^\xi_T(\infty) =\lim_{t\to\infty}(-\dot\xi +[\xi,T_0], ([\xi,T_i])_{i=1}^3) = (-\dot\xi(\infty),0,0,0)$ we only need to assume that $\tilde\xi(\infty) \in \cf$. 
\end{remark}

The object we would like to study is the hyperk\"ahler quotient of $(\tilde{\mathcal A}_{C,\sigma}, \Vert\cdot\Vert_{B,b})$ by the action of the gauge group $\tilde{\mathcal G}_{[C,C]\cap C(\sigma)}(b)$, which we denote by $\tilde{\mathcal Q}_{C,\sigma}(b)$. 
\index{$\tilde{\mathcal Q}_{C,\sigma}(b)$} 

Formally, the tangent space of $\tilde{\mathcal Q}_{C,\sigma}(b)$ at a smooth point $T$ can be identified with the orthogonal complement to the tangent space of the $\tilde{\mathcal G}_{[C,C]\cap C(\sigma)}(b)$-orbit through $T$ inside the tangent space of all solutions to the Nahm equations in $\tilde{\mathcal A}_{C,\sigma}$. That is, $T_T\tilde{\mathcal Q}_{C,\sigma}(b)$ can be described by elements $X\in T_T\tilde{\mathcal A}_{C,\sigma}$ solving the following system of equations
\begin{eqnarray*}
\dot X_1 + [T_0,X_1] + [X_0,T_1] -[T_2,X_3] - [X_2,T_3] &=& 0\\
\dot X_2 + [T_0,X_2] + [X_0,T_2]- [T_3,X_1] - [X_3,T_1] &=& 0\\
\dot X_3 + [T_0,X_3] + [X_0,T_3] - [T_1,X_2] - [X_1,T_2] &=& 0\\
\\
\dot X_0 +\sum_{i=0}^3 [T_i,X_i]&=& 0.\\
\end{eqnarray*}
The first three equations are just the linearised Nahm equations, while the equation at the bottom says that $X$ is perpendicular to the $\tilde{\mathcal G}_{[C,C]\cap C(\sigma)}(b)$-orbit through $T$. This can be seen as follows. Let $\xi\in\mathrm{Lie}(\tilde{\mathcal G}_{[C,C]\cap C(\sigma)}(b))$. Then we have seen in equation \ref{fundvectorfields} that 
$$X^\xi_T = (-\dot\xi + [\xi,T_0], ([\xi,T_i])_{i=1}^3),$$
and 
$$X^\xi_T(\infty) = (-\dot\xi(\infty), 0,0,0).$$
Now let $X = X_{\delta} + \tilde X\in T_T\tilde{\mathcal A}_{C,\sigma}$ and compute using integration by parts:
\begin{eqnarray*}
\langle X^\xi_T,X\rangle_{B,b} &=& b\langle -\dot\xi(\infty), \delta_0\rangle + \int_0^\infty \left(\langle-\dot\xi(t), X_0(t)\rangle -\langle-\dot\xi(\infty), \delta_0 \rangle+ \sum_{i=0}^3\langle[\xi(t),T_i(t)], X_i(t)\rangle\right) dt\\
&=& -b\langle\dot\xi(\infty), \delta_0\rangle + (-\langle\xi(t),X_0(t)\rangle + \langle t\dot\xi(\infty), \delta_0\rangle)|_0^\infty + \int_0^\infty\langle \dot X_0(t) +\sum_{i=0}^3[T_i(t),X_i(t)], \xi\rangle dt.
\end{eqnarray*}
Writing $\xi(t) = t\dot\xi(\infty) -b\dot\xi(\infty)+\tilde\xi$ as usual, and using $\xi(0) = 0$, a calculation almost identical to the one in the proof of Proposition \ref{propmomentmaps} shows 
$$ -b\langle\dot\xi(\infty), \delta_0\rangle + (-\langle\xi,X_0\rangle + \langle t\dot\xi(\infty), \delta_0\rangle)|_0^\infty  =0.$$
Thus, 
$$\langle X^\xi_T,X\rangle_{B,b} = \int_0^\infty\langle \dot X_0 +\sum_{i=0}^3[T_i,X_i], \xi\rangle dt,$$
as desired.

\subsubsection{The Torus Action on the Moduli Space $\tilde{\mathcal Q}_{C,\sigma}(b)$}

On $\tilde{\mathcal Q}_{C,\sigma}(b)$ we have an action of the gauge group $\tilde{\mathcal G}_{C\cap C(\sigma)}(b)$, which preserves the hyperk\"ahler structure.

Of course, $\tilde{\mathcal G}_{C\cap C(\sigma)}(b)$ preserves the hyperk\"ahler structure on $\tilde{\mathcal A}_{C,\sigma}$ and since $\tilde{\mathcal G}_{[C,C]\cap C(\sigma)}(b)$ is a normal subgroup of $\tilde{\mathcal G}_C(b)$ (it is the kernel of the map $\tilde{\mathcal G}_{C\cap C(\sigma)}(b) \to C\cap C(\sigma)/(C(\sigma)\cap[C,C])$ induced by evaluation of $\tilde u$ at infinity, see the notation used in definition of $\tilde{\mathcal G}_G(b)$, Definition \ref{tildeG}), the action of $\tilde{\mathcal G}_{C(\sigma)}(b)$ induces an action of $C\cap C(\sigma)/([C,C]\cap C(\sigma))$ on the moduli space $\tilde{\mathcal Q}_{C,\sigma}(b)$. From the proof of \ref{propmomentmaps} we can read off the moment map of the $C\cap C(\sigma)/([C,C]\cap C(\sigma))$-action.

\begin{corollary}\label{Tmoment}
The moment map for the $C\cap C(\sigma)/([C,C]\cap C(\sigma))$-action on $\tilde{\mathcal Q}_{C,\sigma}(b)$ is evaluation at infinity: 
\begin{eqnarray*}
\mu_I(T) &=& T_1(\infty), \\
\mu_J(T) &=& T_2(\infty), \\
\mu_K(T) &=& T_3(\infty). \\
\end{eqnarray*}
\end{corollary}
\begin{proof}
We give the proof just for $\mu_I$. Most of the work has already been done in the proof of Proposition \ref{propmomentmaps}. Let $T\in \tilde{\mathcal Q}_{C,\sigma}(b)$ and let $X = X_{\delta} + \tilde X \in T_T\tilde{\mathcal Q}_{C,\sigma}(b)$. Recall that we found 
\begin{eqnarray*}
\omega_I(X^\xi_T,X) &=& \langle IX^\xi_T,X\rangle\\
&=& \int_0^\infty \langle \xi(t), -\dot X_1-[T_0(t),X_1(t)] -[X_0(t),T_1(t)] +[T_2(t),X_3(t)]+[X_2(t),T_3(t)]\rangle dt\\
& & +(\langle\xi(t),X_1(t)\rangle- \langle t\dot\xi(\infty), \delta_1\rangle)|_0^\infty  + b\langle\dot\xi(\infty), \delta_1\rangle
\end{eqnarray*}
Since $X$ solves the linearised Nahm equations, the integral in the second line of the above calculation vanishes. We have also seen in the proof of \ref{propmomentmaps} that
$$(\langle\xi(t),X_1(t)\rangle- \langle t\dot\xi(\infty), \delta_1\rangle)|_0^\infty  + b\langle\dot\xi(\infty), \delta_1\rangle = \langle \tilde\xi(\infty),\delta_1\rangle.$$
As we now allow $\tilde\xi(\infty)\in\cf(\tau,\sigma)$ rather than $[\cf,\cf]\cap\cf(\tau,\sigma)$,  we end up with
$$d\mu_T(X)(\xi)= \omega_I(X^\xi_T,X) =  \langle\tilde\xi(\infty),X_1(\infty)\rangle,$$
as desired.
\end{proof}

As remarked earlier, the centraliser subgroup $C$ is connected since we assume that $K$ is simply-connected.
Hence $C$ is isomorphic to a commuting local direct product $Z_0(C)[C,C]$ where
$Z_0(C)$ is the identity component of $Z(C)$ 
This is only a local direct product, as in general $Z_0(C)$ and $[C,C]$ will have a non-trivial finite intersection. Now $T$
is a maximal torus in $K$, hence it is also maximal in $C$, and so it
contains $Z_0(C)$. Let $t = \exp(\tau) \in T$. We
can write any $\tau = \tau^0 + \tau^1$ with $\tau^0\in Z(\cf)$ and $\tau^1\in \tf\cap[\cf,\cf]$ so that $t$ is the  product $t = t_0 t_1$ with $t_0 = \exp(\tau^0)\in
Z_0(C)$ and $t_1=\exp(\tau^1)\in [C,C]$. In particular $t_0$ lies in $C\cap
C(\sigma)$. So we can make $T$ act via taking the $Z_0$-component and
projecting this to $C\cap C(\sigma)/([C,C]\cap C(\sigma))$. Due to the potentially non-trivial intersection of $Z_0(C)$ and $C_{ss}$ the element $t_0$ is not unambiguously defined in general. Observe however that the ambiguity lies in $[C,C]$ and so that all  choices of $t_0$ project to the same element in the quotient. In the
case $\sigma =0$ this is equivalent to the usual map $T\to C \to
C/[C,C]$.

We thus obtain an (in general non-effective) action of the 
maximal torus $T = \exp{\mathfrak t}$ on $\tilde{\mathcal Q}_{C,\sigma}(b)$,
the moment map of which is again evaluation at $\infty$. Note that the
metric on the hyperk\"ahler quotient
$\tilde{\mathcal Q}_{C,\sigma}(b) \hkq T$ induced by
$\Vert\cdot\Vert_{B,b}$ is just the $L^2$-metric, as the level set of the
hyperk\"ahler moment map associated with the $T$-action consists of
solutions to the Nahm equations with $\tau$ fixed. We shall then prove
in $\S 5$:
\begin{theorem} \label{thm4.11}
Suppose $\tau$ is Biquard-regular, i.e. $C=C(\tau_2,\tau_3)$. Then the hyperk\"ahler quotient  $\tilde{\mathcal Q}_{C,\sigma}(b) \hkq_\tau T$ gives a stratum of the Kostant variety defined by $\tau$, i.e. a coadjoint orbit with semi-simple part given by $\tau$ and nilpotent part defined by $\sigma$.
\end{theorem}

\begin{remark} \label{comparison2}
In Remark \ref{comparison} we will consider a modified version of ${\mathcal Q}_{C,\sigma}$ where the $\SU(2)$ triple $\sigma$ is not fixed but instead varies in a $T$-orbit of $\SU(2)$-triples. Theorem \ref{thm4.11} will still hold for this modification with $T$ acting in the natural way.
\end{remark}

\subsubsection{The Action of $K$ on $\tilde{\mathcal Q}_{C,\sigma}(b)$}
We have an action of $K$ on each stratum $\tilde{\mathcal Q}_{C,\sigma}(b)$ given by gauge transformations which are no longer constrained to be the identity at $t=0$ but otherwise satisfy the same asymptotic conditions as elements of $\tilde{\mathcal G}_{[C,C]\cap C(\sigma)}(b)$. This action clearly preserves the hyperk\"ahler structure and we can calculate the associated hyperk\"ahler moment map. 

\begin{lemma}
The hyperk\"ahler moment map $\mu$ for the $K$-action on $\tilde{\mathcal Q}_{C,\sigma}(b)$ is given by  
$$\mu(T_0,T_1,T_2,T_3) = (-T_1(0),-T_2(0),-T_3(0)).$$
\end{lemma}
\begin{proof}
Again, we give the proof just for $\mu_I$. The argument is analogous to the one in the proof of Proposition \ref{propmomentmaps}. Let $T\in \tilde{\mathcal Q}_{C,\sigma}(b)$ and let $X = X_{\delta} + \tilde X \in T_T\tilde{\mathcal Q}_{C,\sigma}(b)$. Recall that we found in the proof of Proposition \ref{propmomentmaps}
\begin{eqnarray*}
\omega_I(X^\xi_T,X) &=& \langle IX^\xi_T,X\rangle_B\\
&=& \int_0^\infty \langle \xi(t), -\dot X_1-[T_0(t),X_1(t)] -[X_0(t),T_1(t)] +[T_2(t),X_3(t)]+[X_2(t),T_3(t)]\rangle dt\\
& & +(\langle\xi(t),X_1(t)\rangle- \langle t\dot\xi(\infty), \delta_1\rangle)|_0^\infty  + b\langle\dot\xi(\infty), \delta_1\rangle\\
\end{eqnarray*}
Since $X$ solves the linearised Nahm equations, we end up with
$$ d\mu_T(X)(\xi)=\omega_I(X^\xi_T,X) = (\langle\xi(t),X_1(t)\rangle- \langle t\dot\xi(\infty), \delta_1\rangle)|_0^\infty  + b\langle\dot\xi(\infty), \delta_1\rangle.$$
Now since $\tilde\xi(\infty)\in[\cf,\cf]$, but we no longer assume $\xi(0) = 0$,  the calculation in the proof of Proposition \ref{propmomentmaps} shows that 
$$\omega_I(X^\xi_T,X) = -\langle\xi(0),X_1(0)\rangle.$$
This implies the desired result.
\end{proof}

\subsection{Equivalence of $\tilde{\mathcal Q}_{C,\sigma}(b)$ and $\mathcal Q_{C,\sigma}(b)$}

We recall that $\mathcal Q_{C,\sigma}(b)$ is the quotient by
 ${\mathcal G}_{[C,C]\cap C(\sigma)}$ of the set of solutions to Nahm's equations in $\mathcal A_{C}$ with fixed
$\lie{su}(2)$-triple $\sigma$, with Bielawski metric $\Vert\cdot\Vert_{B,b}$.
 We shall now produce an equivalence between this space and the hyperk\"ahler
quotient $\tilde{\mathcal Q}_{C,\sigma}(b)$.

The starting point is the observation, that given $T\in\tilde{\mathcal A}_{C,\sigma}$, we can always find a gauge transformation $u \in \tilde{\mathcal G}_{[C,C]\cap C(\sigma)}(b)$ such that $u.T$ is asymptotic to some model solution $T_{\tau,\sigma}$, i.e one with $\tau_0 =0$.

\begin{lemma}
Let $T = T_{\tau_0, \tau,\sigma} + \tilde T \in\tilde{\mathcal A}_{C,\sigma}$. Then there exists a gauge transformation $u \in \tilde{\mathcal G}_{[C,C]\cap C(\sigma)}(b)$ such that $u.T$ is asymptotic to $T_{\tau,\sigma} = T_{0,\tau,\sigma}$. 
\end{lemma}
\begin{proof}
Let $T = (T_0,T_1,T_2,T_3) = T_{\tau_0,\tau,\sigma} + \tilde T\in\tilde{\mathcal A}_{C,\sigma}$. Then we know that we can write $T_0 = \tau_0 + \tilde T_0$ with $\tau_0 \in Z(\cf)$  and $\tilde T_i \in \Omega_\zeta(\cf)\oplus\Omega_{\exp}(\cf^\perp)$ for $i=0,1,2,3$. To find the required gauge transformation $u\in\tilde{\mathcal G}_{[C,C]\cap C(\sigma)}(b)$ that gauges $\tau_0$ to zero, we make an ansatz in the form 
$$u(t) = \exp(\xi(t)),$$
for a suitable $\xi\in \mathrm{Lie}(\tilde{\mathcal G}_{[C,C]\cap C(\sigma)}(b))$, which we take to be $\xi(t) = (t-b+ be^{-ct})\tau_0$ for some $c>0$. We check that this defines an element in $\mathrm{Lie}(\tilde{\mathcal G}_{[C,C]\cap C(\sigma)}(b))$. First of all, $\xi(0) = 0$ and $\dot\xi(\infty) = \tau_0\in Z(\cf)$. Now 
$$\tilde\xi(\infty) = b\dot\xi(\infty) + \lim_{t\to\infty}(\xi(t)-t\dot\xi(\infty)) = b\tau_0 + \lim_{t\to\infty}((-b+be^{-ct})\tau_0) = 0.$$  
Moreover, 
$$\dot\xi^H \equiv 0\qquad \dot\xi-\dot\xi(\infty) = -bc e^{-ct}\tau_0 = \mathrm O(e^{-ct}).$$
Thus, $\xi\in\mathrm{Lie}(\tilde{\mathcal G}_{[C,C]\cap C(\sigma)}(b))$, in fact it lies in $\mathrm{Lie}(\tilde{\mathcal G}_1(b))$, as $\tilde\xi(\infty) =0$. 

Now, noting that $u$ is $Z(C)$-valued and so commutes with the $\tau_i$ and $\sigma_i$, we calculate
\begin{eqnarray*}
u.T&=& \left(\tau_0 +\mathrm{Ad}(u(t))(\tilde T_0) - (1-cbe^{-ct})\tau_0, \left(\tau_i + \frac{\sigma_i}{2(t+1)} + \mathrm{Ad}(u(t))(\tilde T_i)\right)_{i=1}^3\right) \\
&=& T_{\tau,\sigma} + (cbe^{-ct}\tau_0 + \mathrm{Ad}(u)\tilde T_0, (\mathrm{Ad}(u)\tilde T_0)_{i=1}^3),
\end{eqnarray*}
as desired.
\end{proof}

\begin{corollary}
The moduli space $\mathcal Q_{C,\sigma}(b)= (\mathcal N_{C,\sigma}, \Vert\cdot\Vert_{B,b})/\mathcal G_{[C,C]\cap C(\sigma)}$ defined at the end of $\S 3$
may be identified with  the hyperk\"ahler quotient $\tilde{\mathcal Q}_{C,\sigma}(b) = (\tilde{\mathcal A}_{C,\sigma},  \Vert\cdot\Vert_{B,b}) \hkq \tilde{\mathcal G}_{[C,C]\cap C(\sigma)}(b)$.
\end{corollary}
\begin{proof}
Let $T\in \tilde{\mathcal A}_{C,\sigma}$ be a solution to the Nahm equations. The above Lemma shows that we can use a gauge transformation $u\in \tilde{\mathcal G}_{[C,C]\cap C(\sigma)}(b)$ to gauge away $\tau_0$. The resulting solution $T' = u.T$ lies in $\mathcal N_{C,\sigma}$ as defined at the end of $\S 3$ and the remaining gauge freedom by which we have to quotient is just $\mathcal G_{C(\sigma)\cap [C,C]}$ as defined in $\S 3$ (Recall that that for any $b$ we have $\mathcal G_G = \{u\in \tilde{\mathcal G}_G(b)\ | \ s(u) = \lim_{t\to\infty} \dot uu^{-1} = 0\}$).
\end{proof}
Let $\mathcal Q_{C,\tau,\sigma}$, 
\index{$\mathcal Q_{C,\tau,\sigma}$} 
respectively $\mathcal N_{C,\tau,\sigma}$, 
\index{$\mathcal N_{C,\tau,\sigma}$}
denote the subsets given by elements in $\mathcal Q_{C,\sigma}$, respectively $\mathcal N_{C,\sigma}$, asymptotic to a fixed triple $\tau$ with $C(\tau) = C$. Then the same argument as in the proof of the above Corollary together with the moment map calculation for the $T$-action on $\tilde{\mathcal Q}_{C,\sigma}(b)$ gives the following result:

\begin{corollary}
  The moduli space $(\mathcal Q_{C,\tau, \sigma}, \Vert\cdot\Vert_{B,b})/T =
  (\mathcal N_{C,\tau,\sigma}, \Vert\cdot \Vert_{B,b})/\mathcal G_{C\cap
    C(\sigma)}$ may be identified with the hyperk\"ahler quotient
  $(\tilde{\mathcal Q}_{C,\sigma},  \Vert\cdot \Vert_{B,b}) \hkq T = (\tilde{\mathcal
    A}_{C,\sigma},  \Vert\cdot \Vert_{B,b}) \hkq \tilde{\mathcal
    G}_{C\cap C(\sigma)}(b)$.
\end{corollary}

Thus, the constructions of $\S 3$ and $\S 4$ result in the same moduli spaces. But the compatible hypercomplex structure is only visible in the setting used in $\S 4$.

\section{The Relation of $(\mathcal Q_{C,\tau,\sigma}, \Vert\cdot\Vert_{B,b})/T$ and $\tilde{\mathcal Q}_{C,\sigma}(b) \hkq T$ with Coadjoint Orbits}
We now want to relate the hyperk\"ahler torus quotients of the moduli space $\tilde{\mathcal Q}_{C,\sigma}(b)$ of solutions to the Nahm equations, asymptotic to some model solution $T_{\tau_0,\tau,\sigma}$ with $C(\tau) = C$ and equipped with metric $\Vert\cdot\Vert_{B,b}$, to Kostant varieties. 

Our aim is to show that these torus reductions give coadjoint orbits. For this, we have to prove that our setup is essentially equivalent to the one used by Biquard to construct hyperk\"ahler metrics on coadjoint orbits.

\subsection{Nahm's Equations and Complex Adjoint Orbits}
We start by  briefly reviewing the results in \cite{Kronheimer:1990},\cite{Biquard:1996},\cite{Kovalev:1996}, where the existence of hyperk\"ahler structures on coadjoint orbits of semi-simple complex Lie groups was established.   We mainly follow Biquard's treatment. Given fixed triples $\tau$ and $\sigma$ with associated model solution $T_{\tau,\sigma} = \left(0,\left(\tau_i+\frac{\sigma_i}{2(t+1)}\right)_{i=1}^3\right)$ as before, we consider
\begin{equation}\label{ABiquard}
\index{$\mathcal A_{\tau,\sigma}^{Biquard}$}
\mathcal A_{\tau,\sigma}^{Biquard} = \{T_{\tau,\sigma} + \tilde T \ | \tilde T_i^D \in \Omega_\zeta(\cf), \tilde T_i^H, \dot{\tilde T_i}^H\in C^0_{2+\zeta}(\cf^\perp)\}.
\end{equation}
Here we introduced the space 
\begin{equation}\label{C0alpha} 
\index{$C^0_\alpha$}
C^0_{\alpha}(\cf^\perp) = \{f: [0,\infty)\to \cf^\perp\ | \ \Vert f\Vert_\alpha = \sup |(1+t)^{\alpha}f|<\infty\},
\end{equation} 
where $\alpha$ can be an arbitrary real number, in our case $\alpha = 2+\zeta$.
As we have now fixed the limiting triple $\tau$, tangent vectors to $\mathcal A_{\tau,\sigma}^{Biquard}$ have no  $\delta_i$-terms or $\epsilon_i$-terms. Thus, the complex structures as well as the usual $L^2$-metric are well-defined on $\mathcal A_{\tau,\sigma}^{Biquard}$. Recall that in this setting the $L^2$ metric agrees with the Bielawski metric restricted to the tangent space $T\mathcal A_{\tau,\sigma}^{Biquard}$.  This space is acted upon by the gauge group $\mathcal G_{C\cap C(\sigma)}^{Biquard}$, 
\index{$\mathcal G_{C\cap C(\sigma)}^{Biquard}$}\index{$\mathrm{Lie}(\mathcal G_{C\cap C(\sigma)}^{Biquard})$} 
whose Lie algebra is given by 
$$\mathrm{Lie}(\mathcal G_{C\cap C(\sigma)}^{Biquard}) = \{\xi:[0,\infty) \to \kf  \ | \ \xi(\infty) \in\cf\cap\cf(\sigma), \dot \xi^D \in \Omega_\zeta(\cf), \xi^H,\dot\xi^H,\ddot\xi^H \in C^0_{\zeta+2}(\cf^\perp)\},$$
It is useful to decompose $\kf$ as follows:
$$\kf = \cf\cap\cf(\sigma) \oplus \cf\cap(\cf(\sigma)^\perp) \oplus \cf^\perp.$$
We will write $\xi\in\kf$ as $\xi = \xi_0^D  + \xi_1^D + \xi^H$. 
The natural norm on $\mathrm{Lie}(\mathcal G_{C\cap C(\sigma)}^{Biquard})$ is given by 
\index{$\Vert \xi\Vert_{\mathcal G;Biquard}$}
$$\Vert \xi\Vert_{\mathcal G;Biquard} =  \Vert\dot\xi_0^D\Vert_{\Omega_{\zeta}(\cf\cap\cf(\sigma))} + \Vert\xi_1^D\Vert_{C^2_\zeta(\cf\cap\cf(\sigma)^\perp)} + \Vert\xi^H\Vert_{2+\zeta} + \Vert\dot\xi^H\Vert_{2+\zeta} + \Vert\ddot\xi^H\Vert_{2+\zeta},$$
where we used the notation $\Vert\xi\Vert_{C^2_\zeta} = \Vert\xi\Vert_\zeta + \Vert\dot \xi\Vert_{1+\zeta} + \Vert\ddot \xi\Vert_{2+\zeta}.$ 
\index{$\Vert\xi\Vert_{C^2_\zeta}$}

Biquard then showed that the moduli space of solutions to the Nahm equations 
\index{$\mathcal M_{\tau,\sigma,\zeta}^{Biquard}$}
$$\mathcal M_{\tau,\sigma,\zeta}^{Biquard} = \{T \in\mathcal A_{\tau,\sigma}^{Biquard}\ | \ \mu(T) = 0\}/\mathcal G_{C\cap C(\sigma)}^{Biquard}$$
is a smooth hyperk\"ahler manifold diffeomorphic (as a complex symplectic manifold for the complex structure $I$) to the coadjoint orbit through 
$\tau_2+i\tau_3 + \sigma_2 +i\sigma_3$, provided the triple $\tau$ satisfies the condition of Biquard regularity, i.e. $C(\tau_2,\tau_3)= C$.

\subsection{The Relation between $\tilde{\mathcal Q}_{C,\sigma}(b) \hkq T$ and $\mathcal M_{\tau,\sigma}^{Biquard}$}
We have observed before, that the hyperk\"ahler reduction
$\tilde{\mathcal Q}_{C,\sigma}(b) \hkq T$ at level $\tau\in\tf\otimes \R^3$ is the same as the quotient of the space of solutions to the Nahm equations in $\tilde{\mathcal A}_{C,\sigma}(b)$ asymptotic to the fixed triple $\tau$ (and arbitrary $\tau_0$) by the group $\tilde{\mathcal G}_{C\cap C(\sigma)}(b)$. By the results of $\S 4.4.$ this is the same as the quotient by $\mathcal G_{C\cap C(\sigma)}$ of the space $(\mathcal N_{C,\tau,\sigma}, \Vert\cdot\Vert_{L^2})$ consisting of solutions with $\tau_0=0$ converging to $\tau$. We have discussed before that for any such solution to the Nahm equations  the exponential decay rate $\eta$ of the $T_i^H$'s depends only on the triple $\tau$. Hence, the moduli space $\mathcal N_{C,\tau,\sigma}/\mathcal G_{C\cap C(\sigma)}$ is the same as the space $\mathcal N_{C,\tau,\sigma; \eta}/\mathcal G_{C\cap C(\sigma);\eta}$, where we replace $\Omega_{\exp}$ by $\Omega_{\exp;\eta}$ for suitable $\eta$ in the respective definitions. 
\index{$\mathcal N_{C,\tau,\sigma; \eta}$}

On the other hand, we see from \cite{Biquard:1996}, Propositions 2.2. and 3.1. that in Biquard's gauge theoretic setup, one can equally well replace the condition $\tilde T_i^H, \dot{\tilde T_i}^H\in C^0_{\zeta+2}(\cf^\perp)$ by $ \tilde T_i^H\in\Omega_{\exp;\eta}(\cf^\perp)$ for $\eta$ sufficiently small. 

To make this more precise, we need some compactness lemmas, which we state and prove in $\S \ref{analyticdetails}$ below.

These lemmas then imply that it is enough to prove that for a model solution $T_{\tau,\sigma}$ the operator 
$$\nabla_{\tau,\sigma}^*\nabla_{\tau,\sigma}: \mathrm{Lie}(\mathcal G_{C\cap C(\sigma);\eta}) \to C^0_{2+\zeta}(\cf) \oplus C^0_{e^{\eta t}}(\cf^\perp)$$
is an isomorphism. Here $C^0_{e^{\eta t}}(\cf^\perp)$ denotes the space of continous functions on the half-line such that the norm $\Vert f\Vert_{e^{\eta t}} = \sup (e^{\eta t}|f(t)|)$ is finite. This can be done following the proof of Proposition 3.2. 
of \cite{Biquard:1996} with minor modifications on the $\cf^\perp$-part (the proof on the $\cf$-part still works, as our setup is the same as Biquard's on the $\cf$-components). 
For the $\cf^\perp$-part, in Biquard's notation this means $\lambda,\mu,\nu$ not all equal to zero, our compactness results imply (just like in Biquard's proof) that the situation reduces to that of Kronheimer treated in the proof of \cite{Kronheimer:1990}, Propositions 3.8, 3.9. 

Altogether we obtain:

\begin{proposition}
The moduli space $(\mathcal M_{\tau,\sigma,\zeta}^{Biquard}, \Vert\cdot\Vert_{L^2})$ 
may be identified with the hyperk\"ahler quotient $\tilde{\mathcal Q}_{C,\sigma}(b) \hkq T$ at level $\tau$. 
\end{proposition}

\subsubsection{Analytic details}\label{analyticdetails}
As before, let $\alpha, \eta >0$ and denote by $C^0_{\alpha}$, respectively $C^0_{e^{\eta t}}$, 
\index{$C^0_{e^{\eta t}}$}
the spaces of continuous functions on the half-line such that the norm 
\index{$\Vert f\Vert_{\alpha}$} 
$\Vert f\Vert_{\alpha} = \sup ((1+t)^{\alpha}|f(t)|)$, respectively  
\index{$\Vert f\Vert_{e^{\eta t}}$} 
$\Vert f\Vert_{e^{\eta t}} = \sup (e^{\eta t}|f(t)|)$, is finite. It then follows from the asymptotic conditions defining $\mathrm{Lie}(\mathcal G_{C\cap C(\sigma);\eta})$ that for any solution $T\in\mathcal N_{C,\tau,\sigma;\eta}$ the operator $\nabla_T^*\nabla_T$ maps $\mathrm{Lie}(\mathcal G_{C\cap C(\sigma);\eta})$ to $C^0_{2+\zeta}(\cf)\oplus C^0_{e^{\eta t}}(\cf^\perp)$. 

Again we decompose $\kf$ as above,
$$\kf = \cf\cap\cf(\sigma) \oplus \cf\cap(\cf(\sigma)^\perp) \oplus \cf^\perp,$$
and write $\xi\in\kf$ as $\xi = \xi_0^D  + \xi_1^D + \xi^H$. 
We then define on $\mathrm{Lie}(\mathcal G_{C\cap C(\sigma);\eta})$ the following norm: 
$$\Vert\xi\Vert_\mathcal G = \Vert\dot\xi_0^D\Vert_{\Omega_\zeta(\cf\cap\cf(\sigma))} +\Vert\xi_1^D\Vert_{C^2_\zeta(\cf\cap\cf(\sigma)^\perp)}+\Vert\xi^H\Vert_{C^2_{e^{\eta t}}(\cf^\perp)},$$
where the norms $\Vert\cdot\Vert_{C^2_\zeta(\cf\cap\cf(\sigma^\perp)}$ and $\Vert\cdot\Vert_{C^2_{e^{\eta t}}(\cf^\perp)}$ are defined by
\begin{eqnarray*}
\Vert f\Vert_{C^2_\zeta} &=& \Vert f\Vert_\zeta + \Vert\dot f\Vert_{1+\zeta} + \Vert\ddot f\Vert_{2+\zeta}  \\
\Vert f\Vert_{C^2_{e^{\eta t}}} &=& \Vert f\Vert_{e^{\eta t}} + \Vert\dot f\Vert_{e^{\eta t}} + \Vert\ddot f\Vert_{e^{\eta t}}.
\end{eqnarray*}
Note that our setup agrees with the one used by Biquard on the $D$-parts, but on the $H$-part we replace the weight function $(1+t)^{2+\zeta}$ by $e^{\eta t}$.
 
Writing 
\index{$\Omega'_\zeta(\cf\cap\cf(\sigma))$} 
$\Omega'_\zeta(\cf\cap\cf(\sigma)) = \{\xi:[0,\infty)\to \cf\cap\cf(\sigma) | \Vert\dot\xi\Vert_{\Omega_\zeta}<\infty\}$, we get 
$$\mathrm{Lie}(\mathcal G_{C\cap C(\sigma);\eta}) = \{\xi\in \Omega'_\zeta(\cf\cap\cf(\sigma))\oplus C^2_\zeta(\cf\cap\cf(\sigma)^\perp) \oplus C^2_{e^{\eta t}}(\cf^\perp) | \ \xi(0)=0\}.$$
Our aim is to show that $\nabla_T^*\nabla_T$ is a compact perturbation of the operator $\nabla_{\tau,\sigma}^*\nabla_{\tau,\sigma}$ associated with the model solution $T_{\tau,\sigma}$. 

Note that for a solution $T = T_{\tau,\sigma} + \tilde T$ we can write $$\nabla_T^*\nabla_T\xi - \nabla_{\tau,\sigma}^*\nabla_{\tau,\sigma}\xi =  2[T_0,\dot \xi] + A\cdot\xi,$$
where $A$ is the linear operator given by
$$A\cdot\xi = [\dot T_0,\xi] + \sum_{i=1}^3 \mathrm{ad}(T_{\tau,\sigma;i}) (\mathrm{ad}(\tilde T_i )(\xi))+  \mathrm{ad} (\tilde T_i )(\mathrm{ad} (T_{\tau,\sigma; i})(\xi)) + (\mathrm{ad}(\tilde T_i))^2(\xi).$$ 
The following series of corollaries will imply that $\xi\mapsto  2[T_0,\dot \xi] + A\cdot\xi$ defines a compact operator $$\mathrm{Lie}(\mathcal G_{C\cap C(\sigma);\eta}) 
\to C^0_{2+\zeta}(\cf) \oplus C^0_{e^{\eta t}}(\cf^\perp).$$
With respect to this direct sum decomposition, we may view  $A + 2\mathrm{ad}(T_0)\circ \frac{d}{dt}$ as a $(2\times 3)$ block matrix and we check that each of its six block components defines a compact operator between the respective function spaces.

The analytic lemmas we need in order to prove this are all based on the following adaptation of Ascoli's theorem to the space $C^0_0(\kf)$ 
\index{$C^0_0(\kf)$}
of bounded continuous $\kf$-valued functions on the half-line equipped with the sup-norm.

\begin{lemma}\label{ascoli}
Let $\Phi\subset C^0_0(\kf)$ be a 
uniformly bounded family of continuous functions such that there exists a positive constant $C$ and a positive, 
monotonically decreasing continuous function $\lambda$ with $\lim_{t\to\infty}\lambda(t) =0$ such that  for any $f\in \Phi$ and $t>t'$ the following estimate holds:
$$|f(t)-f(t')|\leq C(\lambda(t') - \lambda(t)).$$
Then $\Phi\subset C^0_0(\kf)$ is relatively compact.
\end{lemma}
\begin{proof}
The proof is based on the proof of Ascoli's theorem in (\cite{Lang:1983}, Theorem 3.1.).
The set $\Phi$ forms a uniformly bounded family of continuous functions.  We want to show that $\Phi$ is relatively compact in the space $C^0_0([0,\infty))$ of bounded continuous functions endowed with the $\sup$-norm. 

The inequality 
$$|f(t)-f(t')|\leq C(\lambda(t') - \lambda(t))$$
shows that $\Phi$ is equicontinuous. Moreover, since $\lambda(t)$ converges to zero monotonically as $t\to\infty$ it follows that given $\epsilon >0$, we can cover $[0,\infty)$ by \emph{finitely} many open intervals $V(t_i)$, $i=1,\dots,m$ such that for all $t\in V(t_i)$ we have $|f(t) - f(t_i)| < \epsilon$. Now consider the set $\Phi(t_i) = \{f(t_i)\ | \ f \in \Phi\}$, which is relatively compact in $\kf$. Therefore the (finite) union 
$$Y = \bigcup_{i=1}^m \Phi(t_i)$$
is also relatively compact in $\kf$ and we can cover it by finitely many open balls $B_\epsilon(\xi_j)$, $j=1,\dots,n$ of radius $\epsilon$ centered at $\xi_j$. 
Moreover, there is some map $\sigma: \{1,\dots,m\}\to\{1,\dots,n\}$ such that
for each $i$ we have $f(t_i) \in B_\epsilon(\xi_{\sigma(i)})$.
Therefore if we associate to such a map $\sigma$ the space $\Phi_\sigma$ of all elements $f\in \Phi$ such that $|f(t_i) - \xi_{\sigma(i)}|<\epsilon$
for each $i$, then $\Phi$ is covered by the open sets $\Phi_\sigma$. Finally, we see that if $f,g\in \Phi_\sigma$ and $t\in [0,\infty)$, then $t\in V(t_i)$ for some $i$ and 
$$|f(t)-g(t)| \leq |f(t)-f(t_i)| +|f(t_i)-\xi_{\sigma(i)}|+|\xi_{\sigma(i)}-g(t_i)| + |g(t_i)-g(t)| < 4\epsilon.$$
Thus, we can cover $\Phi$ by a finite number of open sets of prescribed diameter, i.e. $\Phi$ is relatively compact.
\end{proof}

The next corollary shows that the map $$\xi\mapsto  A\cdot\xi = 2[\dot T_0,\xi] + \sum_{i=1}^3 \mathrm{ad}(T_{\tau,\sigma;i}) (\mathrm{ad}(\tilde T_i )(\xi))+  \mathrm{ad} (\tilde T_i )(\mathrm{ad} (T_{\tau,\sigma; i})(\xi)) + (\mathrm{ad}(\tilde T_i))^2(\xi)$$ defines a compact operator. 

\begin{corollary}\label{Acompact}
\begin{enumerate}
\item Let $\zeta>0$ and let $f_\alpha\in \Omega_\alpha$ with $\alpha >1$. Then multiplication by $f_\alpha$ defines a compact operator $f_\alpha: C^2_\zeta\to C^0_{2+\zeta}$.  

\item Let $\eta>0, \zeta>0$ and let $f_\alpha\in C^1([0,\infty),\mathbb R)$ such that $\Vert e^{\eta t}f_\alpha\Vert_{C^0_{\alpha}}+ \Vert e^{\eta t}\dot f_\alpha\Vert_{C^0_{\alpha}}<\infty$, with $\alpha >1-\zeta$. Then multiplication by $f_\alpha$ defines a compact operator $f_\alpha: C^2_{\zeta} \to C^0_{e^{\eta t}}$.  

\item Let $\eta>0$ and let $f_\alpha\in \Omega_\alpha$ with $\alpha >1$. Then multiplication by $f_\alpha$ defines a compact operator $f_\alpha: C^2_{e^{\eta t}}\to C^0_{e^{\eta t}}$.  

\item Let $\eta>0, \zeta>0$ and let $f_\alpha\in \Omega_\alpha$, with $\alpha >1$. Then multiplication by $f_\alpha$ defines a compact operator $f_\alpha: C^2_{e^{\eta t}}\to C^0_{2+\zeta}$.  

\item Let $\eta>0$ and let $f_\alpha\in C^1([0,\infty),\mathbb R)$ such that $\Vert e^{\eta t}f_\alpha\Vert_{C^0_{\alpha}}+ \Vert e^{\eta t}\dot f_\alpha\Vert_{C^0_{\alpha}}<\infty$, with $\alpha >1$. Then multiplication by $f_\alpha$ defines a compact operator $f_\alpha: \Omega_{\zeta}' = \{\xi\in C^2([0,\infty),\kf) | \xi(0) = 0, \dot\xi\in\Omega_{\zeta}\}\to C^0_{e^{\eta t}}$.  

\item Let $\eta>0$ and let $f_\alpha\in C^1([0,\infty),\mathbb R)$ such that $\Vert f_\alpha\Vert_{\Omega_{\alpha}}<\infty$, with $\alpha >1+\zeta$. Then multiplication by $f_\alpha$ defines a compact operator $f_\alpha: \Omega_{\zeta}' = \{\xi\in C^2([0,\infty),\kf) | \xi(0) = 0, \dot\xi\in\Omega_{\zeta}\}\to C^0_{2+\zeta}$.  
\end{enumerate}
\end{corollary}
\begin{proof}
\begin{enumerate}
\item We have to show that the  $f_\alpha$ sends  the unit ball in $C^2_\zeta$ to a relatively compact set in $C^0_{2+\zeta}$. We thus have to check that $\Phi = \{(1+t)^{2+\zeta}f_\alpha \xi | \xi \in C^2_\zeta(\kf), \Vert\xi\Vert_{C^2_\zeta}\leq 1 \} \subset C^0_0$ satisfies the assumptions of Lemma \ref{ascoli}. Now
\begin{eqnarray*}
|(1+t)^{2+\zeta}f_\alpha(t)\xi(t) -(1+t')^{2+\zeta}f_\alpha(t')\xi(t')| &\leq& \int_{t'}^t \left|\frac{d}{ds} (1+s)^{2+\zeta}f_\alpha(s)\xi(s)\right| ds\\
&\leq& \int_{t'}^t |(2+\zeta)(1+s)^{1+\zeta}f_\alpha(s)\xi(s)| ds \\
& &+\int_{t'}^t|(1+s)^{2+\zeta}(\dot f_\alpha(s)\xi(s)+f_\alpha(s)\dot \xi(s))| ds\\
&\leq&  \int_{t'}^t |(2+\zeta)(1+s)f_\alpha(s)\Vert\xi\Vert_{C^2_\zeta}| ds \\
& &+\int_{t'}^t |(1+s)^{2}\dot f_\alpha(s)\Vert\xi\Vert_{C^2_\zeta}| + |(1+s)f_\alpha(s)\Vert\xi\Vert_{C^2_\zeta})| ds\\
&\leq& \int_{t'}^t |(3+\zeta)(1+s)f_\alpha(s)| + |(1+s)^{2}\dot f_\alpha(s)| ds\\
&\leq& (3+\zeta)\Vert f_\alpha\Vert_{\Omega_\alpha} \int_{t'}^t \frac{1}{(1+s)^\alpha}ds\\ 
&=&  C(\zeta,f_\alpha)\left(\frac{1}{(1+t')^{\alpha-1}}-\frac{1}{(1+t )^{\alpha-1}}\right),
\end{eqnarray*}
where $C = \frac{3+\zeta}{\alpha-1}\Vert f_\alpha\Vert_{\Omega_\alpha}$ is a constant depending only on $f_\alpha, \zeta$.  Thus, if $\alpha>1$, the compactness is implied by Lemma \ref{ascoli}.
\item This assertion is proved analogously by estimating $|e^{\eta t}f_\alpha(t)\xi(t) -e^{\eta t'}f_\alpha(t')\xi(t')|$ with $\xi\in C^2_\zeta$,   and applying Lemma \ref{ascoli}.  We omit the details.
\item This assertion is proved analogously by estimating $|e^{\eta t}f_\alpha(t)\xi(t) -e^{\eta t'}f_\alpha(t')\xi(t')|$ with $\xi\in C^2_{e^{\eta t}}$ and applying Lemma \ref{ascoli}.  We omit the details.
\item This assertion is proved analogously by estimating $|(1+t)^{2+\zeta}f_\alpha(t)\xi(t) -(1+t')^{2+\zeta}f_\alpha(t')\xi(t')|$ with $\xi\in C^2_{e^{\eta t}}$ and applying Lemma \ref{ascoli}.  We omit the details.
\item We again have to check that $\Phi = \{e^{\eta t}f_\alpha \xi | \xi \in \Omega_{\zeta}', \Vert \dot \xi\Vert_{\Omega_\zeta}\leq 1 \} \subset C^0_0$ satisfies the assumptions of Lemma \ref{ascoli}. We start by observing that since $\xi(0) =0$  and $\xi\in\Omega_\zeta'$, we may write 
$$\xi(t) = \int_0^t \dot \xi(s)ds,$$
from which we deduce that 
$$|\xi(t)| \leq \int_0^t|\dot \xi(s)|ds \leq \Vert\dot \xi\Vert_{\Omega_\zeta}\int_0^t \frac{1}{(1+s)^{1+\zeta}}ds = \frac{\Vert\dot \xi\Vert_{\Omega_\zeta}}{\zeta}\left(1-\frac{1}{(1+t)^\zeta}\right) \leq \frac{\Vert\dot \xi\Vert_{\Omega_\zeta}}{\zeta}.$$
We now calculate using $|\xi(t)| \leq \frac{\Vert\dot \xi\Vert_{\Omega_\zeta}}{\zeta} = \frac{1}{\zeta}$ and $|\dot \xi(t)| \leq \Vert\dot \xi\Vert_{\Omega_\zeta} = 1$:
\begin{eqnarray*}
|e^{\eta t}f_\alpha(t)\xi(t) -e^{\eta t'}f_\alpha(t')\xi(t')| &\leq& \int_{t'}^t \left|\frac{d}{ds} e^{\eta s}f_\alpha(s)\xi(s)\right| ds\\
&\leq& \int_{t'}^t \left|\eta e^{\eta s}f_\alpha(s)\xi(s)\right| + |e^{\eta s}\dot f_\alpha(s)\xi(s)| + |e^{\eta s}f_\alpha(s)\dot \xi(s))| ds\\
&\leq& \int_{t'}^t \left(1+ \frac{\eta}{\zeta}\right)|e^{\eta s}f_\alpha(s)| + \frac{1}{\zeta}|e^{\eta s}\dot f_\alpha(s)|ds \\
&\leq& \frac{\zeta +1+\eta}{\zeta}(\Vert e^{\eta t}f_\alpha\Vert_{C^0_{\alpha}}+ \Vert e^{\eta t}\dot f_\alpha\Vert_{C^0_{\alpha}})\int_{t'}^{t} \frac{1}{(1+s)^\alpha}ds.\\
&=&  C(\zeta,\eta,f_\alpha)\left(\frac{1}{(1+t')^{\alpha-1}}-\frac{1}{(1+t )^{\alpha-1}}\right).
\end{eqnarray*}
Thus, if $\alpha>1$, we may again apply Lemma \ref{ascoli}.

\item This assertion is proved similarly to assertion $(5)$ by estimating $|(1+t)^{2+\zeta}f_\alpha(t)\xi(t) -(1+t')^{2+\zeta}f_\alpha(t')\xi(t')|$ for $\xi\in\Omega_{\zeta}'$ and applying Lemma \ref{ascoli}.
\end{enumerate}
\end{proof}

The following two results show that $\xi\mapsto \mathrm{ad}(T_0)(\dot\xi)$ is the composition of a bounded and a compact operator, hence compact. 

\begin{corollary}
\begin{enumerate}
\item Let $f\in\Omega_{\alpha}$.
Then point-wise multiplication by $f_\alpha$ defines a compact operator $f_\alpha: \Omega_{\zeta}\to C^0_{2+\zeta}$ provided $\alpha>0$.
\item Let $f\in\Omega_{\alpha}$.
Then point-wise multiplication by $f_\alpha$ defines a compact operator \newline $f_\alpha: \Omega_{\exp;\eta}\to C^0_{e^{\eta t}}$ provided $\alpha>0$.
\item Let $f\in\Omega_{\alpha}$.
Then point-wise multiplication by $f_\alpha$ defines a compact operator \newline$f_\alpha: \Omega_{\exp;\eta}\to C^0_{2+\zeta}$ provided $\alpha>0$.
\item Let $f\in\Omega_{\exp;\alpha}$.
Then point-wise multiplication by $f_\alpha$ defines a compact operator \newline $f_\alpha: \Omega_{\zeta}\to C^0_{e^{\eta t}}$ provided $\alpha\geq\eta$.
\end{enumerate}
\end{corollary}
\begin{proof}
\begin{enumerate}
\item Let $\xi$ be an element of the unit ball in $\Omega_\zeta$. We use the differentiability of the elements in $\Omega_\zeta$ and $f_\alpha$ to calculate
\begin{eqnarray*}
|(1+t)^{2+\zeta}f_\alpha(t)\xi(t) -(1+t')^{2+\zeta}f_\alpha(t')\xi(t')| &\leq& \int_{t'}^t \left|\frac{d}{ds} (1+s)^{2+\zeta}f_\alpha(s)\xi(s)\right| ds\\
&=& \int_{t'}^t \left|(2+\zeta)(1+s)^{1+\zeta}f_\alpha(s)\xi(s)\right| ds \\
& &+\int_{t'}^t\left|(1+s)^{2+\zeta}(\dot f_\alpha(s)\xi(s)+f_\alpha(s)\dot \xi(s))\right| ds\\
&\leq& \Vert\xi\Vert_{\Omega_\zeta} \int_{t'}^{t} (3+\zeta)|f_\alpha(s)| + |(1+s)\dot f_\alpha(s)|ds\\
&\leq&(3+\zeta) \Vert f_\alpha\Vert_{\Omega_\alpha} \int_{t'}^{t} \frac{1}{(1+s)^{1+\alpha}} ds\\
&=& C(\zeta, f_\alpha, \alpha)\left(\frac{1}{(1+t')^{\alpha}}-\frac{1}{(1+t )^{\alpha}}\right).
\end{eqnarray*}
The claim follows again from Lemma \ref{ascoli}. 
\end{enumerate}
The other assertions are proved similarly, see the remarks for assertions $(2),(3),(4)$ and $(6)$ in the proof of Corollary \ref{Acompact}.
\end{proof}

\begin{lemma}
Differentiation $\xi\mapsto \dot \xi$ defines a bounded operator $\mathrm{Lie}(\mathcal G_{C\cap C(\sigma)})\to \Omega_\zeta(\cf)\oplus\Omega_{\exp;\eta}(\cf^\perp)$.
\end{lemma}
\begin{proof}
Let $\xi$ be an element of the unit ball in $\mathrm{Lie}{\mathcal G}_{[C,C];\eta}$.
We have 
\begin{eqnarray*}
\Vert\dot\xi\Vert_{\Omega_\zeta(\cf)\oplus\Omega_{\exp;\eta}(\cf^\perp)} &=& \Vert\dot\xi^D\Vert_{\Omega_\zeta} \oplus \Vert\dot \xi^H\Vert_{\Omega_{\exp;\eta}} \\
&\leq& \Vert\dot\xi_0^{D}\Vert_{\Omega_\zeta(\cf\cap\cf(\sigma))} +\Vert\dot\xi_1^{D}\Vert_{\Omega_\zeta(\cf\cap\cf(\sigma)^\perp)} \oplus \Vert\dot \xi^H\Vert_{\Omega_{\exp;\eta}} \\
&\leq&  \Vert\dot\xi_0^{D}\Vert_{\Omega_\zeta(\cf\cap\cf(\sigma))} +\Vert\xi_1^{D}\Vert_{C^2_\zeta(\cf\cap\cf(\sigma)^\perp)} \oplus \Vert\xi^H\Vert_{C^2_{e^{\eta t}}}\\
&=& \Vert\xi\Vert_{\mathrm{Lie}(\mathcal G_{C\cap C(\sigma)})}.
\end{eqnarray*}
\end{proof}

\section{The Nahm Implosion of $T^*K_\mathbb C$}
Now given a general hyperk\"ahler manifold $M$ with a tri-hamiltonian $K$-action we can define its Nahm hyperk\"ahler implosion to be the hyperk\"ahler quotient 
\index{$M_{\rm hkimpl}$}
$$M_{\rm hkimpl} = (M \times \mathcal Q) \hkq K.$$
Here we recall (see also the discussion in $\S 3$) that $\mathcal Q$ is the stratified hyperk\"ahler space $\mathcal Q = \coprod_{C} \mathcal Q_C$, where the strata themselves are given by $\mathcal Q_C = \coprod_\sigma \mathcal Q_{C,\sigma}$.

The universal symplectic implosion for $K$ may be viewed as the implosion
$(T^*K)_{\rm impl}$ of the cotangent bundle $T^*K$. In this section we 
shall relate $\mathcal Q$ to the Nahm hyperk\"ahler implosion
of $T^* K_{\mathbb C}$, where the latter space has the hyperk\"ahler
structure found by Kronheimer \cite{Kronheimer:1988}, which we recall in the next section.

\subsection{$T^*K_\mathbb C$ and Nahm's Equations}
In \cite{Kronheimer:1988} Kronheimer has realised $T^*K_\mathbb C$ as a moduli space of smooth solutions to the Nahm equations on $[0,1]$. The precise setting is as follows. Let 
$$\mathcal A = C^1([0,1],\kf)^4 \cong C^1([0,1],\kf)\otimes \mathbb H.$$
This becomes a flat infinite-dimensional hyperk\"ahler manifold with hypercomplex structure given by right multiplication by $-i,j,k$ and metric given by the $L^2$-metric.
Define the gauge group 
\begin{equation}\label{G00} 
\index{$\mathcal G_{00}$}
\mathcal G_{00} = \{u\in C^2([0,1],K)\ | \  u(0) = 1 = u(1) \},
\end{equation}
the Lie algebra of which is given by 
\begin{equation}\label{LieG00}
\index{$\mathrm{Lie}(\mathcal G_{00})$}
\mathrm{Lie}(\mathcal G_{00}) = \{\xi\in C^2([0,1],\kf)\ | \ \xi(0) = 0=\xi(1)\},
\end{equation}
i.e. $C^2$-paths in $\mathfrak k$ vanishing at the end points of the interval. Kronheimer then shows that, as a holomorphic symplectic manifold, the moduli  space $\mathcal M_{00}$ of solutions to the Nahm equations modulo gauge transformations may be identified with $T^*K_\mathbb C$. The identification (for the complex structure $I$) goes via putting $\alpha = T_0-iT_1$, $\beta= T_2+iT_3$. Then Nahm's equations are equivalent to 
\begin{eqnarray}
\dot\beta + [\alpha,\beta] &=&0\quad \text{(the complex equation)}\\
\dot \alpha + \dot\alpha^* + [\alpha,\alpha^*] + [\beta,\beta^*] &=&0\quad \text{(the real equation)}.
\end{eqnarray} 
It then turns out that  $\mathcal M_{00} \cong \mathcal M_{00}^\C$, where $\mathcal M_{00}^\C$ is the space of solutions $(\alpha,\beta)$ to the complex equation modulo complex gauge transformations. The result that is needed here is that in the complex gauge orbit of a solution $(\alpha, \beta)$ to the complex equation  there exists a solution to the real equation and any two solutions to the real equation that lie in the same complex gauge orbit are equivalent under a gauge transformation
valued in the compact group.
Then define a map 
$$\phi: \mathcal M_{00}^\C \to K_\mathbb C \times \mathfrak k_\mathbb C\cong T^*K_\mathbb C \qquad (\alpha,\beta)\mapsto (g(1), \beta(0)),$$
where $g$ is the unique complex gauge transformation such that $g(0) = 1$ and $\mathrm{Ad}(g)(\alpha) - \dot gg^{-1} = 0$ (that is, $g$ gauges $\alpha$ to zero).

The following lemma is proved in \cite{DancerSwann:1996} (where a different sign convention is used). The argument is similar to our previous calculation of the moment map for the $T$-action on $\mathcal Q_C$.
\begin{lemma}
Let $\mathcal G_0$ be the group of gauge transformations equal to the identity at $t=0$ and lying in $K$ at $t=1$. Then $K = \mathcal G_0/\mathcal G_{00}$ acts on $\mathcal M_{00}$ preserving the hyperk\"ahler structure with moment map given by 
$$\mu(T_0,T_1,T_2,T_3) = (T_1(1),T_2(1),T_3(1)).$$ 
\end{lemma}

\subsection{The Implosion of $T^*K_\mathbb C$}
\begin{lemma}
Let $(T_0,T_1,T_2,T_3)$ be a solution to Nahm's equations on an interval $[0,a)$ ($a= \infty$ is allowed). Then there exists a gauge transformation $u: [0,a) \to K$ such that 
\begin{enumerate}
\item $u(0) = 1 = u(a)$,
\item There exists an $\epsilon >0$ such that on $[0,\epsilon]$ we have $u.(T_0,T_1,T_2,T_3) = (0,\mathrm{Ad}(u)(T_1), \mathrm{Ad}(u)(T_2), \mathrm{Ad}(u)(T_3))$ and $u(t) \equiv 1$ for $t\geq a/2$, (respectively $t\geq 2$ if $a=\infty$). 
\end{enumerate}
\end{lemma}
\begin{proof}
Let $h$ be a cut-off function such that 
\begin{itemize}
\item $h\equiv 1$ on $[0,1]$
\item $h\equiv 0$ on $[2,\infty)$
\end{itemize}
and define $h_1(t) = h(4t/a)$ (if $a= \infty$ just take $h_1 = h$).
Then applying the gauge transformation $u_1 = \exp(th_1(t)T_0(0))$ we find that 
\begin{eqnarray*}
(u_1.T_0)(0) &=& (\mathrm{Ad}(u_1(t)) T_0(t) - \dot u_1(t)u_1^{-1}(t))|_{t=0}\\
&=& (\mathrm{Ad} (u_1(t)) T_0(t) - h_1(t)T_0(0) - t\dot h_1(t)T_0(0))|_{t=0} \\
&=& T_0(0) - T_0(0) = 0.
\end{eqnarray*}
Writing $\tilde{T}_0 = u_1.T_0$, we see that $\tilde{T}_0(0) = 0$. Let $u_2$ be the
unique solution with $u_2(0)=1$ to 
$$\mathrm{Ad}(u_2)\tilde{T}_0 - \dot u_2u_2^{-1} = 0.$$
(which is equivalent to the linear equation $\dot{u_2} = u_2 \tilde{T}_0$).
Now pick $0 < \epsilon \leq \min(\frac{a}{4},1)$ sufficiently small such that 
$$|\tilde{T}_0(t)| \leq\frac{1}{2}\mathrm{injrad}(K) \qquad \forall t\in [0,\epsilon],$$
where $\mathrm{injrad}(K)$ denotes the injectivity radius of the Riemannian manifold $K$, whose metric is given by the fixed invariant inner product on $\kf$. 

Then we have that the length of the curve $u_2$ in $K$ restricted to the interval $[0,\epsilon]$ can be bounded as follows:
$$l(u_2) = \int_0^\epsilon |u_2^{-1}(s)\dot u_2(s)|ds = 
\int_0^\epsilon |\tilde{T}_0(s)|ds \leq \frac{1}{2}\mathrm{injrad}(K)\epsilon\leq\frac{1}{2}\mathrm{injrad}(K).$$
Thus, for $t\in [0,\epsilon]$ we may write $u_2(t) = \exp(\xi_2(t))$, where $\xi_2$ is a function taking values in $\mathfrak k$ such that $\xi_2(0) = 0$. Now put $h_2(t) = h(t/\epsilon)$ and define $u(t) = \exp(h_2(t)\xi_2(t))u_1(t)$. Then by construction $u(0) = 1$ and $u(t) = 1$ for $t\geq\frac{a}{2}$ and for $0\leq t<\epsilon$ we see
$$u.T_0 = u_2.\tilde{T}_0 = 0.$$
\end{proof}

Now use the map $t\mapsto t+1$ to identify $\mathcal Q$ as a moduli space of solutions to Nahm's equations on $[1,\infty)$ and consider the hyperk\"ahler manifold $T^*K_\mathbb C\times \mathcal Q$. 

\begin{proposition}
Let $b>0$ and consider the diagonal $K$-action on $T^*K_\mathbb C \times \mathcal Q$ given by gauge transformations lying in $K$ at $t=1$. 
Then 
$$((T^*K_\mathbb C, \Vert\cdot\Vert_{L^2}) \times (\mathcal Q, \Vert\cdot\Vert_{B,b})) \hkq K \cong (\mathcal Q, \Vert\cdot\Vert_{B,b+1}).$$
\end{proposition}
\begin{proof}
We work stratum by stratum.
We have already calculated the moment maps for the individual factors and the moment map for the diagonal $K$-action on the product is simply the sum of the moment maps on the factors. Hence for elements $T\in T^*K_\mathbb C$ and 
$\tilde{T} \in \mathcal Q_{C,\sigma}$ we find
$$\mu(T,\tilde{T}) = (T_1(1)-\tilde{T}_1(1),T_2(1)-\tilde{T}_2(1),
T_3(1)-\tilde{T}_3(1)).$$ 
Thus, $\mu^{-1}(0)$ consists of pairs of
solutions $T,\tilde{T}$ to Nahm's equations on $[0,1]$ and
$[1,\infty)$ respectively such that for $i=1,2,3$ we have $T_i(1) =
  \tilde{T}_i(1)$. It follows from the previous lemma that using the
  gauge groups we can represent $T,\tilde{T}$ by solutions such that in a
  neighbourhood of $t=1$ the $T_0$-terms vanish and we can therefore
  represent points in the zero locus of the hyperk\"ahler moment map
  by smooth solutions to the Nahm equations on the whole half-line
  $[0,\infty)$.

Now first quotienting out by the gauge actions on the individual
factors and then quotienting out by the diagonal $K$-action amounts to the
same thing as quotienting out by $\tilde{\mathcal G}_{[C,C]\cap
  C(\sigma)}(b)$.

Let $(T,\tilde{T})$ be a solution such that near $t=1$ we have
$\tilde{T}_0=0=T_0$. Then by uniqueness of solutions to linear ODEs,
tangent vectors to $(T^* K_\C\times \mathcal Q) \hkq K$ at $(T,\tilde{T})$ can
be represented by paris of tangent vectors $(X,\tilde{X})$ on the individual
factors that agree at $t=1$. The product metric applied to such a pair
gives
$$\Vert(X,\tilde{X})\Vert^2 = \Vert X\Vert_{L^2}^2 + \Vert\tilde{X}\Vert^2_{B,b} = \int_0^1 \sum_{i=0}^3|X_i|^2 + b\sum_{i=0}^3 |\tilde X_i(\infty)|^2 + \int_1^\infty  \sum_{i=0}^3 (|\tilde{X}_i|^2-|\tilde{X}_i(\infty)|^2).$$
If we define 
$$\bar X = \begin{cases} X & \mbox{on } [0,1] \\ \tilde{X} &\mbox{on } [1,\infty),\end{cases}$$
then we find 
$$\Vert(X,\tilde{X})\Vert^2 = \Vert\bar X\Vert^2_{B,b+1}.$$
\end{proof}

\begin{remark}
\begin{enumerate}
\item 
See also the discussion at the beginning of $\S 5$ in \cite{Bielawski:1998}, where Bielawski has obtained a similar result in a slightly different setting.
\item Note that there is nothing special about the interval $[0,1]$ in Kronheimer's description of $T^*K_\C$ - any compact interval will do. Thus, the parameter $b$ is shifted by the length of the interval used to define the hyperk\"ahler metric on $T^*K_\C$.
\item For any $b,b'>0$ with $b'=rb$ for $r>0$, the moduli spaces $\mathcal Q_{C,b}$ and $\mathcal Q_{C,b'}$ are isometric by the map $T(t) \mapsto rT(rt)$.
\end{enumerate}
\end{remark}

\section{Complex structures}

In order to analyse the complex structures, as usual we split the Nahm equations
into real and complex parts:

\begin{equation*}
\alpha = T_0 - i T_1, \quad  \beta = T_2 + i T_3 
\index{$\alpha, \beta$}
\end{equation*}

The moduli spaces will be equivalent (cf. \cite{Biquard:1996}) to
the relevant spaces of solutions to the complex equation
\begin{equation} \label{cplexeqn}
\frac{d \beta}{dt} = [\beta, \alpha]
\end{equation}
modulo the group $\tilde{\mathcal G}_{[C,C]\cap C(\sigma)}(b)^\C$ 
\index{$\tilde{\mathcal G}_{[C,C]\cap C(\sigma)}(b)^\C$} 
of complex gauge transformations $g : [0, \infty) \rightarrow K_\C$ satisifying asymptotic conditions analogous to those of $\tilde{\mathcal G}_{[C,C]\cap C(\sigma)}(b)$. For example, $\xi:[0,\infty)\to\gf = \kf_\C$ lies in the Lie algebra $\mathrm{Lie}(\tilde{\mathcal G}_{[C,C]\cap C(\sigma)}(b)^\C)$ if can be written in the form 
$$\xi(t)= (t-b)\dot\xi(\infty) + \tilde\xi, $$
where $\xi(0) = 0$, $\dot\xi(\infty)\in Z(\cf)_\C$ and $\tilde\xi(\infty) \in [\cf_\C,\cf_\C]\cap\cf(\sigma)_\C$. The components $\xi^D$ and $\xi^H$ with respect to $\gf = \cf_\C \oplus \cf^\perp_\C$ are required to satisfy the same decay condition as before. 

We shall also need the group $\tilde{\mathcal G}^0_{[C,C]\cap C(\sigma)}(b)^\C$ 
\index{$\tilde{\mathcal G}^0_{[C,C]\cap C(\sigma)}(b)^\C$}
 of complex gauge transformations $u$ that are not constrained to be equal to the identity at $t=0$.

We write (following the notation of \cite{Biquard:1996})
\[
\alpha = \tau_0 - \tau_r - \frac{H}{t+1} + \ldots
\]
\[
\beta = \tau_\C + \frac{Y}{t+1} + \ldots
\]
where $\tau_r = i \tau_1, \tau_\C = \tau_2 + i \tau_3$ 
\index{$\tau_r, \tau_\C$} and
$H = i \sigma_1, Y = \frac{1}{2}(\sigma_2 + i \sigma_3)$. We denote
$\tau_r + \frac{H}{t+1}$ by $\alpha_0$ and $\tau_\C + \frac{Y}{t+1}$
by $\beta_0$. 
\index{$\alpha_0,\beta_0$} 
Writing $X=\frac{1}{2}(-\sigma_2+i\sigma_3)$, we thus obtain an $\sln(2)$-triple $(H,X,Y)$ satisfying $[H,Y] = -2Y, [H,X] = 2X, [X,Y] = H$. 
\index{$H,X,Y$}

As in \cite{Biquard:1996}, we aim to find a complex gauge transformation in $\tilde{\mathcal G}^0_{[C,C]\cap C(\sigma)}(b)^\C$ (i.e. one that does not necessarily satisfy $g(0) =1$), that gauges $\alpha_0$ to $\alpha$, that is $\alpha = g. \alpha_0 = g \alpha_0 g^{-1} - \dot{g} g^{-1}$.
This amounts to solving
\[
\frac{dg}{dt} = g \alpha_0 - \alpha g
\]
First gauge $\tau_0-\tau_r$ to zero as in $\S 4$ by the transformation
\[
\exp((t - b + be^{-ct}) (\tau_0-i\tau_1))
\]
which lies in $\tilde{\mathcal G}_{[C,C] \cap C(\sigma)}(b)^\C$. From
this observation it follows in particular that the complexified gauge
group does \emph{not} act on the whole space $\tilde{\mathcal
  A}_{C,\sigma,b}$, as it does not preserve the common centraliser of
the limiting triple $\tau$ in general. However, the complexified gauge
group acts on a slightly smaller space $\tilde{\mathcal
  A}_{C,\sigma}^{reg}(b)$ 
  \index{$\tilde{\mathcal A}_{C,\sigma}^{reg}(b)$} 
  which consists of Nahm data asymptotic to
triples $\tau$ which are \emph{Biquard-regular},
i.e. $C=C(\tau_2,\tau_3)$. We now want to describe the complex
structure of the subset $\tilde{\mathcal Q}_{C,\sigma}^{reg}(b)\subset
\tilde{\mathcal Q}_{C,\sigma}(b)$, 
\index{$\tilde{\mathcal Q}_{C,\sigma}^{reg}(b)$} 
the moduli space of solutions
asymptotic to some Biquard-regular triple. The above observation shows
that after a gauge transformation in $\tilde{\mathcal G}_{[C,C] \cap
  C(\sigma)}(b)^\C$ we may assume that the limiting triple is of the
form $(0,\tau_2,\tau_3)$.

 By Proposition 4.1. in \cite{Biquard:1996} we can find a gauge transformation $g\in\tilde{\mathcal G}^0_{C\cap C(\sigma)}(b)^\C$ with \\ $\lim_{t\to\infty}\dot g(t)g^{-1}(t) =0$ such that 
$$(\alpha,\beta) = g.(\alpha_0,\beta_0 ),$$
In particular, there is no nilpotent term ($n$ in the
terminology of \cite{Biquard:1996}) in the $\beta$-coordinate, as we are working with \emph{Biquard-regular} triples. The gauge transformation $g$ has a limit  $g(\infty) \in (C\cap C(\sigma))_\C$ since $\dot gg^{-1}\in \Omega_{\zeta}(\cf_\C) \oplus \Omega_{\exp}(\cf_\C^\perp)$. Since $C$ is the local direct product $Z(C)\times [C,C]$, we can write $g(\infty) = g_0g_1$ with $g_0\in Z(C)$ and $g_1\in [C,C]$ so that 
$g_0^{-1}g$ then satisfies the asymptotic conditions for being a member of $\tilde{\mathcal G}^0_{[C,C]\cap C(\sigma)}(b)^\C$. 

Such a gauge transformation $g\in \tilde{\mathcal G}^0_{[C,C]\cap
  C(\sigma)}(b)^\C$ satisfying $(\alpha,\beta) = g.(\alpha_0,\beta_0)$
is not unique: We are allowed to multiply $g$ by a term of the form
$\mathrm{Ad}(e^{-\int\alpha_0})(p)$ for some $p\in K_\C$. In order for
$g' = g\mathrm{Ad}(e^{-\int\alpha_0})(p)$ to satisfy the asymptotic
conditions of the complex gauge group, we have to impose constraints
on $p$.  Biquard showed that $p$ has to lie in the subgroup whose Lie
algebra is given by $\cf_\C\cap \cf(Y)\subset \cf_\C$, i.e. $p\in
C_\C\cap C(Y)$. Note that $H$ acts on this space with non-positive
eigenvalues, as it does so on $\cf(Y)$. However, 
$$\mathrm{Ad}(\exp(-\int\alpha_0))(p) =\mathrm{Ad}(\exp(\log(t+1)H))(p)$$
has to have a limit in $[C_\C,C_\C]$ and hence $p$ also has to lie in $[C_\C,C_\C]$ as $H$ acts trivially on $Z(C)$. Altogether, the Lie algebra of the group $L$, in which the ambiguity $p$ lives, is given by 
$$\lf = \mathrm{Lie}(L) =[\cf_\C,\cf_\C]\cap \cf(Y).$$
Thus, we have an isomorphism of complex manifolds (with respect to our chosen 
complex structure $I$)
$$\tilde{\mathcal Q}^{reg}_{C,\sigma}(b) \cong K_\C\times_{[C_\C,C_\C]\cap C(Y)}((\tf_\C)_C + Y),$$
where $(\tf_\C)_C = \{\tau_\C\in\tf_\C | \cf(\tau_C) = \cf_\C\}\subset Z(\cf_\C)$.
This isomorphism is induced by assigning to our Nahm data the pair $(g(0), \tau_{\C} + Y)$.

In particular, if we take $C_\C = T_\C$ then $Y$ must be zero and
we get $K_{\C} \times \tf_{\C}^{\rm reg}$ where
$\tf_{\C}^{\rm reg}$ denotes the locus in the complex Cartan algebra
of elements with abelian stabiliser. In the case $K=SU(n)$ it was noted 
in \cite{DKS} (Proposition 8.1.) that this space occurred naturally inside the implosion.

On the other hand, if we look at the case $C_\C = K_\C = [K_\C,K_\C]$, then $\tau$ must be zero and we get $K_\C\times_{C(Y)}Y$, which is the nilpotent orbit in $\kf_\C$ defined by $Y$. There is no non-trivial residual torus action in this case, and the "moment map" is just zero. 

\section{The Case $K=\mathrm{SU}(2)$}
We illustrate our construction in the case $K=\SU(2)$.
We then have two strata $\mathcal Q_T$ and $\mathcal Q_K$ labelled by the two possible centralisers $C=T$ and $C=K$. The second one is given by Nahm data asymptotic to $\tau=0$ and consists of two refined strata which correspond to taking either  $\sigma = 0$ or $\sigma$ to be a standard triple, which we may assume to be such that
$$H = i\sigma_1 = \left(\begin{array}{cc}-1&0\\0&1 \end{array}\right),  \quad Y = \frac{1}{2}(\sigma_2+i\sigma_3) =  \left(\begin{array}{cc}0&1\\ 0& 0\end{array}\right).$$

The triple $\tau=(0,0,0)$ is automatically Biquard-regular and the discussion of $\S 7$ above shows that as complex manifolds (with respect to the complex structure $I$) the two refined strata are given by 
$$\tilde Q_{K,0}(b) = \tilde Q_{K,0}^{reg}(b) = K_\C \times_{K_\C} \{0\} = \{0\}$$
and 
$$\tilde Q_{K,\sigma_0}(b) = \tilde Q_{K,\sigma_0}^{reg}(b) = K_\C \times_{C(Y)} {Y} \cong K_\C/N,$$
where $N=C(Y)$ is the stabiliser of $Y$. With the choice of $\sigma$ (and hence of $H,X,Y$) as above, we have 
$$N = \left\{ \left(\begin{array}{cc}1 & t\\ 0 & 1\end{array}\right)\ | \ t\in\C\right\}.$$
In other words, $\tilde Q_{K,\sigma}(b)$ is just the regular nilpotent orbit in $\sln(2,\C)$.
The torus action is trivial on the bottom stratum as $K = [K, K]$. So the bottom stratum $\mathcal Q_K$ gives the nilpotent variety in the $\mathrm{SU}(2)$-case. The identification goes via $(g,Y)\mapsto gYg^{-1}$, or, on the level of solutions to the complex equation (\ref{cplexeqn}),  $(\alpha,\beta)\mapsto \beta(0)$, which is the complex moment map for the $K_\C$-action given by complex gauge transformations with arbitrary values at $t=0$. Note that $\beta$ is a trace-free $2\times 2$-matrix of rank $1$. 

For the top stratum we can describe the open subset represented by solutions to the Nahm equations that are asymptotic to a Biquard-regular triple (i.e. a triple such that $(\tau_2,\tau_3)$ are not both zero):
$$\tilde{\mathcal Q}_{T}^{reg}(b) = K_\C \times \tf_\C^{reg} \cong \SL(2,\C)\times (\C\setminus\{0\}),$$
where $\tf_\C^{reg} = \C\setminus\{0\}$ denotes the set of regular elements in $\tf_\C$. 

The complex moment map for the $T=\Un(1)$-action is just projection onto the $(\C\setminus\{0\}$-component in the above description. Taking the quotient by $T_\C$ at level $\tau_\C$ is just given by $(g,\tau_\C)\mapsto g\tau_\C g^{-1}$, which on the level of solutions to the complex equation asymptotic to $\tau_\C$ is again the complex moment map for the $K_\C$-action $(\alpha,\beta)\mapsto \beta(0)$.

In contrast to the above discussion, the quiver construction for the universal implosion for $K=\mathrm{SU}(2)$ gives just flat $\HH^2$, which should be thought of as the space of quivers of the form
\begin{equation*}
  \C\stackrel[\mathfrak b]{\mathfrak a}{\rightleftarrows}\C^2. 
\end{equation*}
The maximal torus $T=\Un(1)\subset \SU(2)$ acts in the standard way by $(\mathfrak{a},\mathfrak{b}) \mapsto (e^{i\theta}\mathfrak{a},e^{-i\theta}\mathfrak{b})$.
Consider the $T_\C$-invariant map
$$X: \HH^2 \to \mathfrak{gl}(2,\C) \quad (\mathfrak{a},\mathfrak{b})\mapsto \mathfrak{a}\mathfrak{b}.$$
This is the complex moment map for the $\GL(2,\C)$-action. The complex moment map for the $T$-action is given by $\mu_T(\mathfrak a,\mathfrak b) = \mathrm{tr}(X(\mathfrak a,\mathfrak b))$. 

In $\S 8$ of \cite{DKS} the identification of  $\HH^2\cong \C^4$ with the GIT quotient $(\SL(2,\C)\times \mathfrak b)\sslash N$ is explained, where 
$$\mathfrak b = \left\{ \left(\begin{array}{cc}a & b\\ 0 & -a\end{array}\right)\ | \ a,b\in\C\right\}.$$
The subset corresponding to $a\neq 0$ then gives $\SL(2,\C)\times \tf_\C^{reg}$ and the subset defined by $a=0, b\neq 0$ corresponds to the regular nilpotent orbit. Note however, that there is in this case a non-trivial torus action, leaving $a$ fixed (which is the value of the complex moment map for the $T$-action)  and taking $b$ to $t^2b$. As we have fixed the $\su(2)$-triple $\sigma$, this action is not present on our stratum $\mathcal Q_{C,\sigma}$.  

\begin{remark} \label{comparison}
As we have just noted, the candidate $\mathcal Q$ for the universal hyperk\"ahler implosion which we have constructed using the Nahm equations does not coincide with the quiver version constructed in \cite{DKS} when $K=\SU(n)$. We can modify the construction slightly in order to obtain a candidate which, at least when $K=\SU(2)$, can be identified with the quiver version from \cite{DKS}. For this recall that 
 $\tilde{\mathcal A}_{C,\sigma}$ 
\index{$\tilde{\mathcal A}_{C,\sigma}$}  
was defined 
to be the space of quadruples of $C^1$ functions 
$$(T_0,T_1,T_2,T_3):[0,\infty) \to \mathfrak k\otimes\mathbb R^4$$ 
such that there exists $\tau_0\in Z(\cf)$ and a commuting triple $\tau = (\tau_1, \tau_2, \tau_3)$ in $ \mathfrak t$, such that $C(\tau) = C$,  satisfying
\begin{itemize}
\item $T_0^D-\tau_0 \in \Omega_\zeta(\cf)$
\item $T_i^D - \tau_i -\frac{\sigma_i}{2(t+1)} \in \Omega_\zeta(\cf) \quad i=1,2,3, $
\item $T_i^H\in\Omega_{\exp}(\cf^\perp) \quad i=0,1,2,3.$
\end{itemize}
where the $\SU(2)$ triple $\sigma$ is fixed. If instead we allow $\sigma$ to vary in a single $T$-orbit, essentially the same analysis goes through and we obtain a modifed version $\hat{\mathcal Q}$ of $\mathcal Q$ with a torus fibration
$\hat{\mathcal Q} \to \mathcal Q$ and similar properties (cf. Remark \ref{comparison2}).

\end{remark}

\section{Symplectic Implosion and the `Baby Nahm Equation'}

In this section we develop a gauge-theoretic description of symplectic implosion inspired by Bielawski's discussion of the so-called `Baby Nahm Equation' in \cite{Bielawski:2007}. Let $K$ be a compact Lie group as before with Lie algebra $\kf$ and fix a invariant inner product on $\kf$. Let $\mathcal I \subset \R$ be an interval.  For a pair of $C^1$ functions $(T_0,T_1): \mathcal I\to \kf$ the \emph{Baby Nahm equation} is given by 
\begin{equation}\label{BabyNahm}
\dot T_1 + [T_0,T_1] = 0.
\end{equation}

\subsection{$(T^*K)_{\rm impl}$ and the Baby Nahm equation on $[0,1]$}
Let $\mathcal I=[0,1]$ and consider the space $\mathcal A = (C^1([0,1],\kf))^2$. This is an infinite-dimensional flat K\"ahler manifold equipped with the $L^2$ metric and compatible complex structure which can be seen by writing a tangent vector $X\in T_T\mathcal A$ as $X_0+iX_1$ and taking $I$ to be multiplication by $i$. The gauge group $\mathcal G_{00} = \{u\in C^2([0,1],K) | u(0)=1=u(1)\}$ (see also \ref{G00}) acts on $\mathcal A$ as usual preserving the K\"ahler structure and the associated moment map is given by the baby Nahm equation, as is shown by the following calculation for $\xi \in \mathrm{Lie}(\mathcal G_{00})$ and $Y\in T_T\mathcal A$: 
\begin{eqnarray*}
\omega(X^\xi, Y)  &=& \int_0^1 \langle (I(-\dot\xi+[\xi,T_0], [\xi,T_1]), (Y_0,Y_1)\rangle \\
&=& \int_0^1 \langle \dot\xi-[\xi,T_0], Y_1\rangle + \langle [\xi,T_1], Y_0\rangle\\
&=&\langle \xi,Y_1\rangle|_0^1  - \int_0^1 \langle\xi, \dot Y_1 + [T_0,Y_1]+[Y_0,T_1]\rangle\\
&=& - \int_0^1 \langle\xi, \dot Y_1 + [T_0,Y_1]+[Y_0,T_1]\rangle
\end{eqnarray*}
as $\xi$ vanishes at the endpoints of the interval.

Thus, we can view the moduli space $\mathcal M$  of solutions to the baby Nahm equation modulo $\mathcal G_{00}$ as a K\"ahler quotient. Let $G\subset K$ be a subgroup. On $\mathcal M$ we have an action of $G$ induced from the gauge group $\mathcal G_{0,G}$ of gauge transformations equal to $1$ at $t=0$ and having an arbitrary value in $G$ at $t=1$. From the above calculation we can see 
that if $G=K$ then the moment map is given by evaluation at $t=1$:
 $$\mu_K(T) = T_1(1).$$
 Thus the symplectic implosion of $\mathcal M$ with respect to this $K$-action is given by a quotient
 $$\mathcal M_{\rm impl} = \mu_K^{-1}(\tf^+)/\cong,$$
 where $\tf^+$ is a (closed) positive Weyl chamber in $\tf$, the Lie algebra of the maximal torus. The implosion is stratified by the pre-images $\tf^+_C$
of the faces of $\tf^+$ under $\mu_K$ collapsed according to $\cong$
by the commutator of the associated centraliser subgroup. That is, we have strata
 $$\mathcal M_C = \mu_K^{-1}(\tf^+_C)/[C,C] = \{T\in \mathcal A \; | \; \dot T_1+[T_0,T_1] = 0, T_1(1)\in\tf^+_C\}/\mathcal G_{0,[C,C]}$$
 and 
 $$ \mathcal M_{\rm impl} = \coprod_C \mathcal M_C.$$
To identify the moduli spaces $\mathcal M$, respectively $\mathcal M_C$, we proceed as follows, following
Bielawski \cite{Bielawski:2007}. Given a solution $T$ to the Nahm equation, there exists a unique gauge transformation $u_0\in\mathcal G_{0,K}$ such that 
$$u_0. T_0 := u_0T_0u_0^{-1}-\dot u_0u_0^{-1} = 0.$$
As a consequence of the equation, we see that this forces $u_0.T_1 = u_0T_1u_0^{-1}$ to be constant. We can therefore define a map 
$$\Phi: \mathcal M\to K\times \kf, \qquad \Phi(T_0,T_1) = (u_0(1), T_1(0)).$$
We claim that this is a well-defined bijection. If $u\in\mathcal G_{00}$, then  $u_0u^{-1}$ gives the unique element in $\mathcal G_{0,K}$ that gauges $u.T_0$ to zero. Since $u(0)=u(1)=1$ we see that $\Phi(u.(T_0,T_1)) = \Phi(T_0,T_1)$, i.e. $\Phi$ is $\mathcal G_{00}$-invariant. 

To show injectivity, suppose $\Phi(T_0,T_1) =
\Phi(\tilde{T}_0,\tilde{T}_1)$. Then we have $T_1(1) = \tilde{T}_1(1)$. Let
$u_0$ and $\tilde{u}_0$ be the uniquely determined gauge transformations in
$\mathcal G_{0,K}$ that gauge $T_0$ and $\tilde{T}_0$ respectively to zero.
By assumption $u_0(1) = \tilde{u}_0(1)$. Therefore the gauge 
transformation $u = u_0^{-1} \tilde{u}_0$ lies
in $\mathcal G_{00}$. Also $u. \tilde{T}_{0} = u_0^{-1}.0  = T_0$.

This implies that $T_1 = u\tilde{T}_1u^{-1}$, since both $(T_0,T_1)$ and
$(T_0,u\tilde{T}_1u^{-1}) = u.(\tilde{T}_0,\tilde{T}_1)$ 
solve the baby Nahm equation and $T_1(1) =
u(1)\tilde{T}_1(1)u(1)^{-1} = \tilde{T}_1(1)$. Thus, by uniqueness of solutions of
linear ODEs $T_1\equiv u\tilde{T}_1u^{-1}$ on $[0,1]$, and hence $(T_0, T_1)$
is equivalent via $\mathcal G_{00}$ to $(\tilde{T}_0, \tilde{T}_1)$

The map $\Phi$ is also surjective, and an inverse is given as follows. For $(k,\xi)\in K\times\kf$ define a solution to the Baby Nahm equation by choosing a path $u_0\in \mathcal G_{0,K}$ such that $u_0(1) = k$ and define 
$$(T_0,T_1) = u_0^{-1}.(0,\xi).$$
Observe that any two choices of $u_0$ differ by an element of $\mathcal G_{00}$, so that this construction defines a unique element in $\mathcal M$. 

$\mathcal M$ is thus identified with $K \times \kf$, and hence with $T^*K$
on choosing an invariant inner product on $\kf$.

On $K\times \tf^+_C$ we have an action of $[C,C]$ via $c.(k,\xi) = (kc^{-1}, \xi)$ and a similar argument to the one just given shows that $\Phi$ gives a $[C,C]$ equivariant surjection
$$\Phi: \mu_K^{-1}(\tf^+_C) \to K\times \tf^+_C,$$
which, by the above discussion of injectivity, induces an isomorphism
$$\Phi_C: \mathcal M_C \to (K\times \tf^+_C)/[C,C].$$
Here the right hand side coincides with the stratum of $(T^*K)_{\rm impl}$ associated with $C$ as defined in \cite{GJS:2002}. 

\medskip
Hence we have a gauge-theoretic description of the universal symplectic implosion.

\subsubsection{The Symplectic Structure}
We start by calculating the symplectic structure of $\mathcal M$. Recall that we defined an inverse $\Psi$ to $\Phi:\mathcal M \to K\times \kf$ by $\Psi(k,\xi) = u_0^{-1}.(0,\xi)$, where $u_0$ was any $C^2$-path on $[0,1]$ linking $1$ and $k$. We calculate the derivative of $\Psi$ at $(1,0)$. This is given by 
\begin{eqnarray*}
D\Psi_{(1,0)}(\psi_1,\psi_2) &=& \frac{d}{d\theta}|_{\theta = 0}\Psi(\exp(\theta\psi_1), \theta\psi_2) \\
&=& \frac{d}{d\theta}|_{\theta = 0}\left(-\left(\frac{d}{dt}\exp(-t\theta\psi_1)\right)\exp(t\theta\psi_1), \mathrm{Ad}(\exp(-t\theta\psi_1))(\theta\psi_2)\right)\\
&=& (\psi_1,\psi_2)
\end{eqnarray*}
Therefore the pull-back of the symplectic form on $\mathcal M$ under $\Psi$ is given by 
\begin{eqnarray*}
(\Psi^*\omega_I)_{(1,0)}((\psi_1,\psi_2),(\xi_1,\xi_2)) &=& \int_0^1\langle ID\Psi_{(1,0)}(\psi_1,\psi_2), D\Psi_{(1,0)}(\xi_1,\xi_2) \rangle\\
&=& \int_0^1 \langle\psi_2,\xi_1\rangle - \langle \psi_1,\xi_2\rangle \\
&=& \langle\psi_2,\xi_1\rangle - \langle \psi_1,\xi_2\rangle,
\end{eqnarray*}
which is the standard symplectic form of $K\times \kf \cong T^*K$.
The symplectic form on the stratum $\mathcal M_C$ is induced from this and so coincides with the symplectic structure on the stratum as defined in \cite{GJS:2002}.

\subsubsection{The Complex Structure}
Given an element $(T_0,T_1)\in \mathcal A$ define $\alpha = T_0-iT_1$. Then the Baby Nahm equation takes the form 
$$\dot \alpha +\dot\alpha^* +[\alpha,\alpha^*] = 0.$$
Define the complex gauge group $\mathcal G_{00}^\C = \{v: [0,1]\to K_\C| v(0) = 1 = v(1)\}$. It acts on $\mathcal A$ by 
$$v.\alpha = v\alpha v^{-1} - \dot vv^{-1}.$$
We shall show that in fact $\mathcal M = \mathcal A/\mathcal G_{00}^\C$. First of all, in analogy to the construction above, we have 
$$\mathcal A/\mathcal G_{00}^\C \cong K_\C.$$
The isomorphism works again by sending any $\alpha\in\mathcal A$ to $v_0(1)$, where $v_0 \in \mathcal G_{0,K}^\C$ is the unique gauge transformation gauging $\alpha$ to zero such that $v_0(0) = 1$. Its inverse sends $g\in K_\C$ to $\alpha_g = -\dot v_gv_g^{-1}$, where $v_g$ is any element of $\mathcal G_{0,K}^\C$ such that $v_g(1) = g^{-1}$. We shall now associate to any given $g\in K_\C$ a path $v_g$ such that the resulting $\alpha_g$ will satisfy the Baby Nahm equation. 

This can be done as follows. Use polar decomposition and write $g^{-1} = k\exp(i\xi)$ for some $(k,\xi)\in K\times\kf$. Then observe that if we take $v_\xi(t) = \exp(it\xi)$, then this clearly links $1$ and $\exp(i\xi)$ and satisfies moreover 
$$-\dot v_\xi v_\xi^{-1} = i\xi,$$
which clearly solves the real equation as $i \xi$
is constant and  self-adjoint. Now take any $K$-valued path $u$ linking $1$ and $k$ and define $v_g(t) = u(t)\exp(it\xi)$. Then $\alpha_g = u.(i\xi)$ solves the Baby Nahm equation, as the Baby Nahm equation is invariant under arbitrary $K$-valued gauge transformations.

This argument shows that any $\mathcal G_{00}^\C$-orbit on $\mathcal A$ contains a solution of the Baby Nahm equation. If now $\alpha$ and $\alpha' = v.\alpha$ both solve the Baby Nahm equation, then the convexity argument in section 2 of \cite{Donaldson:1984} implies that $\alpha$ and $\alpha'$ must be equivalent by a $K$-valued gauge transformation. Thus, we see that, as a complex manifold, $\mathcal M = \mathcal A/\mathcal G_{00}^\mathbb C\cong K_\C$ and the isomorphism between $T^*K$ and $K_\C$ is given via polar decomposition.

To identify the complex structure of $\mathcal  M_C$, we take $\alpha$ to be a solution to the baby Nahm equation. We let $A_C = \exp(iZ(\mathfrak c))\subset T_\C$. Then we can define a map 
$$\Phi_C: \mathcal G_{0,[C,C]}^\C.\mu_K^{-1}(\tf^+_C)/\mathcal G_{0,[C,C]}^\C \to (K\times A_C)/[C,C]$$
by $$\Phi(\alpha) = \exp(i\tau_1).u_0(1).$$
Here $u_0$ denotes the unique compact gauge transformation that gauges away $T_0 = 1/2(\alpha-\alpha^*)$ such that $u_0(0) = 1$ and $i\tau_1 = 1/2(\alpha + \alpha^*)(0)$. In other words, we found a complex gauge transformation that gauges away $\alpha$ in stages, by first gauging away $T_0$ and then gauging away the constant remainder $i\tau_1$. The map $\Phi$ is $[C,C]$-equivariant and hence gives rise to a well-defined bijection on quotient spaces.

Moreover, we know from the discussion in $\S 6$ of \cite{GJS:2002} that the right-hand side $(K\times A_C)/[C,C]\cong K_\C/[P_C,P_C]$, where $P_C$ is the parabolic associated to $C$.

\subsection{$(T^*K)_{\rm impl}$ and the Baby Nahm Equation on the half-line}
In this section we shall describe the strata of the universal symplectic implosion $(T^*K)_{\rm impl}$ as moduli spaces of solution to the Baby Nahm equation on the half-line $[0,\infty)$. For those subgroups $C$ of $K$ which arise as centraliser of an element of the positive Weyl chamber,  we consider the gauge group $\tilde{\mathcal G}_{[C,C]}(b)$ defined as in $\S 4$ but with $\Omega_\zeta$ replaced by $\Omega_{\exp}$. It acts on the flat K\"ahler manifold $\mathcal A_{C,b}$, which consists of pairs of $\kf$-valued $C^1$-functions $(T_0,T_1)$ on $[0,\infty)$, which have limits $\tau_1\in(\tf_+)_C$, (i.e. in the face of the positive Weyl chamber such that $C(\tau_1) = C$) and $\tau_0 \in Z(\cf)$ such that $T_i-\tau_i\in\Omega_{\exp}$ for $i=0,1$. The complex structure is defined as in the previous section and the metric is given by the Bielawski bilinear form. The moment map calculations (see proposition \ref{propmomentmaps}) then go through without essential modification and the vanishing condition of the associated moment map gives the Baby Nahm equation
$$\dot T_1+[T_0,T_1] = 0.$$
We have explained before that the resulting moduli space is the same as the space of solutions with $\tau_0 =0$ modulo gauge transformations $u\in \tilde{\mathcal G}_{[C,C]}(b)$ that satisfy $s(u) = \lim_{t\to\infty} \dot uu^{-1} = 0$ analogous to the gauge group defined in $\S 3$. Note that such gauge transformations have a limit $u(\infty) = \tilde u(\infty) \in [C,C]$. To identify the K\"ahler quotient $\mathcal M = \mathcal A_{C,b} \symp \tilde{\mathcal G}_{[C,C]}(b)$, we proceed as in the previous section.

Let $(T_0,T_1)$ be a solution with $\tau_0=0$. Denote by $u$ the unique gauge transformation with $u(0)=1$ such that $T_0 = -\dot uu^{-1}$. Note that since $K$ is compact, $u$ is automatically bounded. Thus, we see that for $t'>t$
$$|u(t')-u(t)| \leq \int_t^{t'}|\dot u| \leq c\int_t^{t'} |T_0| < M e^{-\eta t}$$
for some constant $M>0$. Thus, $u(\infty)\in K$ exists and it follows furthermore that $u_0(t) = u(t)u(\infty)^{-1}$ satisfies the asymptotic conditions of the the gauge group $\mathcal G_{[C,C]}(b)$ (with $s(u_0) = 0$), but not $u_0(0) =1$. We have thus written $(T_0,T_1) = u_0.(0,T_1(\infty))$. Any other such gauge transformation differs from $u_0$ by right multiplication by a constant $c\in C$ and in order to satisfy the asymptotics of the gauge group, we must require $c\in [C,C]$. 

It therefore makes sense to consider the following map:
$$\Phi: \mathcal M_C \to (K\times (\tf_+)_C)/[C,C]\quad \Phi(T_0,T_1) = (u_0(0), T_1(\infty)) = (u_0(0), \tau_1),$$
where $[C,C]$ acts on $K\times (\tf_+)_C$ by $c.(k,\xi) = (kc,\xi)$ and $(\tf_+)_C\subset Z(\cf)$ is the open set of elements in the positive Weyl chamber with centraliser equal to $C$.
Similar arguments as the ones we used before when we studied the baby Nahm equations on $[0,1]$ show that $\Phi$ is a well-defined bijection of sets.

\subsubsection{Complex Structure}

We now want to identify the complex structure of $\mathcal M_C$ at a point of the form $(0,\tau_1)$ (any $K$-orbit on $\mathcal M_C$ contains a point of this form). The tangent space of $\mathcal M_C$ at $(0,\tau_1)$ is given by the solution space of the ODE-system
$$\dot X_0 + [\tau_1,X_1] = 0 \qquad \dot X_1 + [X_0,\tau_1] = 0.$$
We see immediately that this forces $X_i^D$ to be constant, thus equal to an element of $Z(\cf)$, as we need $X_i(\infty)=\delta_i$ to lie in the centre of $\cf$. It is useful to decompose $\cf^\perp\subset \kf$ into two-dimensional root spaces $\cf^\perp = \oplus_\alpha \kf_\alpha$ with respect to the adjoint action of $Z(C)$. On $\kf_\alpha$ we have the relation $[\tau_1,\xi_\alpha] = 2\pi\alpha(\tau) I_0\xi_\alpha$, where $I_0$ denotes the natural complex structure on $\cf^\perp\cong T_{\tau_1}(K.\tau_1)$, the tangent space to the (co)adjoint orbit of $\tau_1$ (see $\S 8.B$ in \cite{Besse:1987}).

Note that since any solution of the Baby Nahm equations in $\mathcal A_C$ is of the form $u.(0,\tau)$ for some $\tau\in\tf^+_C$, where $u$ satisfies the conditions of the gauge group at infinity and has an arbitrary value at $t=0$, we can assume that tangent vectors $(X_0,X_1)$ at $(0,\tau_1)$ are given by combinations of infinitesimal gauge transformations and variations of the limit $\tau_1$, i.e. are of the form 
$$X_0 = -\dot\xi(t)\qquad X_1 = [\xi,\tau_1] + \delta_1,$$
for some $\delta_1\in Z(\cf)$ and some $\xi: [0,\infty)\to\kf$ satisfying all the conditions for being a member of the Lie algebra of the gauge group except that we allow arbitrary $\xi(0)$. These asymptotics are necessary to ensure that the $X_i-\delta_i$ have the right decay behaviour. We write $X^{\xi,\delta_1}$ for such a tangent vector. The differential equations satisfied by $(X_0,X_1)$ hence translate into a second order ODE on $\xi$, namely
\begin{equation}\label{tangentspace}
\ddot \xi + [\tau_1,[\tau_1,\xi]] = 0,
\end{equation}
since $[\tau_1,\delta_1] = 0$. The ODE on $X_1$ is automatically satisfied, as this is just the linearised Baby Nahm equation and the Baby Nahm equation is invariant under arbitrary gauge transformations. From the above discussion we get that $\dot\xi^D = \dot\xi(\infty)\in Z(\cf)$, which implies $\xi^D = (t-b)\dot\xi(\infty) + \tilde\xi$ for some constant $\tilde\xi\in [\cf,\cf]$ as we want to satisfy the condition of the gauge group.  

To determine the other components, decompose $\xi = \sum_\alpha \xi_\alpha$ according to the above root space decomposition. Then we have on each space $\kf_\alpha$
$$\ddot \xi_\alpha - 4\pi^2\alpha(\tau_1)^2\xi_\alpha = 0.$$
The minus sign comes from the relation $[\tau_1,\xi_\alpha] = 2\pi\alpha(\tau) I_0\xi_\alpha$ on $\kf_\alpha$. The solution to this equation which decays exponentially is given by 
$$\xi_\alpha(t) = \xi_\alpha(0)\exp(-2\pi\alpha(\tau_1)t),$$
where we choose from $\pm\alpha$ the one such that $\alpha(\tau_1)>0$ (see the discussion in \cite{Besse:1987}, $\S 8.B$). 
Now $D\Phi$ sends such a tanget vector $X^{\xi,\delta_1}$ to $(\xi(0), \delta_1)\in T_{1,\tau_1}(K/[C,C]\times (\tf_+)_C) \cong (\kf/[\cf,\cf]) \oplus Z(\cf) \cong (Z(\cf)\oplus \cf^\perp) \oplus Z(\cf)$. This in fact equals 
$$D\Phi(X^{\xi,\delta_1}) = (-b\dot\xi(\infty) + \sum_{\alpha;\alpha(\tau_1)>0} \xi_\alpha(0), \delta_1) \in (Z(\cf)\oplus \cf^\perp) \oplus Z(\cf)$$
If we apply the complex structure, we get by a direct calculation 
$$D\Phi(IX^{\xi,\delta_1}) = (-b\delta_1 + \sum_{\alpha;\alpha(\tau_1)>0} I_0\xi_\alpha(0), -\dot\xi(\infty)),$$
i.e. we arrive at the following proposition.

\begin{proposition}
The complex structure $I$ on $(K\times (\tf_+)_C)/[C,C]$ induced by the complex structure $I$ on $\mathcal M_C$ is given at a point $(1,\tau_1)$ by 
$$I(v+v^\perp, w) = (-bw+ I_0v^\perp, \frac{1}{b}v),$$
where $v,w\in Z(\cf), v^\perp\in \cf^\perp$, so that $v+v^\perp \in T_1(K/[C,C])$ and $w\in T_{\tau_1}(\tf_+)_C$ and $I_0$ denotes the canonical complex structure on the orbit through $\tau_1$ (see $\S 8.B$ in \cite{Besse:1987}).
\end{proposition}

\subsubsection{The metric on $\mathcal M_C$}
From the above discussion we know that by decomposing $\kf$ into the root spaces of $\tau_1$, we can write the general solution to the equation (\ref{tangentspace}) explicitly as 
$$\xi(t) = (t-b)\dot\xi(\infty) + \sum_{\alpha;\alpha(\tau_1)>0} \xi_\alpha(t),$$
where $\xi_\alpha(t) = \xi_\alpha(0)\exp(-2\pi\alpha(\tau_1)t)$. And this corresponds to the tangent vector $X^{\xi,\delta_1}= (X_0,X_1) = (-\dot\xi, [\xi,\tau_1]+\delta_1)$ where $\delta_1\in Z(\cf)$.  We can then calculate $\Vert X^{\xi,\delta_1}\Vert^2_b$ using the relation $[\tau_1,\delta_1]=0$: 
\begin{eqnarray*}
\Vert X^{\xi,\delta_1}\Vert^2_b &=& b(|\dot\xi(\infty)|^2 + |\delta_1|^2) + \int_0^\infty |\dot\xi|^2 + |[\tau_1,\xi] + \delta_1|^2 - (|\dot\xi(\infty)|^2 + |\delta_1|^2) dt\\
&=& b(|\dot\xi(\infty)|^2 + |\delta_1|^2) + \int_0^\infty (\frac{d}{dt}\langle\dot\xi,\xi\rangle) - \langle\ddot\xi,\xi\rangle+ |[\tau_1,\xi]|^2 + 2\langle[\tau_1,\xi],\delta_1\rangle - |\dot\xi(\infty)|^2 dt\\
&=& b(|\dot\xi(\infty)|^2 + |\delta_1|^2) + \int_0^\infty (\frac{d}{dt}\langle\dot\xi,\xi\rangle) + \langle[\tau_1,[\tau_1,\xi]],\xi\rangle+ |[\tau_1,\xi]|^2 - |\dot\xi(\infty)|^2 dt\\
&=& b(|\dot\xi(\infty)|^2 + |\delta_1|^2) + \int_0^\infty (\frac{d}{dt}\langle\dot\xi,\xi\rangle) - |\dot\xi(\infty)|^2 dt\\
&=& b(|\dot\xi(\infty)|^2 + |\delta_1|^2) + [\langle\dot\xi,\xi\rangle - t|\dot\xi(\infty)|^2]_0^\infty.
\end{eqnarray*}
We evaluate the expression in brackets by writing $\xi = (t-b)\dot\xi(\infty) + \tilde\xi$, so that $\dot\xi = \dot\xi(\infty) + \dot{\tilde\xi}$. Thus, (see the calculation in the proof of Proposition \ref{propmomentmaps})
\begin{eqnarray*}
[\langle\dot\xi,\xi\rangle - t|\dot\xi(\infty)|^2]_0^\infty &=& -b|\dot\xi(\infty)|^2 - \langle\dot\xi(0),\xi(0)\rangle.
\end{eqnarray*}
We arrive at
$$\Vert X^{\xi,\delta_1}\Vert^2_{B,b} = b|\delta_1|^2 - \langle\dot\xi(0),\xi(0)\rangle.$$
Using the explicit form of $\xi(t)$, we can rewrite this as 
\begin{eqnarray*}
\Vert X^{\xi,\delta_1}\Vert^2_{B,b} &=& b|\delta_1|^2 - \langle\dot\xi(0),\xi(0)\rangle\\
&=& b|\delta_1|^2 - \langle \dot\xi(\infty) + \sum_{\alpha;\alpha(\tau_1)>0}\dot\xi_\alpha(0), -b\dot\xi(\infty) + \sum_{\alpha; \alpha(\tau_1)>0}\xi_\alpha(0)\rangle\\
&=& b(|\delta_1|^2 + |\dot\xi(\infty)|^2) - \sum_{\alpha;\alpha(\tau_1)>0} \langle \dot\xi_\alpha(0), \xi_\alpha(0)\rangle\\
&=& b(|\delta_1|^2 + |\dot\xi(\infty)|^2) + \sum_{\alpha;\alpha(\tau_1)>0} 2\pi\alpha(\tau_1)|\xi_\alpha(0)|^2,
\end{eqnarray*}
which is hence \emph{positive definite}, since $\alpha(\tau_1)>0$. In the calculation of the tangent space at a model solution we have used the fact that any solution to the Baby Nahm equation is gauge equivalent to a model solution by the $K$-action given by gauge transformations with arbitrary values at $t=0$. This $K$-action preserves the Bielawski metric on the moduli space. Hence we see that the metric must be positive definite \emph{everywhere}.

\begin{proposition}
The Bielawski metric $\Vert\cdot\Vert_{B,b}$ is positive definite on the moduli space $\mathcal M_C$ whenever $b>0$. At a model solution it is given by the formula
$$\Vert X^{\xi,\delta_1}\Vert^2_{B,b} = b(|\delta_1|^2 + |\dot\xi(\infty)|^2) + \sum_{\alpha;\alpha(\tau_1)>0} 2\pi\alpha(\tau_1)|\xi_\alpha(0)|^2.$$
Using the formula for $D\Phi$ given in the previous section, we can write this on $T_{(1,\tau_1)}(K\times (\tf_+)_C)/[C,C]$ as 
$$|(v+v^\perp,w)|^2 = b|w|^2 + b^{-1}|v|^2 + \sum_{\alpha;\alpha(\tau_1)>0} 2\pi\alpha(\tau_1)|v_\alpha^\perp|^2,$$
where $v,w\in Z(\cf)$ and $v^\perp\in \cf^\perp$. 
\end{proposition}

\subsubsection{Symplectic Structure}
We can now read off the symplectic form on $T_{(1,\tau_1)}(K\times (\tf_+)_C)/[C,C]$ from the above two propositions. 
\begin{eqnarray*}
\omega_b((v_1+v_1^\perp, w_1),(v_2+v_2^\perp, w_2)) &=& \langle I(v_1+v_1^\perp, w_1),(v_2+v_2^\perp, w_2)\rangle_b\\
&=& \langle (-bw_1+I_0v_1^\perp, b^{-1}v_1),(v_2+v_2^\perp, w_2)\rangle_b\\
&=& b\langle b^{-1}v_1,w_2\rangle + b^{-1}\langle-bw_1,v_2\rangle + \sum_{\alpha;\alpha(\tau_1)>0}2\pi\alpha(\tau_1)\langle I_0(v_1^\perp)_\alpha,(v_2^\perp)_\alpha\rangle\\
&=& \langle v_1,w_2\rangle - \langle w_1,v_2\rangle + \sum_{\alpha;\alpha(\tau_1)>0} \langle [\tau_1,(v_1^\perp)_\alpha],(v_2^\perp)_\alpha\rangle\\
&=&  \langle v_1,w_2\rangle - \langle w_1,v_2\rangle + \langle [\tau_1,v_1^\perp],v_2^\perp\rangle\\
&=&   \langle v_1,w_2\rangle - \langle w_1,v_2\rangle + \langle \tau_1,[v_1^\perp,v_2^\perp]\rangle,
\end{eqnarray*}
where we used the defining properties of the canonical complex structure $I_0$ as given in \cite{Besse:1987}, $\S 8.B$. We state the result of this calculation in the following proposition
\begin{proposition}
The induced symplectic form on $(K\times (\tf_+)_C)/[C,C]$ at the point $(1,\tau)$ is given by 
$$\omega_b((v_1+v_1^\perp, w_1),(v_2+v_2^\perp, w_2)) = \langle v_1,w_2\rangle - \langle w_1,v_2\rangle + \omega^{KKS}_{\tau_1}(v_1^\perp,v_2^\perp),$$
where $\omega^{KKS}$ denotes the standard symplectic form on the coadjoint orbit $\mathcal O_{\tau_1}$. In particular, $\omega$ is independent of the parameter $b$ in the definition of the Bielawski metric.
\end{proposition}

\printindex

\bibliographystyle{plain}

\end{document}